\documentclass{preprint}

\usepackage[marginparwidth=2.4cm, marginparsep=1mm]{geometry}
\usepackage[full]{textcomp}
\usepackage[osf]{newtxtext}
\usepackage[title,titletoc]{appendix}

\usepackage{ifthen} 
\usepackage{hyperref}
\usepackage{breakurl}
\usepackage{amsmath}
\usepackage{amssymb}
\usepackage{mhequ}
\usepackage{mhenvs}
\usepackage{mhsymb}
\usepackage{mathrsfs}
\usepackage{microtype}
\usepackage{wasysym}
\usepackage{centernot}
\usepackage{booktabs}
\usepackage{tikz}
\usetikzlibrary{snakes}
\usepackage{colortbl}
\usepackage{scalerel}

\usepackage{hyperref}
\usepackage{amsmath}
\usepackage{amssymb}
\usepackage{mathrsfs}
\usepackage{mathtools}
\usepackage{bbm}
\usepackage{enumitem}
\usepackage{graphicx} 
\usepackage{marginnote}
\usepackage{tikz}
\usepackage{framed}
\usepackage{comment}
\usepackage{esint}
\usepackage{xcolor}
\usepackage{stmaryrd}
\usepackage{subfig}
\usepackage[normalem]{ulem}

\allowdisplaybreaks

\definecolor{darkblue}{rgb}{0.13,0.13,0.39}

\newcommand{\epsEH}{\mathfrak{e}}

\colorlet{symbols}{black}

\newtheorem{assumption}{Assumption}
\newtheorem{example}{Example}

\newcommand{\set}[1]{\llbracket #1 \rrbracket}

\def\1{\mathbbm{{1}}}
\def\s{\mathfrak{s}}

\def\CCE{\mathbb{E}}
\def\CCV{\mathbb{V}}
\def\CCG{\mathbb{G}}
\def\CCT{\mathbb{T}}
\newcommand\STAR{\scalebox{0.3}{\tikz{\node[draw,fill=black,star,star point height=.7em,minimum size=1em]{};}}}
\def\i{\mathfrak{i}}
\def\CK{\mathcal{K}}

\def\CP{\mathcal{P}}

\def\CA{\mathcal{A}}
\def\CG{\mathcal{G}}

\def\L{\mathtt{L}}
\def\var{\mathtt{var}}
\def\nil{\mathtt{nil}}
\def\CN{\mathcal{N}}
\def\bz{\mathbf{z}}
\def\bp{\mathbf{p}}

\def\kone{\mathbf{k}}

\def\fC{\mathfrak{C}}
\def\Cd{\mathtt{C}}
\def\n{\mathbf{n}}

\def\SC{\mathscr{C}}
\def\SP{\mathscr{P}}
\def\SD{\mathscr{D}}
\def\SY{\mathscr{Y}}
\def\CL{\mathcal{L}}

\newcommand{\smallgeq}{\scaleobj{0.8}{\geq}}

\renewcommand\emptyset{\varnothing}

\def\v{\nu}
\def\w{\omega}

\def\Amin{{A^{\!-}}}

\def\cadlag{c\`{a}dl\`{a}g }

\def\BDG{Burkholder-Davis-Gundy }

\def\CD{\mathcal{D}}
\def\CC{\mathcal{C}}
\def\CI{\mathcal{I}}

\def\fK{\mathfrak{K}}

\newcommand{\M}{\mathbb{M}}
\newcommand{\Mnew}{\mathbb{N}}

\newcommand{\nn}{\mathbbm{n}}
\newcommand{\T}{\mathbb{T}}

\newcommand{\Ff}{\mathscr{F}}

\newcommand{\Nm}{\ov{\M}}

\newcommand{\bM}{\mathbf{M}}

\newcommand{\SR}{\mathscr{R}}

\newcommand{\al}{\alpha}

\newcommand{\ga}{\gamma}
\newcommand{\de}{\delta	}

\renewcommand{\subset}{\subseteq}

\renewcommand{\d}{\mathrm{d}}

\newcommand{\TV}{\mathrm{TV}}
\newcommand{\ov}{\overline}

\newcommand{\F}{\mathscr{F}}

\renewcommand{\hat}{\widehat}

\renewcommand{\phi}{\varphi}
 								
\newcommand{\Le}{\Lambda_\eps}		 					
\renewcommand{\ae}{\star_\eps} 							

\def\Md#1{\bM^{#1}_{\eps, \mathrm{diag}}}

\makeatletter
\newcommand*\bigcdot{\mathpalette\bigcdot@{.5}}
\newcommand*\bigcdot@[2]{\mathbin{\vcenter{\hbox{\scalebox{#2}{$\m@th#1\bullet$}}}}}
\makeatother

\usetikzlibrary{shapes}
\usetikzlibrary{shapes.misc}
\usetikzlibrary{shapes.symbols}
\usetikzlibrary{decorations}
\usetikzlibrary{decorations.markings}
\usetikzlibrary{decorations.pathreplacing}
\usetikzlibrary{calc,intersections,through,backgrounds}

\colorlet{testcolor}{green!60!black}

\def\drawx{\draw[-,solid] (-3pt,-3pt) -- (3pt,3pt);\draw[-,solid] (-3pt,3pt) -- (3pt,-3pt);}
\tikzset{
	root/.style={circle,fill=testcolor,inner sep=0pt, minimum size=2mm},
	dot/.style={circle,fill=black,inner sep=0pt, minimum size=1mm},
	small_dot/.style={circle,fill=black,inner sep=0pt, minimum size=0.9mm},
	var/.style={circle,fill=white,draw=black,inner sep=0pt, minimum size=2mm},
	var_blue/.style={circle,fill=blue!20,draw=blue,inner sep=0pt, minimum size=2mm},
	var_very_blue/.style={circle,fill=blue,draw=blue,inner sep=0pt, minimum size=2mm},
	var_red_square/.style={regular polygon,regular polygon sides=4, draw, fill=red!20, draw=red, inner sep=0pt, minimum size=2.5mm, shape border rotate=45},
	var_red_triangle/.style={regular polygon,regular polygon sides=3, draw, fill=red!20, draw=red, inner sep=0pt, minimum size=3mm, shape border rotate=180},
	var_red/.style={circle,fill=red!20,draw=red,inner sep=0pt, minimum size=2mm},
	circ/.style={circle,fill=white,draw=black,inner sep=0pt, minimum size=1.2mm},
	keps/.style= {semithick,shorten >=1pt,shorten <=1pt,->},
	dotred/.style={circle,fill=black!50,inner sep=0pt, minimum size=2mm},
	generic/.style={semithick,shorten >=1pt,shorten <=1pt},
	gepsilon/.style={semithick,shorten >=1pt,shorten <=1pt,densely dashed},
	rootlab/.style={font=\scriptsize, transform canvas={yshift=-0.23cm}},
	dist/.style={ultra thick,draw=testcolor,shorten >=1pt,shorten <=1pt},
	testfcn/.style={ultra thick,testcolor,shorten >=1pt,shorten <=1pt,<-},
	testfcnx/.style={ultra thick,testcolor,shorten >=1pt,shorten <=1pt,<-,
		postaction={decorate,decoration={markings,mark=at position 0.6 with {\drawx}}}},
	kprime/.style={semithick,shorten >=1pt,shorten <=1pt,dotted,->},
	kprimex/.style={semithick,shorten >=1pt,shorten <=1pt,densely dashed,->,
		postaction={decorate,decoration={markings,mark=at position 0.4 with {\drawx}}}},
	kernel/.style={semithick,shorten >=1pt,shorten <=1pt,->},
	multx/.style={shorten >=1pt,shorten <=1pt,
		postaction={decorate,decoration={markings,mark=at position 0.5 with {\drawx}}}},
	kernelx/.style={semithick,shorten >=1pt,shorten <=1pt,->,
		postaction={decorate,decoration={markings,mark=at position 0.4 with {\drawx}}}},
	kepsilon/.style={semithick,shorten >=1pt,shorten <=1pt,densely dashed,->},
	kernel1/.style={->,semithick,shorten >=1pt,shorten <=1pt,postaction={decorate,decoration={markings,mark=at position 0.45 with {\draw[-] (0,-0.1) -- (0,0.1);}}}},
	kernel2/.style={->,semithick,shorten >=1pt,shorten <=1pt,postaction={decorate,decoration={markings,mark=at position 0.45 with {\draw[-] (0.05,-0.1) -- (0.05,0.1);\draw[-] (-0.05,-0.1) -- (-0.05,0.1);}}}},
	kernelBig/.style={semithick,shorten >=1pt,shorten <=1pt,decorate, decoration={zigzag,amplitude=1.5pt,segment length = 3pt,pre length=2pt,post length=2pt}},
	rho/.style={dotted,semithick,shorten >=1pt,shorten <=1pt},
	renorm/.style={shape=circle,fill=white,inner sep=1pt},
	labl/.style={shape=rectangle,fill=white,inner sep=1pt},
	xi/.style={circle,fill=symbols!10,draw=symbols,inner sep=0pt,minimum size=1.2mm},
	xix/.style={crosscircle,fill=symbols!10,draw=symbols,inner sep=0pt,minimum size=1.2mm},
	xib/.style={circle,fill=symbols!10,draw=symbols,inner sep=0pt,minimum size=1.6mm},
	xibx/.style={crosscircle,fill=symbols!10,draw=symbols,inner sep=0pt,minimum size=1.6mm},
	not/.style={circle,fill=symbols,draw=symbols,inner sep=0pt,minimum size=0.5mm},
	>=stealth,
	}

\def\drawx{\draw[-,solid] (-3pt,-3pt) -- (3pt,3pt);\draw[-,solid] (-3pt,3pt) -- (3pt,-3pt);}
\tikzset{
	root/.style={circle,fill=testcolor,inner sep=0pt, minimum size=2mm},
	dot/.style={circle,fill=black,inner sep=0pt, minimum size=1mm},
	empty_dot/.style={circle, draw=black,fill=white,inner sep=0pt, minimum size=1mm},
	small_dot/.style={circle,fill=black,inner sep=0pt, minimum size=0.9mm},
	var/.style={circle,fill=white,draw=black,inner sep=0pt, minimum size=2mm},
	var_blue/.style={circle,fill=blue!20,draw=blue,inner sep=0pt, minimum size=2mm},
	var_very_blue/.style={circle,fill=blue,draw=blue,inner sep=0pt, minimum size=2mm},
	var_teal/.style={circle,fill=teal!20, draw=teal,inner sep=0pt, minimum size=2mm},
	var_red_square/.style={regular polygon,regular polygon sides=4, draw, fill=red!20, draw=red, inner sep=0pt, minimum size=2.5mm, shape border rotate=45},
	var_red_triangle/.style={regular polygon,regular polygon sides=3, draw, fill=red!20, draw=red, inner sep=0pt, minimum size=3mm, shape border rotate=180},
	var_red/.style={circle,fill=red!20,draw=red,inner sep=0pt, minimum size=2mm},
	circ/.style={circle,fill=white,draw=black,inner sep=0pt, minimum size=1.2mm},
	keps/.style= {semithick,shorten >=1pt,shorten <=1pt,->},
	keps-random/.style= {semithick,shorten >=1pt,shorten <=1pt,->, decorate, decoration = {zigzag, segment length=1mm, amplitude=1mm}},
	kepsdot/.style= {semithick,densely dashed,shorten >=1pt,shorten <=1pt,->},
	dotred/.style={circle,fill=black!50,inner sep=0pt, minimum size=2mm},
	generic/.style={semithick,shorten >=1pt,shorten <=1pt},
	gepsilon/.style={semithick,shorten >=1pt,shorten <=1pt,densely dashed},
	rootlab/.style={font=\scriptsize, transform canvas={yshift=-0.23cm}},
	dist/.style={ultra thick,draw=testcolor,shorten >=1pt,shorten <=1pt},
	testfcn/.style={ultra thick,testcolor,shorten >=1pt,shorten <=1pt,<-},
	testfcnx/.style={ultra thick,testcolor,shorten >=1pt,shorten <=1pt,<-,
		postaction={decorate,decoration={markings,mark=at position 0.6 with {\drawx}}}},
	kprime/.style={semithick,shorten >=1pt,shorten <=1pt,dotted,->},
	kprimex/.style={semithick,shorten >=1pt,shorten <=1pt,densely dashed,->,
		postaction={decorate,decoration={markings,mark=at position 0.4 with {\drawx}}}},
	kernel/.style={semithick,shorten >=1pt,shorten <=1pt,->},
	multx/.style={shorten >=1pt,shorten <=1pt,
		postaction={decorate,decoration={markings,mark=at position 0.5 with {\drawx}}}},
	kernelx/.style={semithick,shorten >=1pt,shorten <=1pt,->,
		postaction={decorate,decoration={markings,mark=at position 0.4 with {\drawx}}}},
	kepsilon/.style={semithick,shorten >=1pt,shorten <=1pt,densely dashed,->},
	kernel1/.style={->,semithick,shorten >=1pt,shorten <=1pt,postaction={decorate,decoration={markings,mark=at position 0.45 with {\draw[-] (0,-0.1) -- (0,0.1);}}}},
	kernel2/.style={->,semithick,shorten >=1pt,shorten <=1pt,postaction={decorate,decoration={markings,mark=at position 0.45 with {\draw[-] (0.05,-0.1) -- (0.05,0.1);\draw[-] (-0.05,-0.1) -- (-0.05,0.1);}}}},
	kernelBig/.style={semithick,shorten >=1pt,shorten <=1pt,decorate, decoration={zigzag,amplitude=1.5pt,segment length = 3pt,pre length=2pt,post length=2pt}},
	rho/.style={dotted,semithick,shorten >=1pt,shorten <=1pt},
	renorm/.style={shape=circle,fill=white,inner sep=1pt},
	labl/.style={shape=rectangle,fill=white,inner sep=1pt},
	xi/.style={circle,fill=symbols!10,draw=symbols,inner sep=0pt,minimum size=1.2mm},
	xix/.style={crosscircle,fill=symbols!10,draw=symbols,inner sep=0pt,minimum size=1.2mm},
	xib/.style={circle,fill=symbols!10,draw=symbols,inner sep=0pt,minimum size=1.6mm},
	xibx/.style={crosscircle,fill=symbols!10,draw=symbols,inner sep=0pt,minimum size=1.6mm},
	not/.style={circle,fill=symbols,draw=symbols,inner sep=0pt,minimum size=0.5mm},
	>=stealth,
	}

\makeatletter
\def\DeclareSymbol#1#2#3{\expandafter\gdef\csname MH@symb@#1\endcsname{\tikz[baseline=#2,scale=0.15,draw=symbols]{#3}}\expandafter\gdef\csname MH@symb@#1s\endcsname{\scalebox{0.7}{\tikz[baseline=#2,scale=0.15,draw=symbols]{#3}}}}
\def\<#1>{\csname MH@symb@#1\endcsname}
\makeatother

\DeclareSymbol{1}{0}{\draw[white] (-.4,0) -- (.4,0); \draw (0,0)  -- (0,1.4) node[dot] {};}
\DeclareSymbol{2}{0}{\draw (-0.5,1.3) node[dot] {} -- (0,0) -- (0.5,1.3) node[dot] {};}
\DeclareSymbol{3}{0}{\draw (0,0) -- (0,1.3) node[dot] {}; \draw (-.7,1) node[dot] {} -- (0,0) -- (.7,1) node[dot] {};}
\DeclareSymbol{30}{-3}{\draw (0,0) -- (0,-1); \draw (0,0) -- (0,1.3) node[dot] {}; \draw (-.7,1) node[dot] {} -- (0,0) -- (.7,1) node[dot] {};}
\DeclareSymbol{31}{-3}{\draw (0,0) -- (0,-1) -- (1,0) node[dot] {}; \draw (0,0) -- (0,1.3) node[dot] {}; \draw (-.7,1) node[dot] {} -- (0,0) -- (.7,1) node[dot] {};}
\DeclareSymbol{32}{-3}{\draw (0,0) -- (0,-1) -- (1,0) node[dot] {}; \draw (0,0) -- (0,-1) -- (-1,0) node[dot] {}; \draw (0,0) -- (0,1.3) node[dot] {}; \draw (-.7,1) node[dot] {} -- (0,0) -- (.7,1) node[dot] {};}
\DeclareSymbol{20}{-3}{\draw (0,0) -- (0,-1);\draw (-.7,1) node[dot] {} -- (0,0) -- (.7,1) node[dot] {};}
\DeclareSymbol{22}{-3}{\draw (0,0.3) -- (0,-1) -- (1,0) node[dot] {}; \draw (0,0.3) -- (0,-1) -- (-1,0) node[dot] {};\draw (-.7,1) node[dot] {} -- (0,0.3) -- (.7,1) node[dot] {};}

\begin{document}

\date{\today}
\title{Martingale-driven integrals and singular SPDEs}
\author{P.~Grazieschi$^1$, K.~Matetski$^2$ and H.~Weber$^3$}
\institute{University of Bath, \email{p.grazieschi@bath.ac.uk} \and Michigan State University, \email{matetski@msu.edu} \and University of M\"{u}nster, \email{hendrik.weber@uni-muenster.de}}
\date{\today}
\titleindent=0.65cm

\maketitle

\begin{abstract}
We consider multiple stochastic integrals with respect to \cadlag martingales, which approximate a cylindrical Wiener process. We define a chaos expansion, analogous to the case of multiple Wiener stochastic integrals, for these integrals and use it to show moment bounds. Key tools include an iteration of the \BDG inequality and a multi-scale decomposition similar to the one developed in \cite{HQ18}.
 Our method can be combined with the recently developed discretisation framework for regularity structures  \cite{HM18, EH19} to prove convergence of interacting particle systems to singular stochastic PDEs.  A companion article \cite{3dIsingKac} applies the results of this paper to prove convergence of a rescaled Glauber dynamics for the three-dimensional Ising-Kac model near criticality to the $\Phi^4_3$ dynamics on a torus.
\end{abstract}

\setcounter{tocdepth}{1}
\tableofcontents

\section{Introduction}

We consider a class of \cadlag martingales which approximate a cylindrical Wiener process over a $d$-dimensional spatial domain, i.e. integrated-in-time space-time white noise. 
We develop a theory of iterated integrals with respect to these martingales and derive moment bounds. 

Our results serve as a technical tool for proving the convergence of Interacting Particle Systems (IPSs) to solutions of non-linear stochastic partial differential equations (SPDEs). The limiting SPDEs are usually of the form 
\begin{equation}\label{eq:general}
	\CL u = F(u,\nabla u) + \sigma(u) \xi,
\end{equation}
where $\CL$ is a linear parabolic operator (e.g. $\CL = \partial_t - \Delta$), $\xi$ is an irregular random noise (e.g. a Gaussian white noise) and $F$ and $\sigma$ are  local non-linearities. 
There are by now a number of convergence results of this type. These include $1+1$-dimensional surface growth models rescaling to the KPZ equation
\begin{equation}\label{eq:KPZ}
\partial_t  h - \partial_x^2 h = -  (\partial_x h)^2 + \xi,
\end{equation}
e.g. \cite{BG,MR3176353,Dembo:2013aa,1505.04158,1602.01908}, long range (Kac) spin models 
rescaling to $\phi^{2n}$ dynamics
\begin{equation}\label{eq:phi_2n}
\partial_t \phi - \Delta \phi = - \phi^{2n-1} + \xi,
\end{equation}
in one  \cite{MR1317994,FR} and two dimensions \cite{MR1661764,IsingKac,ShenWeber,Iberti}  as well as
diffusions in random environment rescaling to the parabolic Anderson model / multiplicative stochastic heat equation 
\begin{equation}\label{eq:SHE}
\partial_t u - \Delta u = u \xi, 
\end{equation}
\cite{MartinPerkowski,EH21}. Ultimately, the  motivating goal of the theory developed in this article is to show the convergence of the Ising-Kac model to the $\phi^4$ dynamics in three dimensions, and this is accomplished in our companion article \cite{3dIsingKac}. 

A common feature of all of these limiting results is that particle systems are simultaneously rescaled (i.e. observed on large scales) while a certain parameter is changed. The specific nature of this parameter depends on the model under consideration; examples are the strength of the \emph{weak} asymmetry in exclusion processes approximating the KPZ equation \cite{BG}, or the range of the interaction in Kac-models \cite{MR1661764}.  The typical strategy is to tune down the effect of the ``non-linearity'' as one moves to larger scales. 
This procedure is necessary to obtain convergence to one of the SPDEs \eqref{eq:KPZ}, \eqref{eq:phi_2n}, \eqref{eq:SHE}   and  reflects the fact that the SPDEs are themselves not scale-invariant. The fact that a relatively small class of SPDEs arises as scaling limit of this type for a relatively large number of particle systems sharing just a few key characteristics is sometimes referred to as  weak universality.

A key  technical challenge in deriving such scaling results is the low regularity of the solutions of the limiting equations \eqref{eq:KPZ},\eqref{eq:phi_2n},\eqref{eq:SHE}: the noise term $\xi$ is typically very irregular, leading to irregular solutions which in turn lead to difficulties in dealing with the non-linearities. This problem does not appear in more common Gaussian fluctuation limits \cite{KL99} --- while the solutions of the limiting equations there are typically also irregular, this is less problematic due to the absence of a non-linear term.
Good theories for non-linear SPDEs  and their renormalisation have only been developed over the last years, including Hairer's theory of regularity structures \cite{Regularity}, the theory of paracontrolled distributions put forward by Gubinelli, Imkeller and Perkowski \cite{MR3406823} and more recently theories of weak solutions for specific equations, in particular the KPZ equation \cite{MR3176353,MR3327509,MR3758149,MR4168394}.

The theory of regularity structures and the theory of paracontrolled distributions both build on a two-step approach: first, the construction of approximate solutions building a local expansion (the \emph{model} in the jargon of \cite{Regularity}) which relies on probabilistic tools, in particular Gaussian analysis and explicit calculations of covariance functions, and  second, analytic techniques (in particular regularity estimates and commutator estimates) for  dealing with the remainder. 
The weak solution theories developed in \cite{MR3176353,MR3327509,MR3758149,MR4168394} use a very different approach and make explicit use of the invariant Gaussian measure to give a direct characterisation of the generator of the dynamics. 

In principle, both approaches can be used to study scaling limits. In situations, where a simple invariant measure for an interacting particle system is given, the weak solution approach has proved highly efficient, see e.g.  \cite{MR3176353,MR4312852,MR4060852,MR3908648}. The approach which consists of mimicking the theory of regularity structure / paracontrolled distributions has also been implemented in a few examples, in particular \cite{IsingKac,MartinPerkowski,MR3592748,EH21}. 
Still,  implementing this programme for ``interesting'' limiting equations remains a challenging enterprise: for the second, deterministic, step of the analysis a systematic theory has been developed in \cite{HM18,EH19}, but the first probabilistic part remains challenging, because the number of terms in this perturbative expansion (the ``trees'') can become prohibitively large when looking at interesting equations. For the continuum there is by now a very systematic treatment for the trees (see \cite{HQ18,1612.08138,OttoTsatsoulisAlia,Hairer_Steele}). The aim of this paper is to develop a --- at least somewhat --- systematic approach to bound these trees for approximations of white noise. A particular focus is on the jump martingales that typically arise in the analysis of IPSs. 

On a technical level: the noise approximations we deal with are of bounded variation, but discontinuous because of the jumps. Therefore, the non-linear functionals that make up the model can rigorously be written in terms of integrals with respect to product measure in the underlying noise. We then decompose these integrals according to ``diagonals'' or ``contractions''. This is in the spirit of the Wiener chaos decomposition, however many more terms than in the Gaussian case arise, as in the latter only diagonals where precisely two coordinates coincide, make a non-vanishing contribution. In our noise approximations, many more ``diagonals'' appear, and  we aim to show that their impact vanishes as space-time white noise is approached. 
Our main technical tool is an iteration of a \BDG (BDG) type inequality. For our purpose, the most convenient form is in terms of the predictable quadratic variation with an error term that depends on the size of jumps, as was used previously in \cite[Lemma 4.1]{IsingKac} . The advantage is that under our assumptions (which are motivated by the analysis of the Ising-Kac model \cite{3dIsingKac}), explicit and optimal bound on the predictable quadratic variation are available. The error term does not matter too much in the Ising-Kac application, because the size of individual jumps is suppressed by the smoothing from the Kac-potential.
Another key assumption we need to make, is that the magnitude of the jumps of the martingales is fixed by a deterministic constant. This allows to rewrite contractions of an odd number of variables in terms of a martingale and ultimately permits to prove that in the Kac-Ising application these contractions  vanish in the limit, even though they are integrated against a very singular kernel. 

\subsection{Structure of the article}

In Section~\ref{sec:MultIntegrals} we define multiple stochastic integrals with respect to \cadlag~square integrable martingales. In Section~\ref{sec:OneFoldBounds} we derive moment bounds on stochastic integrals with respect to only one variable, while moment bounds on multiple integrals are obtained in Section~\ref{sec:GeneralBounds}. Section~\ref{sec:renormalized-integrals} is devoted to renormalised stochastic integrals and their moment bounds. In Section~\ref{sec:GeneralizedConvolutions} we analyse stochastic integrals with kernels given by generalised convolutions, which are typical objects in the theory of regularity structures. As an example, we apply the result in Section~\ref{sec:applicationStochQuant} to a discrete approximation of the $\Phi^4_3$ equation. 

\subsection{Notation}
We use the standard notation $\N = \{ 1, 2, 3, \ldots \}$ for the set of natural numbers, $\N_0$ for $\N \cup \{ 0 \}$ and the set $\R_+ := [0, \infty)$ for the time variables. For $n \in \N$ we define $\set{n} := \{ 1, \ldots, n \}$.\label{lab:set_n} We use $\1_A$ for the indicator function of the set $A$.

For $\Lambda_0$ being either $(\R / \Z)^d$ or $\R^d$ we use the standard notation $\SD'(\Lambda_0)$ for the space of distributions on $\Lambda_0$. For $n \geq 0$, the space $\SC^n(\Lambda_0)$ contains all $n$-times continuously differentiable functions on $\Lambda_0$, and we write $\SC(\Lambda_0)$ for this space when $n = 0$. The Skorokhod space of \cadlag functions on $[0, T]$ with values in $\SD'(\Lambda_0)$ is denoted by $D([0,T], \SD'(\Lambda_0))$. 

Given a random variable $X$ and some $p \geq 1$, we use the following shorthand notation for the stochastic $L^p$ norm
\begin{equation}\label{eq:MomentNotation}
	\E_p X :=  \E \big[ | X |^p \big]^{1 / p}.
\end{equation}

In estimates we often use ``$\lesssim$'', which means that the bound ``$\leq$'' holds up to a constant which is independent of the quantities relevant in our statements, which will be always clear from the context. If we want to indicate dependence of the proportionality constant on some parameters $\alpha$, $\beta$, $\ldots$, we write ``$\lesssim_{\alpha, \beta, \ldots}$''.

Finally, let $\Le = (\eps \Z \slash \Z)^d$ be a discrete torus with mesh size $\eps > 0$. 
For $T > 0$, $p \geq 1$ and for a function $F: [0, T] \times \Le \to \R$, we define
\begin{equation}\label{eq:norm-eps}
	\Vert F \Vert_{L^p_\eps} := \biggl( \eps^d \sum_{x \in \Le} \int_{0}^T | F(r,x) |^p \d r \biggr)^{\frac{1}{p}},
\end{equation}
that is, we take the $L^p$ norm in time and the $l^p$ norm in space with a weight $\eps^d$ on the points of the lattice. This and several other norms in the article depend on the parameter $T$, but we omit this dependence from our notation. 

\subsection*{Acknowledgments}
PG was supported by a scholarship from the EPSRC Centre for Doctoral Training in Statistical Applied Mathematics at Bath (SAMBa), under the project EP/L015684/1.

KM was partially supported by NSF grant DMS-2321493. HW was supported by the Royal Society through the University Research Fellowship UF140187, by the Leverhulme Trust through a Philip Leverhulme Prize and by the European Union (ERC, GE4SPDE, 101045082). HW acknowledges funding by the Deutsche Forschungsgemeinschaft under Germany’s Excellence Strategy EXC 2044 390685587, Mathematics M\"{u}nster: Dynamics -- Geometry -- Structure.

PG and HW thank the Isaac Newton Institute for Mathematical Sciences for hospitality during the programme \textit{Scaling limits, rough paths, quantum field theory}, which was supported by EPSRC Grant No. EP/R014604/1.

\section{Integrals with respect to \cadlag~martingales}
\label{sec:MultIntegrals}

\subsection{Properties of \cadlag~martingales}
\label{sec:Appendix2}

Following \cite[Ch.~I.4]{JS03}, we recall some properties of martingales which are used in the article. Let $(M_t)_{t \geq 0}$ and $(N_t)_{t \geq 0}$ be two \cadlag~square-integrable martingales on the same filtered probability space. Their {\it predictable quadratic covariation} $\langle M, N\rangle_t$ is the unique adapted process with bounded total variation, such that $M_t N_t - \langle M, N\rangle_t$ is a martingale. The {\it quadratic covariation} $[M, N]_t$ is defined by
\begin{equation}
	[M, N]_t := M_t N_t - M_0 N_0 - \int_{0}^t M_{s^-} \d N_s - \int_{0}^t N_{s^-} \d M_s\;,
\end{equation}
where $M_{s-} := \lim_{r \uparrow s} M_{r}$ is the left limit of $M$ at time $s$. Another way to define these quadratic covariations is the following: if $0 = t_0 \leq \cdots \leq t_n = t$ is a partition with diameter $\max_i (t_{i+1} - t_{i})$ tending to zero as $n \to \infty$, then $[M, N]_t$ is equal to the limit in probability of the sums $\sum_{i=0}^{n-1} (M_{t_{i+1}} - M_{t_{i}}) (N_{t_{i+1}} - N_{t_{i}})$ as $n \to \infty$ (see \cite[Thm.~I.4.47]{JS03}), and $\langle M, N\rangle_t$ is the probability limit of the sums $\sum_{i=0}^{n-1} \E [(M_{t_{i+1}} - M_{t_{i}}) (N_{t_{i+1}} - N_{t_{i}}) | \Ff_{t_i} ]$, where $(\Ff_t)_{t \geq 0}$ is the underlying filtration \cite[Prop.~I.4.50]{JS03}. The difference of the two bracket processes $[M, N]_t - \langle M, N\rangle_t$ is always a \cadlag~martingale \cite[Prop.~I.4.50]{JS03}. In the case $M = N$, it will be convenient to use the shorthands $[M]_t = [M, M]_t$ and $\langle M \rangle_t = \langle M, M\rangle_t$.

We will use the \BDG inequality in the following form, which is obtained by approximating $M$ by discrete-time martingales and applying the discrete-time \BDG inequality \cite{HH80}. 

\begin{proposition}\label{prop:BDG}
	Let $(M_t)_{t \in [0, T]}$ be a \cadlag square integrable martingale. Then, for any $p \geq 1$ there exists a constant $C > 0$ depending on $p$ such that 
	\begin{equation}\label{eq:BDG}
		\E \biggl[ \sup_{t \in [0, T]} | M_t |^p \biggr]^{\frac{1}{p}} \leq C \biggl( \E \left[ \langle M \rangle_t^{p/2}\right]^{\frac{1}{p}} + \E \biggl[ \sup_{t \in [0, T]} | \Delta_t M |^p \biggr]^{\frac{1}{p}} \biggr),
	\end{equation}
	where $\Delta_t M := M_t - M_{t-}$ is a jump at time $t$.
\end{proposition}

\subsection{Assumptions on martingales}

Let $d \in \N$ and let $\Lambda_0$ be a $d$-dimensional torus $(\R \slash \Z)^d$. For $\eps > 0$, let $\Le$\label{lab:domains} be a discretisation of $\Lambda_0$ with mesh size $\eps$, i.e. $\Le$ is a $d$-dimensional discrete torus $(\eps \Z \slash \Z)^d$ (in this case we need $\eps^{-1}$ to be integer). The moment bounds for stochastic integrals, which we prove in the following sections, depend on the Lebesgue measure of the domain $\Lambda_0$, which is bounded. Let $\CD_\eps := \R_+ \times \Lambda_\eps$ be a discretised space-time domain, and let $\CD_{\eps, t} := [0, t) \times \Lambda_\eps$  be the space-time domain with time horizon $t > 0$. 

For a function $u_\eps$ on the domain $\CD_\eps$, we introduce its natural extensions to the space of distributions
\begin{equation}\label{eq:extension}
(\iota_\eps u_\eps)(\phi) := \sum_{x \in \Lambda_\eps} \eps^d \int_{\R_{+}}\!\! u_\eps(t,x) \phi(t,x) \d t, \qquad (\iota_\eps u_\eps)(t, \psi) := \sum_{x \in \Lambda_\eps} \eps^d u_\eps(t,x) \psi(x),
\end{equation}
where $\phi : \R_{+} \times \Lambda_0 \to \R$ and $\psi : \Lambda_0 \to \R$ are smooth compactly supported functions.

Let  $(\Omega, \Ff, (\Ff_t)_{t \geq 0}, \P)$ be a filtered probability space, which satisfies the ``usual conditions'' (i.e. completeness and right-continuity \cite[Def.~I.1.3]{JS03}). We then introduce a family of \cadlag martingales $( \M_{\eps}(t, x) )_{t \geq 0}$, indexed by points $x \in \Le$. Let $\M_{\eps}(t-, x) = \lim_{\de \to 0+} \M_{\eps}(t-\de, x)$ be the left-limit of $\M_{\eps}(x)$ at time $t$ and let $\Delta_t \M_{\eps}(x) := \M_{\eps}(t, x) - \M_{\eps}(t-, x)$ denote the jump at time $t$. We make the following assumption on these martingales.

\begin{assumption} \label{a:Martingales}
	For $\eps > 0$, we assume that $(\M_{\eps}(t, x))_{t \geq 0}$ are \cadlag~square-integrable martingales with the following properties.
	\begin{enumerate}
		\item\label{it:bracket} The predictable quadratic covariation $\big\langle \M_{\eps}(x), \M_{\eps}(y) \big\rangle_t$ vanishes whenever $x \neq y$, and
		\begin{equation}\label{eq:quadr_var_formula}
			\big\langle \M_{\eps}(x)\big\rangle_t = \eps^{-d} \int_0^t \C_{\eps}(s, x) \d s,
		\end{equation}
		where $(s, x) \mapsto \C_{\eps}(s, x)$ is a progressively measurable stochastic process satisfying $|\C_{\eps}(s, x) | \lesssim 1$ a.s. uniformly in $s$ and $x$.  The proportionality constant in this bound is non-random.
		
		\item\label{it:jumps1} Two martingales almost surely never jump simultaneously, i.e. for any $T > 0$
		\[ \P \bigl( \Delta_t \M_{\eps}(x) \Delta_t \M_{\eps}(y) = 0,\, \forall x \neq y,\, \forall t \in [0,T] \bigr) = 1. \]
		
		\item\label{it:jumps2} There exist $\kone > -\frac{d}{2}$\label{lab:k2} and a non-random value $c > 0$ such that if $\Delta_t \M_{\eps}(x) \neq 0$, then $|\Delta_t \M_{\eps}(x)| = c \eps^{\kone}$ a.s. for all $x \in \Le$ and $t \geq 0$.
		\item \label{it:dynamics} The martingale $\M_{\eps}(s, x)$ follows a dynamics which can be expressed in the form 
		\begin{equation}\label{eq:martingales_dynamics}
			\M_{\eps}(t, x) = J_\eps(t, x) - \eps^{-\kone - d} \int_0^t \Cd_\eps(s, x) \d s,
		\end{equation}
		where $t \mapsto J_\eps(t, x)$ is a pure jump process (i.e. $J_\eps(t, x) = \sum_{0 \leq s \leq t} \Delta_s \M_{\eps}(x)$) and where $t \mapsto  \Cd_\eps(t, x)$ is  a progressively measurable process such that $| \Cd_\eps(t, x) | \lesssim 1$ a.s. uniformly in $x$ and $t$.  The proportionality constant in the last estimate is non-random.
	\end{enumerate}
\end{assumption}

\begin{remark}
	Assumption~\ref{a:Martingales}\eqref{it:bracket} implies that in the case $\C_{\eps}(s, x) = 1$ the quadratic variation of the martingales approximates the quadratic variation of a cylindrical Wiener process, see also the following Lemma \ref{lem:Wiener}. Assumption~\ref{a:Martingales}\eqref{it:jumps1} is satisfied in many applications, e.g. when jumps are sub-sampled from independent Poisson processes. Assumption~\ref{a:Martingales}\eqref{it:jumps2} implies that the size of an individual jump is smaller than the size of $\M_{\eps}(t, x)$ for bounded $t$, which is of order $\eps^{-d/2}$ by Assumption~\ref{a:Martingales}\eqref{it:bracket}. We show in Lemma~\ref{lem:jumps} below that Assumptions ~\ref{a:Martingales}\eqref{it:bracket} and \ref{a:Martingales}\eqref{it:jumps2} combined imply that  jumps happen with frequency $\eps^{d + 2 \kone}$.
\end{remark}

We will use the following martingales (see Section~\ref{sec:Appendix2})
\begin{equation}\label{eq:renorm-martingale}
	\Nm_\eps(t, x) := \eps^{-\kone} \bigl( \big[ \M_{\eps}(x) \big]_t - \big\langle \M_{\eps}(x) \big\rangle_t \bigr).
\end{equation}
The multiplier $\eps^{-\kone}$ in \eqref{eq:renorm-martingale} is chosen to have the following. 

\begin{lemma}\label{lem:martingales-N}
The martingales $\Nm_\eps(t,x)$ satisfy Assumption~\ref{a:Martingales} with the same value of $\kone$ (but with the constant $c^2$ in place of $c$) and with $\Cd_\eps$ replaced by $\C_{\eps}$.
 
\end{lemma} 

\begin{proof}
Assumptions~\ref{a:Martingales}(\ref{it:bracket}) and (\ref{it:dynamics}) yield $[ \M_{\eps}(x) ]_t = \sum_{0 \leq s \leq t} (\Delta_s \M_{\eps}(x))^2$ and 
\begin{equation*}
\Nm_\eps(t, x) = \eps^{-\kone} \sum_{0 \leq s \leq t} (\Delta_s \M_{\eps}(x))^2 - \eps^{-\kone -d} \int_0^t \C_{\eps}(s, x) \d s.
\end{equation*}
This lemma follows readily from these two identities and properties of the martingales $\M_{\eps}(t,x)$.
\end{proof}

For $x \in \Le$ and for a bounded set $A \subset \R_+$, we define\label{lab:jumps_number}
\begin{equation}\label{eq:number-of-jumps}
	\nn^{\eps}_{A}(x) := \#\{t \in A : \Delta_t \M_{\eps}(x) \neq 0\}
\end{equation}
to be the number of jumps of the martingale $\M_{\eps}(x)$ in $A$. We are going to show that Assumption~\ref{a:Martingales} implies moment bounds for the number of jumps.

\begin{lemma}\label{lem:jumps}
For any $[a,b] \subset \R_+$ and any $p \geq 1$ the number of jumps satisfies
		\begin{equation}\label{eq:jumps_bound}
			\sup_{\eps \in (0,1]} \sup_{x \in \Le} \E_p | \nn^\eps_{[ \eps^{d + 2 \kone} a, \eps^{d + 2 \kone} b]} (x)| < \infty,
		\end{equation}
		locally uniformly in $a, b$.
\end{lemma}

\begin{proof}
Assumptions~\ref{a:Martingales}(\ref{it:jumps2}) and (\ref{it:dynamics}) yield $[ \M_{\eps}(x) ]_t = \sum_{0 \leq s \leq t} (\Delta_s \M_{\eps}(x))^2 = c^2 \eps^{2 \kone} \nn^{\eps}_{[0, t]}(x)$. Then for any $p \geq 1$ we have $\E_p | \nn^\eps_{[0, t]} (x)| = c^{-2} \eps^{-2 \kone} \E_p [ \M_{\eps}(x) ]_t$. Using the martingales \eqref{eq:renorm-martingale} and applying Minkowski's inequality, we get furthermore 
\begin{equation*}
\E_p | \nn^\eps_{[0, t]} (x)| \leq c^{-2} \eps^{-2 \kone} \E_p \langle \M_{\eps}(x) \rangle_t + c^{-2} \eps^{-\kone} \E_p |\Nm_\eps(t, x)|.
\end{equation*}
The first term is bounded using Assumptions~\ref{a:Martingales}(\ref{it:bracket}) as $\eps^{-2 \kone} \E_p \langle \M_{\eps}(x) \rangle_t \lesssim \eps^{-d-2 \kone} t$, while to bound the second term we apply the \BDG inequality \eqref{eq:BDG}:
\[
\eps^{-\kone} \E_p |\Nm_\eps(t, x)| \lesssim \eps^{-\kone} \E_p \langle\Nm_\eps(x) \rangle_t^{1/2} + \eps^{-\kone} \E_p \sup_{s \in [0, t]} | \Delta_s \Nm_\eps(x) |,
\]
where the proportionality constant depends only on $p$. Lemma~\ref{lem:martingales-N} allows to bound the preceding expression by a constant multiple of $\eps^{-\kone - d/2} t^{1/2} + 1$. Hence, we have 
\begin{equation*}
\E_p | \nn^\eps_{[0, t]} (x)| \lesssim \eps^{-d-2 \kone} t + \eps^{-\kone - d/2} t^{1/2} + 1.
\end{equation*}
Since the proportionality constant in this bound is independent of $t$, we can replace $t$ by $\eps^{d + 2 \kone} t$ to get $\E_p | \nn^\eps_{[0, \eps^{d + 2 \kone} t]} (x)| \lesssim t + 1$ from which the required bound follows.
\end{proof}

\begin{lemma}\label{lem:TV}
Let $\Vert \bigcdot \Vert_{\TV([0,T])}$ be the total variation norm on the interval $[0, T]$. Then
\begin{equation*}
\sup_{\eps \in (0, 1]} \sup_{x \in \Le} \eps^{\kone + d}\, \E_p \Vert \M_{\eps}(x) \Vert_{\TV([0,T])}  < \infty
\end{equation*}
for any $T > 0$ and $p \geq 1$. 
\end{lemma}

\begin{proof}
Assumption~\ref{a:Martingales}(\ref{it:dynamics}) yields 
\begin{equation*}
\Vert \M_{\eps}(x) \Vert_{\TV([0,T])} \leq \sum_{0 \leq s \leq T} |\Delta_s \M_{\eps}(x)| + \eps^{-\kone - d} \int_0^T |\Cd_\eps(s, x)| \d s \lesssim \eps^{\kone} \nn^{\eps}_{[0, T]}(x) + T \eps^{-\kone - d} 
\end{equation*}
a.s. uniformly in $x$, where we used the properties $|\Delta_s \M_{\eps}(x)| \lesssim \eps^{\kone}$ and $|\Cd_\eps(s, x)| \lesssim 1$. Then the required bound follows from Lemma~\ref{lem:jumps}.
\end{proof}

The next result shows that martingales satisfying Assumption~\ref{a:Martingales} weakly converge to a cylindrical Wiener process \cite{DPZ}.

\begin{lemma}\label{lem:Wiener}
Let martingales $(\M_{\eps}(t, x))_{t \geq 0}$, with $x \in \Le$ and either $\Le=(\eps \Z \slash \Z)^d$ or $\Le= \eps \Z^d$, satisfy Assumption~\ref{a:Martingales} (except possibly Assumption~\ref{a:Martingales}(\ref{it:dynamics})) and let $\M_{\eps}(0, x) = 0$. For every continuous, compactly supported function $\phi : \Lambda_0 \to \R$, for every fixed $t > 0$ and some constant $\sigma > 0$, let the following limit hold in distribution
\begin{equation}\label{eq:C_converges}
 \lim_{\eps \to 0}\int_0^t (\iota_\eps \C_{\eps})(s, \phi) \d s = \sigma t \int_{\Lambda_0} \phi(x) \d x.
\end{equation}
Then the martingales $(\M_{\eps}(t, x))_{t \geq 0}$ weakly converge in the Skorokhod topology $D(\R_+, \SD'(\Lambda_0))$ to a cylindrical Wiener process on $L^2(\Lambda_0)$ with variance $\sigma$. 
\end{lemma}

\begin{proof}
Let us take a continuous and compactly supported function $\phi$ and consider the martingales $t \mapsto (\iota_\eps \M_{\eps})(t, \phi)$. To prove tightness of the laws of these martingales, we will use a version of the Aldous' criterion \cite[Corollary~16.11]{Billingsley} (see \cite{Aldous} for the original result).

By Assumption~\ref{a:Martingales}(\ref{it:jumps1}),(\ref{it:jumps2}), for any $T > 0$, the jumps of this martingale can be a.s. bounded by $\sup_{0 \leq t \leq T}|\Delta_t (\iota_\eps \M_{\eps})(\phi)| \lesssim \eps^{d + \kone} \|\phi\|_{L^\infty}$, which vanishes as $\eps \to 0$. Furthermore, for fixed $t \geq 0$ the process $\delta \mapsto \M_{\eps}(t + \delta, x) - \M_{\eps}(t, x)$ is a martingale with respect to the filtration $(\F_{t + \delta})_{\delta \geq 0}$ with the predictable quadratic covariation $\eps^{-d} \int_{t}^{t+\delta} \C_{\eps}(s, x) \d s$. Hence, for any stopping time $\tau \in [0, T]$ and for any $\delta \in (0,1]$ we apply the \BDG inequality \eqref{eq:BDG} to get
\begin{align*}
&\E  | (\iota_\eps \M_{\eps})(\tau + \delta, \phi) - (\iota_\eps \M_{\eps})(\tau, \phi)| \leq \E \Bigl[ \sup_{0 \leq t \leq T} \E \bigl[  | (\iota_\eps \M_{\eps})(t + \delta, \phi) - (\iota_\eps \M_{\eps})(t, \phi)| \big| \F_t \bigr]\Bigr] \\
&\qquad \lesssim \E \Bigl[ \sup_{0 \leq t \leq T} \E \Bigl[ \int_t^{t + \delta} (\iota_\eps \C_{\eps})(s, \phi^2) \d s \Big| \F_t \Bigr]^{\frac{1}{2}} \Bigr] + \E \Bigl[ \sup_{0 \leq t \leq T + \delta} \E \bigl[ |\Delta_t (\iota_\eps \M_{\eps})(\phi)| \big| \F_t \bigr] \Bigr].
\end{align*}
Using Assumption~\ref{a:Martingales}(\ref{it:bracket}), the first term is bounded by a constant proportional to $\delta^{1/2}$, while the second term is bounded by a constant proportional to $\eps^{d + \kone}$. Hence, for any $\alpha > 0$ the Markov inequality yields 
\begin{equation*}
\P \bigl( | (\iota_\eps \M_{\eps})(\tau + \delta, \phi) - (\iota_\eps \M_{\eps})(\tau, \phi)| \geq \alpha \bigr) \leq C \frac{1}{\alpha} (\delta^{1/2} + \eps^{d + \kone}),
\end{equation*}
and the assumptions of \cite[Corollary~16.11]{Billingsley} are satisfied. This gives tightness of the stochastic processes $t \mapsto (\iota_\eps \M_{\eps})(t, \phi)$ in $D([0,T], \R)$, and moreover every limiting point is in $\SC([0,T], \R)$. From \cite[Corollary~IX.1.19]{JS03} we conclude that every limiting point is a martingale. Finally, \cite{Mitoma} yields tightness of $t \mapsto \iota_\eps \M_{\eps}(t)$ in $D([0,T], \SD'(\Lambda_0))$. Convergence \eqref{eq:C_converges} and Assumption~\ref{a:Martingales}(\ref{it:bracket}) imply the limit in distribution
\begin{equation*}
 \lim_{\eps \to 0} \langle (\iota_\eps \M_{\eps})(\phi)\rangle_t = \sigma t \|\phi\|_{L^2}^2,
\end{equation*}
combining which with the L\'{e}vy characterization theorem we conclude that the limit of $\M_{\eps}$ in the Skorokhod topology $D([0,T], \SD'(\Lambda_0))$ is a cylindrical Wiener process with variance $\sigma$.
\end{proof}

\subsection{Iterated integrals with respect to martingales}
\label{sec:iterated-integrals}

Let $\bM_{\eps}$ be the random measure (recall that the paths of $\M_\eps$ are almost surely of bounded variation) over $\CD_\eps$ such that, for any $F: \CD_{\eps} \to \R$ which is continuous in the time variable,
\begin{equation}\label{eq:one-fold_integral}
	\int_{\CD_{\eps}} F(z) \, \d \bM_{\eps}(z) = \sum_{x \in \Le} \eps^d \int_{s = 0}^\infty F(s,x) \, \d \M_{\eps}(s, x).
\end{equation}
Atoms of $\bM_{\eps}$ correspond to jumps of the martingales; in fact, by Assumption~\ref{a:Martingales}, the magnitude of the jumps of the martingales is deterministic in absolute value and equal to $c\eps^{\kone}$; as such, given that different martingales never jump simultaneously, the absolute value of atoms of $\bM_{\eps}$ is always equal to $c \eps^{d+\kone}$ and this quantity---see again Assumption~\ref{a:Martingales}\eqref{it:jumps2}---goes to zero as $\eps \to 0$. 

Let $n \in \N$ and let $\bM_{\eps}^{ n}$ be the product measure on $\CD_\eps^{ n}$.  
We want to analyse integrals of the form 
\begin{equation}\label{eq:integral}
	\int_{\CD_{\eps}^{ n}} F \, \d \bM_{\eps}^{ n} := \int_{\CD_{\eps}^{n}} F ( z_1, \ldots, z_n ) \prod_{i=1}^n \d \bM_{\eps} (z_i).
\end{equation}
Here and throughout this section $F: \CD_{\eps}^{n}  \to \R$ is a function of $n$ space-time variables, which is continuous in all time variables. 
\smallskip

\subsubsection{Contractions and orderings}
\label{sec:contractions}

The following is motivated by the analysis of $n$-fold iterated integrals against space-time white noise, which are conventionally defined as limits of Riemann sums that cut out diagonals (see e.g.  \cite[Section~1.1.2]{nualart2006malliavin}, \cite[Section 9]{kuo2005introduction}, or \cite[Appendix A]{chandraweberlecturenotes}):
we call a \emph{contraction} on $\set{n}$ any equivalence relation on $\set{n}$, and its equivalent classes are called \emph{components}. We use the symbol $\fC(n)$\label{lab:contractions} to denote the set of contractions and the symbol $\fC_m(n)$ to denote contractions on $\set{n}$ with $m$ components.  For $a, b \in \set{n}$ and $\ga \in \fC(n)$, we use the notation $a \sim_\ga b$ to indicate that $a$ and $b$ belong to the same component, and we denote by $[a]_\ga $ the component containing $a$. 

For $\ga \in \fC(n)$ we define\label{lab:E_gamma} 
\[
	\widetilde{\CD}_\ga := \Big\{ (z_1, \ldots, z_n) \in \CD_\eps^{n} : \, z_i = z_j \text{ if and only if } i \sim_\ga j \Big\}.
\]
The sets $\widetilde{\CD}_\ga$ form a partition of $\CD^{ n}_\eps = \bigsqcup_{\ga \in \fC(n)} \widetilde{\CD}_\ga$,
so that we can write the integral in \eqref{eq:integral} as
\begin{equation}\label{eq:expansion1}
	\int_{\CD_\eps^{ n}} F \, \d \bM^{ n}_{\eps} = \sum_{\ga \in \fC(n)} \int_{\widetilde{\CD}_\ga} F \, \d \bM^{n}_{\eps}.
\end{equation}

The next lemma shows that  under the measure $\bM^{n}_\eps$  we can disregard all the points in $\CD^{ n}_\eps$ which have different space component but  the same time component. 
To this end, for $i \neq j \in \set{n}$ and for $T>0$ we define
\begin{equation*}
	\CC_{i,j}(T) := \big\{ (z_1, \ldots, z_n) \in \CD^{ n}_\eps: \, z_i = (s_i, x_i), z_j = (s_j, x_j) \text{ with } s_i = s_j  \in [0,T), \, x_i \neq x_j \big\}, 
\end{equation*}
as well as 
\begin{equation*}
	\CC:=  \bigcup_{T \in \N}  \bigcup_{i \neq j \in \set{n}} \CC_{i,j}(T).
\end{equation*}

\begin{lemma}\label{lem:nullMeasureLemma}
	Let $n \in \N$, $n \geq 2$. Then  $\bM^{ n}_{\eps}(\CC) = 0$ almost surely.
\end{lemma}
\begin{proof}
It suffices to show that for all $T>0$ and all  $i \neq j  \in \set{n}$ we have $\bM_\eps^{n} (\CC_{i,j}(T)) = 0$.  We have
\[
\bM^{ n}_{\eps} \big( \CC_{i,j}(T) \big)  = \bM_\eps^{ 2}  \Big( \Big\{ \big( (s, x_1), (s, x_2) \big): \, x_1 \neq x_2 \text{ and } s \in [0, T) \Big\} \Big) \bM^{ n-2}_{\eps} \big( \CD^{ n-2}_{\eps,T} \big).
\]
The quantity $\bM^{ n-2}_{\eps} \big( \CD^{n-2}_{\eps,T} \big) $ is almost surely finite. Recall that the measure $\bM^{ n}_{\eps}$ is defined in terms of the martingales $\M_\eps$, see \eqref{eq:one-fold_integral}, and that the martingales $\M_{\eps}$ are given by a sum of a jump part and an absolutely continuous part, see Assumption~\ref{a:Martingales}\eqref{it:dynamics}. The absolutely continuous part does not contribute to the diagonal considered here, and  we get
	\begin{align*}
\bM^{ 2}_{\eps} \Big( \Big\{ \big( (s, x_1), (s, x_2) \big): \, x_1 \neq x_2 \text{ and } s \in [0, T) \Big\} \Big)  
		\leq  \sum_{\substack{x_1, x_2 \in \Le \\ x_1 \neq x_2}} \eps^{2d}  \sum_{0 \leq s < T} \big|  \Delta_s \M_{\eps}(x_1) \Delta_s \M_{\eps}(x_2)  \big|.
	\end{align*}
	 By Assumption~\ref{a:Martingales}\eqref{it:jumps1}  the last expression is $0$ almost surely. 
\end{proof}

Combining \eqref{eq:expansion1} with Lemma \ref{lem:nullMeasureLemma}, we  obtain
\begin{equation}\label{eq:expansion2}
	\int_{\CD_\eps^{ n}} F \, \d \bM^{ n}_{\eps} = \sum_{\ga \in \fC(n)} \int_{\CD_\ga} F \, \d \bM^{ n}_{\eps},
\end{equation}
where $\CD_{\ga} := \widetilde{\CD}_\ga \setminus \CC$.\
\medskip

Lemma~\ref{lem:nullMeasureLemma} ensures that we can decompose the integral further according to the order between the different components with respect to the time argument. 
For a given $\ga \in \fC_m(n)$, we denote by  $\Sigma_\ga$  the set of bijections from $\set{m}$ to the  components of $\gamma$. We interpret $\sigma \in \Sigma_\ga$ as an ordering of the 
components and write  $[i]_\ga \prec_\sigma [j]_\ga$ if $\sigma^{-1} ([i]_\gamma) < \sigma^{-1}([j]_\gamma)$. 
Given an ordering $\sigma \in \Sigma_\ga$ over the components of a contraction $\ga \in \fC(n)$, we define the sets
\[
	\CD_{\ga, \sigma} := \Big\{ ( (s_1,x_1), \ldots, (s_n,x_n) ) \in \CD_\ga:\, \; s_i < s_j \text{ whenever } [i]_\ga \prec_\sigma [j]_\ga \Big\}.
\]
Now \eqref{eq:expansion2} becomes
\begin{equation}\label{eq:expansion}
	\int_{\CD_\eps^{ n}} F \, \d \bM^{ n}_{\eps} = \sum_{\ga \in \fC(n)} \sum_{\sigma \in \Sigma_\ga} \int_{\CD_{\ga, \sigma}} F \, \d \bM^{ n}_{\eps}.
\end{equation}

\subsubsection{A recursive representation}
\label{sec:representation}

We introduce the notation 
\begin{equation}\label{eq:integral_contraction}
	( \CI^\eps_{\ga, \sigma} F )_t := \int_{\CD_{\ga, \sigma} \cap \CD^{ n}_{\eps, t}} F \, \d \bM^{ n}_{\eps}.
\end{equation}
The aim of the next sections is to derive an estimate on moments of $( \CI^\eps_{\ga, \sigma} F )_t$ and this will be done recursively. To this end we introduce some more notation:
 first, for  given $\ga \in \fC_m(n)$ and $\sigma \in \Sigma_\ga$ we define the function $F^{\ga,\sigma} \colon \CD_\eps^m \to \R$ by
\begin{equation}\label{eq:F_gamma}
	F^{\ga,\sigma} \big(z_1, \ldots, z_m) := F(\bar{z}_1, \ldots , \bar{z}_n), 
\end{equation}
where $\bar{z} = (\bar{z}_1, \ldots, \bar{z}_n) \in \CD_\eps^{ n}$ is defined by 
\[
\bar{z}_i = z_j \qquad \Longleftrightarrow  \qquad  i  \text{ is a member of the $j$-th equivalence class according to $\sigma$}.
\]
Furthermore, for any $n$ we define the measure $\Md{n}$ on $\CD_\eps$ as 
\begin{align}\label{def:Mn}
 \Md{n} = \imath^{*} \bM_\eps^{ n} \lfloor_{\CD^{n}_{\eps, \mathrm{diag}}},
\end{align}
where  $\CD^{n}_{\eps, \mathrm{diag}}$ denotes the \emph{full diagonal} $\CD^{n}_{\eps, \mathrm{diag}} := \{ (z_1, z_2, \ldots , z_n) \in \CD_\eps^{ n} \colon z_1 = z_2 = \cdots  = z_n \}$, $\bM_\eps^{ n} \lfloor_{\CD^{n}_{\eps, \mathrm{diag}}}$ is the restriction of the product measure $\bM_\eps^{ n}$ to $\CD^{n}_{\eps, \mathrm{diag}}$ and $\imath^{*}$ denotes the image measure under the identification $ \imath \colon \CD^{n}_{\eps, \mathrm{diag}}  \to \CD_\eps$ given by $\imath(z, z, \ldots, z) = z$.
With this notations in place, the following recursive formula follows immediately.

\begin{lemma}\label{lem:recForm}
For any $F \in \CD_{\eps}^{ n}  \to \R$ which is continuous in all time variables, for any contraction $\ga \in \fC_m(n)$ and any ordering $\sigma \in \Sigma_\ga$ we have
\begin{equation}\label{eq:recForm} 
	\big( \CI^\eps_{\gamma,\sigma} F \big)_t = \int_ {s_m < t}  \int_{s_{m-1} < s_m} \; \cdots \; \int_{s_1 < s_2} F^{\gamma, \sigma} (z_1, \ldots, z_m)   
	\,\d  \Md{|\gamma_{\sigma(1)}|}(z_1) \cdots \d \Md{|\gamma_{\sigma(m)}|}(z_m).
\end{equation}
The subscript $s_m < t$ in the first integral of \eqref{eq:recForm} is used as a shorthand for $\{z_m =  (s_m, x_m) \in \CD_\eps \colon s_m < t \}$, and similarly the subscripts  $s_i < s_{i+1}$ mean that the corresponding integrals are taken over  $\{ z_i = (s_i, x_i) \in \CD_\eps \colon s_i < s_{i+1}\}$.
\end{lemma}

\subsection{Analysis of the measure on the diagonals}

We analyse further the measures $\Md{n}$ defined in \eqref{def:Mn}. By definition, we have for $F \colon \CD_\eps \to \R$
\begin{equation*}
\int_{\CD_\eps}  F  \,\d \Md{n} = \int_{\CD^{n}_{\eps, \mathrm{diag}}} F(z_1)  \,  \d \bM_\eps^{ n}(z_1, z_2, \ldots, z_n). 
\end{equation*}
In this formula we may allow $F$ to be random which does not affect our computations. 
For $n\geq 2$ Assumption~\ref{a:Martingales}\eqref{it:dynamics}, equation \eqref{eq:martingales_dynamics}, implies that that only the jump parts of the martingales produce  non-trivial contributions. We get that 
\begin{align}
	\int_{s<t} F(s,x) \,  \d \Md{n}  (s,x) &= \sum_{x \in \Le} \eps^{nd} \sum_{0 \leq s < t}  F ( s, x ) \big( \Delta_{s} \M_{\eps}(x) \big)^{n} 
	\label{eq:Integral_gamma_bound_equality}
\end{align} 

In the case $n=2$ we use the bracket processes of the martingales (see Section~\ref{sec:Appendix2}) to write \eqref{eq:Integral_gamma_bound_equality} as
\begin{align}
\notag
	&\int_{s<t} F(s,x)\, \d \Md{2}  (s,x)
	 = \sum_{x \in \Le} \eps^{2d} \int_{s<t}  F ( s, x )\, \d [ \M_{\eps}(x) ]_s  \\
	 \notag
	&\qquad = \sum_{x \in \Le} \eps^{2d}  \int_{s<t}  F ( s, x )\, \d  \langle \M_{\eps}(x) \rangle_s  + \sum_{x \in \Le} \eps^{2d} \int_{s<t}  F ( s, x )\, \d \bigl([ \M_{\eps}(x) ]_s - \langle \M_{\eps}(x) \rangle_s \bigr)\\
	&\qquad = \sum_{x \in \Le} \eps^{d} \int_{s<t}  F ( s, x ) \C_{\eps}(s, x)\, \d s  + \sum_{x \in \Le} \eps^{2 d + \kone} \int_{s<t}  F ( s, x )  \, \d \Nm_{\eps}(s,x),
	\label{eq:Integral_gamma_bound_equality_n=2}
\end{align} 
where in the last equality we have used Assumption \ref{a:Martingales}\eqref{it:bracket} in the first term and the definition of the martingales \eqref{eq:renorm-martingale} in the second one. 

To bound \eqref{eq:Integral_gamma_bound_equality} for $n \geq 3$ we make crucial use of the Assumption~\ref{a:Martingales}\eqref{it:jumps2} that guarantees that the jumps $| \Delta_t \M( x)|$ are of fixed size $c \eps^{\kone}$. Then for $n$ odd we get
\begin{align}
\notag
	&\int_{s<t} F(s,x) \,  \d \Md{n}  (s,x) = \sum_{x \in \Le} \eps^{nd} \sum_{0 \leq s < t}  F ( s, x ) (c \eps^{\kone} )^{n-1} \Delta_{s}    \M_{\eps}(x) \\
	\notag
	&\qquad = c^{n-1} \eps^{(d + \kone) (n-1)} \sum_{x \in \Le} \eps^{d} \int_{0 \leq s < t}  F ( s, x )\,  \d  \M_{\eps}(s, x) \\
	\label{eq:oddContr}
	& \qquad\qquad \qquad  - c^{n-1} \eps^{(d + \kone) (n-2)} \sum_{x \in \Le} \eps^{d} \int_{0 \leq s < t}  F ( s, x ) \Cd_\eps( s, x )\, \d  s,
\end{align} 
where in the last identity we made use of Assumption~\ref{a:Martingales}\eqref{it:dynamics}. Similarly, when $n$ is even we get 
\begin{align}
\notag
	&\int_{s<t} F(s,x) \,  \d \Md{n}  (s,x) =  (c \eps^{\kone} )^{n-2} \sum_{x \in \Le} \eps^{nd}  \sum_{0 \leq s < t}  F ( s, x )\, \d [  \M_{\eps}(x)]_s \\
	\notag
	&\qquad = c^{n-2} \eps^{(d + \kone) (n-1)} \sum_{x \in \Le} \eps^{d}  \int_{0 \leq s<t}  F ( s, x )  \, \d \Nm_{\eps}(s, x) \\
	\label{eq:evenContr}
	&\qquad \qquad \qquad  + c^{n-2} \eps^{(d + \kone) (n-2)} \sum_{x \in \Le} \eps^{d}  \int_{0 \leq s<t}  F ( s, x ) \C_{\eps}(s, x)\, \d s,
\end{align} 
where we made use of the martingale \eqref{eq:renorm-martingale} and Assumption~\ref{a:Martingales}\eqref{it:bracket}. Remarkably, the equations \eqref{eq:oddContr} and \eqref{eq:evenContr} are of exactly the same structure and consequently, even and odd contractions can be bounded in the same way. 

\section{Moment bounds for iterated integrals}
\label{sec:moment-bounds-simple}

We aim to estimate integrals \eqref{eq:recForm}. This is done recursively and here we perform the recursive step by deriving an estimate on 
\begin{equation}\label{eq:integral-to-be-bounded}
\int_{\CD_\eps} G_t(t,x)\, \d \Md{n_0}(t,x) \;, 
\end{equation}
where 
\begin{equation}\label{eq:form-for-G}
G_t(s,x) = \int_ {s_m < t}  \int_{s_{m-1} < s_m} \; \cdots \; \int_{s_1 < s_2} F (s,x; z_1, \ldots, z_m) \;\d  \Md{n_1}(z_1) \cdots \d   \Md{n_m}(z_m) \;
\end{equation}
and  $n_0, n_1, \ldots, n_m \geq 1$.
Throughout this section we make the assumption that $F: \CD_{\eps}^{m+1}  \to \R$ is a deterministic function that is $\SC^1$ in each time variable. We note that the domain of integration in \eqref{eq:form-for-G} guarantees that for any fixed $s$ the function $t \mapsto G_t(s,x)$ is predictable.
 In the following  Section \ref{sec:OneFoldBounds} we bound the simple integral \eqref{eq:integral-to-be-bounded} in terms of $G$ and $F$. The resulting estimate is then used in a recursive argument 
 to bound the full iterated integral $\CI^\eps_{\gamma,\sigma} F$ in the subsequent Section \ref{sec:GeneralBounds}.
 
 To control the function $F$ we will use the following norm 
 \begin{equation}\label{eq:norm-for-F}
 \Vert F \Vert_{\SC_s^1(L^\infty_\eps)} := \Vert F \Vert_{L^\infty_\eps} + \Vert  \partial_s F \Vert_{L^\infty_\eps},
 \end{equation}
 where the subscript $s$ refers to the variable with respect to which the $\SC^1$-norm is computed. 

\subsection{Simple integrals}
\label{sec:OneFoldBounds}

We will need the following result.

\begin{lemma}\label{lem:Sobolev}
	Let $f : I \to \R$ be a $\SC^1$ function on a interval $I \subset \R$ with length $|I| > 0$. Then for any $p \geq 1$ 
	\begin{equation}\label{eq:SobTypeIneq}
		\sup_{t \in I} | f(t) |^p \leq  \frac{1}{|I|}\int_{I} | f(r) |^p \d r + p \biggl( \int_{I} | f(r) |^p \d r \biggr)^{\frac{p-1}{p}} \biggl( \int_{I} |  f'(r) |^p \d r \biggr)^{\frac{1}{p}}.
	\end{equation}
\end{lemma}

\begin{proof}
	For any fixed $t, t_0 \in I$ we can write $f(t)^p = f(t_0)^p + \int_{t_0}^t p f(r)^{p-1} f'(r) \d r$. Taking absolute values, then using the H\"{o}lder inequality and finally taking the supremum over $t$, we deduce 
	that
	\begin{equation*}
		\sup_{t \in I} | f(t) |^p \leq | f(t_0) |^p + p \biggl( \int_{I} | f(r) |^p \d r \biggr)^{\frac{p-1}{p}} \biggl( \int_{I} | f'(r) |^p \d r \biggr)^{\frac{1}{p}}.
	\end{equation*}
	We conclude by averaging the variable $t_0$ over the interval $I$.
\end{proof}

The following proposition provides moment bounds for a simple stochastic integral.

\begin{proposition}\label{prop:one-fold-bound} 
 Let $G_t : \CD_{\eps} \to \R$ be a possibly random function, such that the function $t \mapsto G_t(t,x)$ is predictable. Then for any $p \geq 2$ and $T \in [0,1]$ we have 
\begin{align}\label{eq:BoundIntegral}
		&\E_p \sup_{t \in [0, T]}  \Big| \int_{\CD_{\eps, t} } G_s(s,x)  \,\d \bM_\eps(s,x) \;  \Big|\\
		& \qquad \lesssim_{p} \biggl( \eps^{d} \sum_{y \in \Le} \int_{0}^T  \big( \E_p G_{r}(r, y) \big)^2  \d r \biggr)^{\frac{1}{2}} + \eps^{d + \kone}\, \E_p \sup_{(s,x) \in \CD_{\eps, T}} | G_s(s, x) |. \nonumber
\end{align}
If $G_t(s,x)$ is of the form \eqref{eq:form-for-G}, then we have 
	\begin{equation}\label{eq:BoundIntegral-error}
		\E_p \sup_{(s,x) \in \CD_{\eps, T}} | G_s(s, x) | \lesssim \eps^{-(d + (d + \kone) n) \frac{1}{p}} \Big\Vert \E_p \sup_{t \in [0, T]} | G_t | \Big\Vert_{L^\infty_\eps}^{\frac{p-1}{p}} \Vert F \Vert_{\SC_s^1(L^\infty_\eps)}^{\frac{1}{p}},
	\end{equation}
	where $n_1 + \cdots + n_m = n$ and where we use the norm \eqref{eq:norm-for-F}.
\end{proposition}

\begin{remark}
In our application, $\partial_s F$ will be badly behaved and in general blow up as negative power in $\eps$, similarly to the other exploding terms.  However, as $p$ can and will be chosen arbitrarily large, all can be absorbed in small pre-factor $\eps^{d+\kone}$ and therefore  these error terms are all harmless.
\end{remark}

\begin{proof}[of Proposition~\ref{prop:one-fold-bound}]
	Using the \BDG inequality (Proposition~\ref{prop:BDG}) and Assumption~\ref{a:Martingales}, we get
	\begin{align}
		&\E_p \sup_{t \in [0, T]}  \Big| \int_{\CD_{\eps, t} } G_s(s,x)  \,\d \bM_\eps(s,x) \;  \Big|\\
		 &\qquad \lesssim_{p} \E \Biggl[ \biggl( \sum_{y \in \Le} \eps^{2d} \int_{r=0}^T G_{r}(r, y)^2 \d \langle \M_{\eps}(y) \rangle_r \biggr)^{\frac{p}{2}} \Biggr]^{\frac{1}{p}} + \E \biggl[ \sup_{t \in [0, T]} | G_t(t,x)  \Delta_t \M_\eps(x) |^p \biggr]^{\frac{1}{p}} \nonumber 
		\\ & \qquad \lesssim_{p}  \E \Biggl[ \biggl( \sum_{y \in \Le} \eps^{d} \int_{0}^T G_{r}(r, y)^2 \d r \biggr)^{\frac{p}{2}} \Biggr]^{\frac{1}{p}} + \eps^{d + \kone} \E \biggl[ \sup_{(s,x) \in \CD_{\eps, T}} | G_s(s, x) |^p \biggr]^{\frac{1}{p}}  , \label{eq:BDGApplication}
	\end{align}
	where the proportionality constant in the last bound comes from Assumption~\ref{a:Martingales}\eqref{it:bracket} and \eqref{it:jumps2}. The first term on the right hand side of \eqref{eq:BDGApplication} can be controlled by the first term  in \eqref{eq:BoundIntegral} by an application of Minkowski's inequality. This yields the required bound \eqref{eq:BoundIntegral}.

	Now we sill prove \eqref{eq:BoundIntegral-error}. First, the supremum over the lattice points is replaced by a sum at the expense of a small negative power of $\eps$ and second, we use $\sup_{s \in [0,T]} |G_s(s,x)| \leq \sup_{s, t \in [0,T]} |G_t(s,x)|$ to arrive at 
	\begin{equation} \label{eq:summation_over_space_points}
		\E \biggl[ \sup_{(s,x) \in \CD_{\eps, T}} | G_s(s,x) |^p \biggr]^{\frac{1}{p}} \leq \eps^{-\frac{d}{p}} \Biggl( \eps^d \sum_{y \in \Le} \E \biggl[ \sup_{s, t \in [0,T]} | G_t(s, y) |^p \biggr] \Biggr)^{\frac{1}{p}}.
	\end{equation}
	We note that this bound holds because $\Le$ is a finite lattice. 
	We apply first \eqref{eq:SobTypeIneq}  to the supremum in the variable $s$ and then H\"older's inequality   to bound the right-hand side of  \eqref{eq:summation_over_space_points} by a constant times
	\begin{align*}
		&\eps^{-\frac{d}{p}} \Biggl( \eps^d \sum_{y \in \Le} \E  \sup_{t \in [0,T]} \Biggl[  \frac{1}{T}\int_{0}^T | G_t (s,y) |^p \d s \\
		& \hspace{4cm} + \Biggl( \int_0^T | G_t(s, y) |^p \d s \Biggr)^{\frac{p-1}{p}} \Biggl( \int_0^T | \partial_s G_t(s, y) |^p \d s \Biggr)^{\frac{1}{p}} \Biggr] \Biggr)^{\frac{1}{p}} \\
		&\lesssim\eps^{-\frac{d}{p}} \Biggl( \eps^d \sum_{y \in \Le} \frac{1}{T}\int_{0}^T \E \biggl[ \sup_{t \in [0,T]} | G_t (s,y) |^p \biggr] \d s \\ 
		&\quad + \Biggl( \eps^d \sum_{y \in \Le}  \int_0^T  \E \biggl[ \sup_{t \in [0,T]} | G_t(s, y) |^p\,  \biggr]  \d s \Biggr)^{\frac{p-1}{p}} \Biggl( \eps^d \sum_{y \in \Le} \int_0^T \E \biggl[  \sup_{t \in [0,T]} | \partial_s G_t(s, y) |^p\,  \biggr] \d s \Biggr)^{\frac{1}{p}} \Biggr)^{\frac{1}{p}}.
	\end{align*}
To bound the norms of the function $G_t$, we observe that
\[
 \eps^d \sum_{y \in \Le} \int_{0}^T \E \biggl[ \sup_{t \in [0,T]} | G_t (s,y) |^p \biggr] \d s \lesssim T \sup_{(s,y) \in \CD_{\eps, T}} \E \biggl[ \sup_{t \in [0,T]} | G_t (s,y) |^p \biggr],
\]
where we used that the grid $\Le$ is finite. Furthermore, the definition \eqref{eq:form-for-G} yields
\begin{equation}\label{eq:G-bound}
| \partial_s G_t(s, y) | \lesssim \Vert \partial_s F \Vert_{L^\infty_\eps} \sum_{x_1, \ldots, x_m \in \Le} \eps^{m d} \prod_{i = 1}^m \| \Md{n_i}(x_i) \|_{\TV((s_{i-1}, s_{i+1}))},
\end{equation}
with $s_0 = 0$ and $s_{m+1} = t$. Identities \eqref{eq:oddContr}/\eqref{eq:evenContr} and Lemma~\ref{lem:TV} allow to bound 
\begin{equation*}
\E_p \| \Md{n_i}(x_i) \|_{\TV((s_{i-1}, s_{i+1}))} \lesssim \eps^{-(\kone + d) n_i}.
\end{equation*}
The total variation norms are computed in \eqref{eq:G-bound} on non-intersecting intervals, and using respective conditioned expectations we can bound 
\begin{equation}\label{eq:bound-on-deriv}
\E_p \sup_{t \in [0,T]} | \partial_s G_t(s, y) | \lesssim \Vert \partial_s F \Vert_{L^\infty_\eps} \prod_{i = 1}^m \eps^{-(\kone + d) n_i} \lesssim \Vert \partial_s F \Vert_{L^\infty_\eps} \eps^{-(\kone + d) n},
\end{equation}
where $n_1 + \cdots + n_m = n$. Combining the preceding bounds, we conclude that \eqref{eq:summation_over_space_points} is estimated by a constant multiple of 
\begin{align} \nonumber
&\eps^{-\frac{d}{p}} \Biggl( \Big\Vert \E_p \sup_{t \in [0, T]} | G_t | \Big\Vert^p_{L^\infty_\eps} + \eps^{-(d + \kone) n} \Vert  \partial_s F \Vert_{L^\infty_\eps} \Big\Vert \E_p \sup_{t \in [0, T]} | G_t | \Big\Vert_{L^\infty_\eps}^{p-1} \Biggr)^{\frac{1}{p}} \\
&\qquad = \eps^{-\frac{d}{p}} \Big\Vert \E_p \sup_{t \in [0, T]} | G_t | \Big\Vert_{L^\infty_\eps}^{\frac{p-1}{p}} \Biggl( \Big\Vert \E_p \sup_{t \in [0, T]} | G_t | \Big\Vert_{L^\infty_\eps} + \eps^{-(d + \kone) n} \Vert  \partial_s F \Vert_{L^\infty_\eps} \Biggr)^{\frac{1}{p}}, \label{eq:intermediate-bound}
\end{align}
where we used our assumption $T \leq 1$. Similarly to \eqref{eq:bound-on-deriv} we get
\begin{equation*}
\E_p \sup_{t \in [0,T]} | G_t(s, y) | \lesssim \Vert F \Vert_{L^\infty_\eps} \prod_{i = 1}^m \eps^{-(\kone + d) n_i} \lesssim \Vert F \Vert_{L^\infty_\eps} \eps^{-(\kone + d) n},
\end{equation*}
and \eqref{eq:intermediate-bound} is estimated by a constant multiple of 
\begin{equation*}
\eps^{-\frac{d}{p} -(d + \kone) \frac{n}{p}} \Big\Vert \E_p \sup_{t \in [0, T]} | G_t | \Big\Vert_{L^\infty_\eps}^{\frac{p-1}{p}} \Vert F \Vert_{\SC_s^1(L^\infty_\eps)}^{\frac{1}{p}},
\end{equation*}
where we used the norm \eqref{eq:norm-for-F}. This gives the required bound \eqref{eq:BoundIntegral-error}.
\end{proof}

In connection with the representation developed in Section~\ref{sec:representation} we get the following. 

\begin{proposition} \label{prop:boundIntegralContr}
Let $G_t$ be as in \eqref{eq:BoundIntegral} and let $n_0 \geq 2$. Then for any $p \geq 2$ and $T \in [0,1]$ we have 
	\begin{align}\nonumber
		&\E_p \sup_{t \in [0, T]} \Big| \int_{\CD_{\eps, t}} G_s(s,x)  \,\d \Md{n_0}(s,x) \;  \Big| \lesssim_{p} \eps^{(d + \kone) (n_0-1)} \biggl( \eps^{d} \sum_{y \in \Le} \int_{0}^T  \big( \E_p G_{r}(r, y) \big)^2  \d r \biggr)^{\frac{1}{2}} \\
		&\hspace{2cm} + \eps^{(d + \kone) (n_0-2)} \biggl\Vert \E_p \sup_{t \in [0, T]} | G_t | \biggr\Vert_{L^1_\eps} + \eps^{(d + \kone) n_0} \E_p \sup_{(s,x) \in \CD_{\eps, T}} | G_s(s, x) |. \label{eq:boundIntegralContr}
	\end{align}
If $G_t(s,x)$ is of the form \eqref{eq:form-for-G}, then the last expectation is bounded by \eqref{eq:BoundIntegral-error}.
\end{proposition}

\begin{proof}
	Using \eqref{eq:oddContr}/\eqref{eq:evenContr} and Assumption~\ref{a:Martingales} we get
	\begin{align}\label{eq:simple-integral-proof}
	\Big| \int_{\CD_{\eps, t}} G_s(s,x)  \,\d \Md{n_0}(s,x) \;  \Big| &\lesssim \eps^{(d + \kone) (n_0-1)} \biggl| \int_{\CD_{\eps, t}}  G_s(s,x)\,  \d \Mnew_{\eps}(s, x) \biggr|\\
	& \qquad\qquad \qquad + \eps^{(d + \kone) (n_0-2)} \sum_{x \in \Le} \eps^{d} \int_{0 \leq s \leq t} |G_s(s,x)|\, \d  s \nonumber
	\end{align}
	a.s., where the martingale $\Mnew_{\eps}$ is either $\bM_\eps$ or $\ov{\bM}_\eps$, depending on whether $n_0$ is odd or even. In particular, the martingale satisfies Assumption~\ref{a:Martingales}. For the first term in \eqref{eq:simple-integral-proof} we use Proposition~\ref{prop:one-fold-bound} to get
	\begin{align*}
		\E_p \sup_{t \in [0,T]} \biggl| \int_{\CD_{\eps, t}}  G_s(s,x)\,  \d \Mnew_{\eps}(s, x) \biggr| &\lesssim \biggl( \eps^{d} \sum_{y \in \Le} \int_{0}^T  \big( \E_p G_{r}(r, y) \big)^2  \d r \biggr)^{\frac{1}{2}} \\
		&\qquad + \eps^{d + \kone}\, \E_p \sup_{(s,x) \in \CD_{\eps, T}} | G_s(s, x) |,
	\end{align*}
	and for the second term in \eqref{eq:simple-integral-proof} we have
	\begin{equation*}
		\E_p \biggl| \sum_{x \in \Le} \eps^{d} \int_{0 \leq s \leq t} |G_s(s,x)|\, \d  s \Biggr| \lesssim \biggl\Vert \E_p \sup_{t \in [0, T]} | G_t | \biggr\Vert_{L^1_\eps}.
	\end{equation*}
This gives the required bound \eqref{eq:boundIntegralContr}.
\end{proof}

\subsection{Bounds for general iterated integrals}
\label{sec:GeneralBounds}

Let $n, m \in \N$ with $m \leq n$ be fixed. For a contraction $\ga \in \fC_{m}(n)$ and a permutation $\sigma \in \Sigma_\ga$, we want to prove moment bounds for general multiple iterated integrals \eqref{eq:integral_contraction}. For this, we will define a norm on the function $F^{\ga,\sigma}$ from \eqref{eq:F_gamma}. 

For $m \geq 1$ we denote by $\SP_{\!m}$ the set of all functions $\bp : \set{m} \to \{1, 2, \infty\}$. Then for $\bp \in \SP_{\!1}$ we set $\Vert F^{\ga,\sigma} \Vert_{L^{\bp}_\eps} = \Vert F^{\ga,\sigma} \Vert_{L^{\bp(1)}_\eps}$, and for $m \geq 2$ and $\bp \in \SP_{\!m}$ we define the norm recursively 
\begin{equation}\label{eq:norms-def}
\Vert F^{\ga,\sigma} \Vert_{L^{\bp}_\eps} := \Big\Vert \bigl\Vert (F^{\ga,\sigma})^{(z_m)} \1_{s_m > s_{m-1}} \bigr\Vert_{L^{\bp \restriction_{\SP_{\!m-1}}}_\eps} \Big\Vert_{L^{\bp(m)}_\eps},
\end{equation}
where $\bp \restriction_{\set{m-1}}$ is the restriction of $\bp$ to $\set{m-1}$, the function $(F^{\ga,\sigma})^{(z_m)} \colon \CD_\eps^{m-1} \to \R$ is defined as $(F^{\ga,\sigma})^{(z_m)}(z_1, \ldots, z_{m-1}) = F^{\ga,\sigma}(z_1, \ldots, z_{m-1}, z_m)$, and the outer norm in \eqref{eq:norms-def} is computed with respect to the variable $z_m$. The indicator $\1_{s_m > s_{m-1}}$, with the convention $s_0 = 0$, is needed to respect the domain of integration in \eqref{eq:recForm}. 

For $\bp \in \SP_{\!m}$ we will also use the standard notation $\bp^{-1} \colon \{1, 2, \infty\} \to 2^{\set{m}}$ for the inverse function. For any $\ga \in \fC_{m}(n)$, we define the set of non-contracted variables
\begin{equation}\label{eq:Gamma-sets}
	\Gamma_{\!1}(\ga) := \{ i \in \set{m}:  |\ga_{i}| = 1 \}.
\end{equation}

The following is our main result for the general iterated integrals. 

\begin{theorem}\label{thm:IterIntMomentBound}
	Let martingales $( \M_{\eps} (t, x) )_{t \geq 0}$ satisfy Assumption~\ref{a:Martingales}, $\ga \in \fC_{m}(n)$, $\sigma \in \Sigma_\ga$ be any ordering for the components of $\ga$, and let $F \colon \CD_\eps^n \to \R$ be $\SC^1$ in each time variable. Then for every $p \geq 2$ and $T \in [0,1]$
	\begin{align}\label{eq:IterIntMomentBound}
		&\E_p \biggl[ \sup_{t \in [0, T]} \big| ( \CI^\eps_{\ga, \sigma} F )_t \big| \biggr] \\
		&\qquad \lesssim_{p} \sum_{\substack{\bp \in \SP_{\!m} :\\ \bp^{-1}(1) \cap \Gamma_{\!1}(\ga) = \emptyset}} \eps^{\alpha_\ga(\bp)} \Vert F^{\ga,\sigma} \Vert_{L^{\bp}_\eps}^{\beta_{\ga, p}(\bp)} \Biggl(\prod_{\substack{i \in \bp^{-1}(\infty) \setminus \Gamma_{\!1}(\ga) : \\ i \geq 2}} \eps^{-\kappa_{\ga, i}(\bp)} \Vert F^{\ga, \sigma} \Vert^{\beta_{\ga^{\smallgeq i}, p}(\bp^{\smallgeq i})}_{\SC_{s_{i}}^1(L^\infty_\eps)}\Biggr)^{\frac{1}{p}}, \nonumber
	\end{align}
	for some constants $\kappa_{\ga, i}(\bp) > 0$, where the function $F^{\ga,\sigma}$ is defined in \eqref{eq:F_gamma}, the powers are
	\begin{align}\label{eq:alpha-power}
		\alpha_\gamma(\bp) &:= (d+\kone) \biggl(\sum_{i = 1}^m |\ga_i| - 2 |\bp^{-1}(1)| - |\bp^{-1}(2)|\biggr), \\
		\beta_{\ga, p}(\bp) &:= \prod_{\substack{i \in \bp^{-1}(\infty) \setminus \Gamma_{\!1}(\ga) : \\ i \geq 2}} \left(\frac{p-1}{p}\right), \label{eq:beta-power}
	\end{align}
	the contraction $\ga^{\smallgeq i}$ has components $\ga^{\smallgeq i}_1, \ldots, \ga^{\smallgeq i}_{m-i+1}$ such that $\ga^{\smallgeq i}_j = \ga_{i + j-1}$, and the function $\bp^{\smallgeq i} \in \SP_{\!m-i+1}$ is defined as $\bp^{\smallgeq i}(j) = \bp(i + j-1)$.
\end{theorem}

\begin{remark}
Precise values of the constants $\kappa_{\ga, i}(\bp)$ in \eqref{eq:IterIntMomentBound} will not be important to us, although they may be obtained from the proof of Theorem~\ref{thm:IterIntMomentBound}. We show below that for $p$ sufficiently large the divergent factors $\eps^{-\kappa_{\ga, i}(\bp)/ p}$ are compensated by the multiplier $\eps^{\alpha_\ga(\bp)}$.
\end{remark}

\begin{remark}\label{rem:bound-like-Gaussian}
One can see that for any function $\bp$ in \eqref{eq:IterIntMomentBound} satisfying $\bp^{-1}(\infty) \neq \emptyset$ we have $\alpha_\gamma(\bp) \geq d + \kone$. Indeed, we can estimate $\sum_{i = 1}^m |\ga_i| \geq |\Gamma_{\!1}(\ga)| + 2 (m - |\Gamma_{\!1}(\ga)|) = 2m - |\Gamma_{\!1}(\ga)|$, because the sum over $|\Gamma_{\!1}(\ga)|$ components $\gamma_i$ of cardinality $1$ is exactly $|\Gamma_{\!1}(\ga)|$ and the sum over the other $m - |\Gamma_{\!1}(\ga)|$ component is at least $2 (m - |\Gamma_{\!1}(\ga)|)$. On the other hand, the assumptions on $\bp$ yield $2 |\bp^{-1}(1)| + |\bp^{-1}(2)| < m + |\bp^{-1}(1)| \leq 2 m - |\Gamma_{\!1}(\ga)|$. Thus, we conclude that $\alpha_\gamma(\bp) \geq d + \kone$.

Similarly, we have $\alpha_\gamma(\bp) \geq d + \kone$ if the contraction $\gamma$ has a component $\gamma_i$ such that $|\gamma_i| \geq 3$. Repeating the preceding computations, we get $\sum_{i = 1}^m |\ga_i| > 2m - |\Gamma_{\!1}(\ga)|$ and $2 |\bp^{-1}(1)| + |\bp^{-1}(2)| \leq 2 m - |\Gamma_{\!1}(\ga)|$, which yield the required estimate. 

Finally, one can see that we have $\alpha_\gamma(\bp) \geq d + \kone$ if there is a component $\gamma_i$ such that $i \in \bp^{-1}(2)$ and $|\gamma_i| \geq 2$.

Hence, if we have a sufficiently good control on the norms $\Vert F^{\ga,\sigma} \Vert_{L^{\bp}_\eps}$ and the product in \eqref{eq:IterIntMomentBound} over $i \in \bp^{-1}(\infty)$ is bounded by $C \eps^{-\kappa}$ for some $\kappa > 0$, we can take $p$ sufficiently large such that $\eps^{\alpha_\ga(\bp)} \eps^{-\kappa/p}$ vanishes as $\eps \to 0$. So that we expect that the only non-vanishing terms in the limit $\eps \to 0$ are those with correspond to contractions $\gamma$ with components of cardinalities at most $2$ and functions $\bp$ satisfying $\bp^{-1}(\infty) = \emptyset$ and $\bp^{-1}(2) = \Gamma_{\!1}(\ga)$. In these cases we have $\alpha_\gamma(\bp) = 0$ and $\beta_{\ga, p}(\bp) = 1$, and the estimate \eqref{eq:IterIntMomentBound} is similar to the bound for a multiple Wiener integral. 
\end{remark}

\begin{remark}
The bound \eqref{eq:IterIntMomentBound} is homogeneous with respect to $F$, i.e. multiplication of $F$ by a constant $\lambda > 0$ is equivalent to multiplication of the two sides of \eqref{eq:IterIntMomentBound} by $\lambda$. While the homogeneity of the left-hand side is trivial, seeing it for the right-hand side is more complicated. Let us consider the term in the sum in \eqref{eq:IterIntMomentBound} corresponding to a function $\bp$. 

If $\bp^{-1}(\infty) \setminus \Gamma_{\!1}(\ga) = \emptyset$ or $\bp^{-1}(\infty) \setminus \Gamma_{\!1}(\ga) = \{1\}$, then $\beta_{\ga, p}(\bp) = 1$ and the product in the parentheses in \eqref{eq:IterIntMomentBound} equals $1$. The respective term in the sum in \eqref{eq:IterIntMomentBound} equals $\eps^{\alpha_\ga(\bp)} \Vert F^{\ga,\sigma} \Vert_{L^{\bp}_\eps}$ and is homogeneous with respect to $F$.

Let us now look at the case when the set $\{ i \in \bp^{-1}(\infty) \setminus \Gamma_{\!1}(\ga) : i \geq 2\}$ is non-empty, and let $N$ be the magnitude of this set. Then the $\lambda$-multiplier corresponding to this term in the sum in \eqref{eq:IterIntMomentBound} equals 
\begin{equation*}
\lambda^{\beta_{\ga, p}(\bp)} \Biggl(\prod_{\substack{i \in \bp^{-1}(\infty) \setminus \Gamma_{\!1}(\ga) : \\ i \geq 2}} \lambda^{\beta_{\ga^{\smallgeq i}, p}(\bp^{\smallgeq i})}\Biggr)^{\frac{1}{p}}.
\end{equation*}
Furthermore, from the definition \eqref{eq:alpha-power} we conclude that the power of $\lambda$ may be written as
\begin{align*}
\left(\frac{p-1}{p}\right)^{N} + \frac{1}{p} \sum_{i = 1}^N \left(\frac{p-1}{p}\right)^{N- i} = 1.
\end{align*}
Hence, this term in the sum in \eqref{eq:IterIntMomentBound} is homogeneous with respect to $F$.
\end{remark}

\begin{proof}[of Theorem~\ref{thm:IterIntMomentBound}]
We prove this theorem by induction over the number $m$ of components in $\gamma$.

	The base of induction is $m = 1$, in which case $\gamma$ has only one component $\gamma_1$ such that $|\gamma_1| = n$. If $n = 1$ then the required bound \eqref{eq:IterIntMomentBound} is given in Proposition~\ref{prop:one-fold-bound}. (In this case, only two functions $\bp$ contribute to the sum in \eqref{eq:IterIntMomentBound}: $\bp(1) = 2$ and $\bp(1) = \infty$, which correspond to $\alpha_\gamma(\bp) = 0$ and $\alpha_\gamma(\bp) = d + \kone$ respectively. In both cases $\beta_{\ga, p}(\bp) = 1$. Then the two terms on the right-hand side of \eqref{eq:IterIntMomentBound} coincide with the two terms in \eqref{eq:BoundIntegral}.) If $n \geq 2$ then the bound \eqref{eq:IterIntMomentBound} is provided by Proposition~\ref{prop:boundIntegralContr}. (The three functions $\bp$ contributing to the sum in \eqref{eq:IterIntMomentBound} are $\bp(1) = 1$, $\bp(1) = 2$ and $\bp(1) = \infty$, which correspond to $\alpha_\gamma(\bp) = (d + \kone) (n-2)$, $\alpha_\gamma(\bp) = (d + \kone) (n-1)$ and $\alpha_\gamma(\bp) = (d + \kone) n$ respectively. In all cases $\beta_{\ga, p}(\bp) = 1$.)
	
	To make an inductive step, we assume that \eqref{eq:IterIntMomentBound} holds for all contractions having $m$ components and we will prove it for a contraction $\ga \in \fC_{m+1}(n)$. Lemma~\ref{lem:recForm} yields 
	\begin{equation}\label{eq:iterated-formula}
		( \CI^\eps_{\ga, \sigma} F )_t = \int_ {s_{m+1} < t} \Big( \CI^\eps_{\bar \ga, \bar \sigma} F^{( z_{m+1})} \Big)_{s_{m+1}-} \d   \Md{|\gamma_{m+1}|}(z_{m+1}),
	\end{equation}
	where $\bar \ga$ is the contraction obtained from $\ga$ by removing the $(m+1)^{\text{th}}$ component, $\bar \sigma$ is obtained by restricting $\sigma$ to $\set{m}$, and the function $F^{( z_{m+1})}$ is obtained from $F$ by setting all the variables, whose labels are in $(m+1)$-st equivalence class according to $\sigma$, to $z_{m+1}$. If $|\gamma_{m+1}| = 1$, we bound \eqref{eq:iterated-formula} using Proposition~\ref{prop:one-fold-bound}:
	\begin{align}
		\E_p \biggl[ \sup_{t \in [0, T]} | ( \CI^\eps_{\ga, \sigma} F )_t | \biggr] &\lesssim_p \Bigl\Vert \E_{p} \sup_{t \in [0,T]} \Bigl| \Bigl( \CI^\eps_{\bar \ga, \bar \sigma} F^{( z_{m+1})} \Bigr)_t \Bigr| \Bigr\Vert_{L^2_\eps} \nonumber \\
		&\qquad + \eps^{d+\kone} \Bigl\Vert \E_{p} \sup_{t \in [0, T]} \Bigl| \Bigl( \CI^\eps_{\bar \ga, \bar \sigma} F^{( z_{m+1})} \Bigr)_t \Bigr| \Bigr\Vert_{L^\infty_\eps}, \label{eq:M-integral-1}
	\end{align}
	 and if $|\gamma_{m+1}| \geq 2$, we use Proposition~\ref{prop:boundIntegralContr} and the estimate \eqref{eq:BoundIntegral-error}:
	\begin{align}
		&\E_p \biggl[ \sup_{t \in [0, T]} | ( \CI^\eps_{\ga, \sigma} F )_t | \biggr] \lesssim_p \eps^{(d+\kone)(|\ga_{m+1}| - 1)} \Bigl\Vert \E_{p} \sup_{t \in [0,T]} \Bigl| \Bigl( \CI^\eps_{\bar \ga, \bar \sigma} F^{( z_{m+1})} \Bigr)_t \Bigr| \Bigr\Vert_{L^2_\eps} \nonumber \\
		&\qquad + \eps^{(d+\kone)(|\ga_{m+1}| - 2)} \biggl\Vert \E_{p} \sup_{t \in [0,T]} \Bigl| \Bigl( \CI^\eps_{\bar \ga, \bar \sigma} F^{( z_{m+1})} \Bigr)_t \Bigr| \biggr\Vert_{L^1_\eps} \label{eq:M-integral-2} \\
		&\qquad + \eps^{(d+\kone)|\ga_{m+1}|} \eps^{-(d + (d + \kone) (n-|\ga_{m+1}|)) \frac{1}{p}} \Bigl\Vert \E_{p} \sup_{t \in [0, T]} \Bigl| \Bigl( \CI^\eps_{\bar \ga, \bar \sigma} F^{( z_{m+1})} \Bigr)_t \Bigr| \Bigr\Vert_{L^\infty_\eps}^{\frac{p-1}{p}} \Vert F^{\ga, \sigma} \Vert_{\SC_{s_{m+1}}^1(L^\infty_\eps)}^{\frac{1}{p}}. \nonumber
	\end{align}
	The function inside the expectations is itself an iterated integral of the function $F^{( z_{m+1})}$ with the contraction $\bar \gamma$ having $m$ components. We can use the induction hypothesis and the simple bound $\Vert (F^{\bar\ga,\bar\sigma})^{(z_{m+1})} \Vert_{\SC_{s_{i}}^1(L^\infty_\eps)} \leq \Vert F^{\ga, \sigma} \Vert_{\SC_{s_{i}}^1(L^\infty_\eps)}$ to get moment bounds for the expectation: 
	\begin{equation}\label{eq:induction}
		\E_p \biggl[ \sup_{t \in [0, T]} \Bigl| \Bigl( \CI^\eps_{\bar \ga, \bar \sigma} F^{( z_{m+1})}\Bigr)_t \Bigr| \biggr] \lesssim_{p} \sum_{\substack{\bar\bp \in \SP_{\!m} :\\ \bar\bp^{-1}(1) \cap \Gamma_{\!1}(\bar \ga) = \emptyset}} \eps^{\alpha_{\bar \ga}(\bar\bp)} \Vert (F^{\bar\ga,\bar\sigma})^{(z_{m+1})} \Vert^{\beta_{\bar \ga, p}(\bar\bp)}_{L^{\bar\bp}_\eps}\, E_p(\bar \ga, \bar \bp),
	\end{equation}
	where we denoted 
	\begin{equation}\label{eq:E-def}
	E_p(\bar \ga, \bar \bp) := \Biggl(\prod_{\substack{i \in \bar\bp^{-1}(\infty) \setminus \Gamma_{\!1}(\bar \ga) : \\ i \geq 2}} \eps^{-\kappa_{\bar \ga, i}(\bar \bp)} \Vert F^{\ga, \sigma} \Vert^{\beta_{\bar\ga_{\smallgeq i}, p}(\bar \bp_{\smallgeq i})}_{\SC_{s_{i}}^1(L^\infty_\eps)}\Biggr)^{\frac{1}{p}},
	\end{equation}
	for some constants $\kappa_{\bar \ga, i}(\bar \bp) > 0$. 
Then we use the preceding bound in \eqref{eq:M-integral-1} to get
	\begin{align} \label{eq:M-integral-3}
		\E_p \biggl[ \sup_{t \in [0, T]} | ( \CI^\eps_{\ga, \sigma} F )_t | \biggr] &\lesssim_p \sum_{\substack{\bar\bp \in \SP_{\!m} :\\ \bar\bp^{-1}(1) \cap \Gamma_{\!1}(\bar \ga) = \emptyset}} \eps^{\alpha_{\bar \ga}(\bar\bp)} \Bigl\Vert \Vert (F^{\bar\ga,\bar\sigma})^{(z_{m+1})} \Vert^{\beta_{\bar \ga, p}(\bar\bp)}_{L^{\bar\bp}_\eps} \Bigr\Vert_{L^2_\eps} E_p(\bar \ga, \bar \bp)  \\
		&\qquad + \sum_{\substack{\bar\bp \in \SP_{\!m} :\\ \bar\bp^{-1}(1) \cap \Gamma_{\!1}(\bar \ga) = \emptyset}} \eps^{\alpha_{\bar \ga}(\bar\bp) + d+\kone} \Bigl\Vert \Vert (F^{\bar\ga,\bar\sigma})^{(z_{m+1})} \Vert^{\beta_{\bar \ga, p}(\bar\bp)}_{L^{\bar\bp}_\eps} \Bigr\Vert_{L^\infty_\eps} E_p(\bar \ga, \bar \bp). \nonumber
	\end{align}
	We have $\Bigl\Vert \Vert (F^{\bar\ga,\bar\sigma})^{(z_{m+1})} \Vert^{\beta_{\bar \ga, p}(\bar\bp)}_{L^{\bar\bp}_\eps} \Bigr\Vert_{L^\infty_\eps} \leq \Bigl\Vert \Vert (F^{\bar\ga,\bar\sigma})^{(z_{m+1})} \Vert_{L^{\bar\bp}_\eps} \Bigr\Vert_{L^\infty_\eps}^{\beta_{\bar \ga, p}(\bar\bp)}$. Moreover, we have $\beta_{\bar \ga, p}(\bar\bp) \leq 1$ and Jensen's inequality yields $\Bigl\Vert \Vert (F^{\bar\ga,\bar\sigma})^{(z_{m+1})} \Vert^{\beta_{\bar \ga, p}(\bar\bp)}_{L^{\bar\bp}_\eps} \Bigr\Vert_{L^2_\eps} \lesssim \Bigl\Vert \Vert (F^{\bar\ga,\bar\sigma})^{(z_{m+1})} \Vert_{L^{\bar\bp}_\eps} \Bigr\Vert_{L^2_\eps}^{\beta_{\bar \ga, p}(\bar\bp)}$. We introduce new functions $\bp \in \SP_{\!m+1}$, such that $\bp(i) = \bar \bp(i)$ for $i \in \set{m}$, and $\bp(m+1) = 2$ and $\bp(m+1) = \infty$ in the two sums in \eqref{eq:M-integral-3} respectively. Then \eqref{eq:alpha-power} and $|\ga_{m+1}| = 1$ imply that the powers of $\eps$ in \eqref{eq:M-integral-3} are exactly $\alpha_{\ga}(\bp)$. Furthermore, \eqref{eq:beta-power} yields $\beta_{\bar \ga, p}(\bar\bp) = \beta_{\ga, p}(\bp)$ and \eqref{eq:E-def} yields $E_p(\bar \ga, \bar \bp) = E_p(\ga, \bp)$. Hence, recalling the definition \eqref{eq:norms-def} we get the required bound \eqref{eq:IterIntMomentBound} for $m+1$. 
	
	Now, we use the bound \eqref{eq:induction} in \eqref{eq:M-integral-2} and get 
	\begin{align} \nonumber
		&\E_p \biggl[ \sup_{t \in [0, T]} | ( \CI^\eps_{\ga, \sigma} F )_t | \biggr] \\
		&\quad \lesssim_p \sum_{\substack{\bar\bp \in \SP_{\!m} :\\ \bar\bp^{-1}(1) \cap \Gamma_{\!1}(\bar \ga) = \emptyset}} \eps^{\alpha_{\bar \ga}(\bar\bp) + (d+\kone)(|\ga_{m+1}| - 1)} \Bigl\Vert \Vert (F^{\bar\ga,\bar\sigma})^{(z_{m+1})} \Vert^{\beta_{\bar \ga, p}(\bar\bp)}_{L^{\bar\bp}_\eps} \Bigr\Vert_{L^2_\eps} E_p(\bar \ga, \bar \bp) \nonumber \\
		&\qquad + \sum_{\substack{\bar\bp \in \SP_{\!m} :\\ \bar\bp^{-1}(1) \cap \Gamma_{\!1}(\bar \ga) = \emptyset}} \eps^{\alpha_{\bar \ga}(\bar\bp) + (d+\kone)(|\ga_{m+1}| - 2)} \Bigl\Vert \Vert (F^{\bar\ga,\bar\sigma})^{(z_{m+1})} \Vert^{\beta_{\bar \ga, p}(\bar\bp)}_{L^{\bar\bp}_\eps} \Bigr\Vert_{L^1_\eps} E_p(\bar \ga, \bar \bp) \label{eq:M-integral-4} \\
		&\qquad + \sum_{\substack{\bar\bp \in \SP_{\!m} :\\ \bar\bp^{-1}(1) \cap \Gamma_{\!1}(\bar \ga) = \emptyset}} \eps^{\alpha_{\bar \ga}(\bar\bp) \frac{p-1}{p} + (d+\kone)|\ga_{m+1}| -(d + (d + \kone) (n-|\ga_{m+1}|)) \frac{1}{p}} \Bigl\Vert \Vert (F^{\bar\ga,\bar\sigma})^{(z_{m+1})} \Vert^{\beta_{\bar \ga, p}(\bar\bp)}_{L^{\bar\bp}_\eps} \Bigr\Vert_{L^\infty_\eps}^{\frac{p-1}{p}} \nonumber \\
		&\hspace{7cm} \times E_p(\bar \ga, \bar \bp)^{\frac{p-1}{p}} \Vert F^{\ga, \sigma} \Vert_{\SC_{s_{m+1}}^1(L^\infty_\eps)}^{\frac{1}{p}}, \nonumber
	\end{align}
	where in the last line we used subadditivity of the function $x \mapsto x^{\frac{p-1}{p}}$ for $x \geq 0$. As above, we estimate the norms by moving the power $\beta_{\bar \ga, p}(\bar\bp)$ to the outer norms. Furthermore, we introduce functions $\bp \in \SP_{\!m+1}$, such that $\bp(i) = \bar \bp(i)$ for $i \in \set{m}$, and $\bp(m+1) = 2$, $\bp(m+1) = 1$ and $\bp(m+1) = \infty$ in the three sums respectively. Then the powers of $\eps$ in the first and second sums in \eqref{eq:M-integral-4} equal $\alpha_{\ga}(\bp)$. Moreover, we have $\beta_{\bar \ga, p}(\bar\bp) = \beta_{\ga, p}(\bp)$ and $E_p(\bar \ga, \bar \bp) = E_p(\ga, \bp)$ in these sums. The last sum in \eqref{eq:M-integral-4} is more complicated. The power of $\eps$ equals $\alpha_{\ga}(\bp) - \kappa_{\ga, m+1}(\bp) \frac{1}{p}$ with $\kappa_{\ga, m+1}(\bp) = \alpha_{\bar \ga}(\bar \bp) + d + (d + \kone) (n- |\ga_{m+1}|) > 0$. Furthermore, $\beta_{\bar \ga, p}(\bar\bp) \frac{p-1}{p} = \beta_{\ga, p}(\bp)$. Setting $\kappa_{\ga, i}(\bp) = \kappa_{\bar\ga, i}(\bar\bp) \frac{p-1}{p}$ for $i \leq m$, we get 
	\begin{align*}
	&\eps^{-\kappa_{\ga, m+1}(\bp) \frac{1}{p}} E_p(\bar \ga, \bar \bp)^{\frac{p-1}{p}} \Vert F^{\ga, \sigma} \Vert_{\SC_{s_{m+1}}^1(L^\infty_\eps)}^{\frac{1}{p}} \\
	&\qquad =  \Biggl(\eps^{-\kappa_{\ga, m+1}(\bp)}\Vert F^{\ga, \sigma} \Vert_{\SC_{s_{m+1}}^1(L^\infty_\eps)} \Biggl(\prod_{\substack{i \in \bar\bp^{-1}(\infty) \setminus \Gamma_{\!1}(\bar \ga) : \\ i \geq 2}} \eps^{-\kappa_{\bar \ga, i}(\bar \bp)} \Vert F^{\ga, \sigma} \Vert^{\beta_{\bar\ga_{\smallgeq i}, p}(\bar \bp_{\smallgeq i})}_{\SC_{s_{i}}^1(L^\infty_\eps)}\Biggr)^{\frac{p-1}{p}}\Biggr)^{\frac{1}{p}} \\
	&\qquad =  \Biggl(\prod_{\substack{i \in \bp^{-1}(\infty) \setminus \Gamma_{\!1}(\ga) : \\ i \geq 2}} \eps^{-\kappa_{\ga, i}(\bp)} \Vert F^{\ga, \sigma} \Vert^{\beta_{\ga_{\smallgeq i}, p}(\bp_{\smallgeq i})}_{\SC_{s_{i}}^1(L^\infty_\eps)}\Biggr)^{\frac{1}{p}}, 
	\end{align*}
	where we used the identities $\beta_{\ga_{\smallgeq m+1}, p}(\bp_{\smallgeq m+1}) = 1$ and $\beta_{\ga_{\smallgeq i}, p}(\bp_{\smallgeq i}) = \beta_{\bar\ga_{\smallgeq i}, p}(\bar \bp_{\smallgeq i}) \frac{p-1}{p}$ which follow from the definitions. The preceding expression has the form \eqref{eq:E-def} for the contraction $\ga$ and the function $\bp$. Hence, recalling the definition \eqref{eq:norms-def}, we get from \eqref{eq:M-integral-4} the required bound \eqref{eq:IterIntMomentBound} for $m+1$. 
\end{proof}

\subsubsection{Renormalised iterated integrals} 
\label{sec:renormalized-integrals}
 
In the theory of regularity structures \cite{Regularity}, there is usually the need to renormalise stochastic objects. Introducing renormalised integrals against martingales is the goal of this section.

Let $(\M_{\eps}(t, x))_{t \geq 0}$ be martingales satisfying Assumption~\ref{a:Martingales}. Let a function $F \colon \CD_\eps^n \to \R$ be as in \eqref{eq:integral_contraction}, where the contraction $\ga$ has only one component $\ga_1$, such that $|\ga_1| = n$ is even, and let the permutation $\sigma$ be trivial. We define the integral
\begin{equation}\label{eq:RenormIntegralsLeb}
	( \CI^{\eps, \triangledown}_{\ga, \sigma} F )_t := \eps^{(d+\kone)(n - 2)} \int_{\CD_{\eps, t}} F^{\ga,\sigma}(z) \C_\eps(z) \, \d z,
\end{equation}
where we use the function $F^{\ga,\sigma}$ defined in \eqref{eq:F_gamma}. In this expression we integrate the contracted variable with respect to the bracket process \eqref{eq:quadr_var_formula} of the martingale. Furthermore, we define the renormalised integral 
\begin{equation}\label{eq:RenormIntegrals}
	( \CI^{\eps, \diamond}_{\ga, \sigma} F )_t := (\CI^\eps_{\ga, \sigma} F )_t - ( \CI^{\eps, \triangledown}_{\ga, \sigma} F )_t = \eps^{(d+\kone)(n - 1)} \int_{\CD_{\eps, t}} F^{\ga,\sigma}(s, x) \, \d \Nm_\eps(s, x),
\end{equation}
where the martingale $\Nm_\eps$ is defined in \eqref{eq:renorm-martingale}. As we will see in our application in Section~\ref{sec:applicationStochQuant}, we will consider the situation when the integral $\CI^\eps_{\ga, \sigma}$ diverges as $\eps \to 0$, and in order to control the latter we need to consider its renormalisation $\CI^{\eps, \diamond}_{\ga, \sigma}$ instead. If the noise was Gaussian, then the renormalising term $ \CI^{\eps, \triangledown}_{\ga, \sigma}$ would be deterministic. In our case, it is however a stochastic process.

In general, let $\ga \in \fC_{m}(n)$ with $1 \leq m \leq n$, and let $\sigma \in \Sigma_\ga$ be a permutation. Moreover, let us label components of $\ga$ using $\L \in \{ \triangledown, \diamond, \nil \}^{\set{m}}$, such that the label $\L(i)$ assigned to a component shows with respect to which process the variable is integrated. For $m=1$ we set
\[
	(\CI^{\eps, \L}_{\ga, \sigma} F)_t := \left\{ 
		\begin{aligned} & (\CI^{\eps, \triangledown}_{\ga_{1}, \sigma} F)_t &&\text{if } \L(1) = \triangledown, \\
		& (\CI^{\eps, \diamond}_{\ga_{1}, \sigma} F)_t  &&\text{if } \L(1) = \diamond\,, \\
		& (\CI^{\eps}_{\ga_{1}, \sigma} F)_t &&\text{if } \L(1) = \nil\,.
	\end{aligned} \right.
\]
Since we defined the integrals \eqref{eq:RenormIntegralsLeb} and \eqref{eq:RenormIntegrals} only for even $n$, we will always assume that $\L(i) = \nil$ for any $i$  such that $|\ga(i)|$ is odd. Then for $m \geq 2$ we define recursively 
\begin{equation}\label{eq:general_integral}
	(\CI^{\eps, \L}_{\ga, \sigma} F)_t := \left\{
		\begin{aligned} & \eps^{(d+\kone)(|\ga_m| - 2)} \int_ {\CD_{\eps, t}} \Bigl( \CI^{\eps, \L\restriction_{\set{m-1}}}_{\bar\ga, \bar \sigma} F^{(z_{m})} \Bigr)_{s_{m}-} \C_\eps(z_m)\, \d z_{m}  &&\text{if } \L ( m) = \triangledown, \\
		& \eps^{(d+\kone)(|\ga_m| - 1)} \int_ {\CD_{\eps, t}} \Bigl( \CI^{\eps, \L\restriction_{\set{m-1}}}_{\bar\ga, \bar \sigma} F^{(z_{m})} \Bigr)_{s_{m}-} \d \Nm_\eps(z_{m}) &&\text{if } \L(m) =\diamond\,, \\
		& \int_ {\CD_{\eps, t}} \Bigl( \CI^{\eps, \L\restriction_{\set{m-1}}}_{\bar\ga, \bar \sigma} F^{(z_{m})} \Bigr)_{s_{m}-} \, \d \Md{|\gamma_{m}|}(z_m) &&\text{if } \L(m) = \nil\,,
	\end{aligned} \right.
\end{equation}
where the function $F^{(z_{m})}$ is obtained from $F$ by setting the values of the variables in $\{ \sigma(i) : i \in \gamma_{m}\}$ to $z_{m}$, where $\bar\ga$ is the contraction that removes the $m^{\text{th}}$ component of $\ga$, where the labeling $\L$ is restricted to the indices in $\set{m-1}$, and where $\bar \sigma$ is the restriction of $\sigma$ to $\set{m-1}$.

For a labeling $\L$ it will be convenient to define the sets
\begin{equation*}
\L^{-1}(\triangledown) := \{ i: \L(i) = \triangledown \}, \qquad \L^{-1}(\diamond) := \{ i: \L(i) = \diamond \}, 
\end{equation*}
which contain the indices of the components labeled by ``$\triangledown$'' and ``$\diamond$'' respectively. Similarly to \eqref{eq:Gamma-sets} we define the set 
\begin{equation}\label{eq:Gamma-set-new}
	\Gamma(\ga) := \Gamma_1(\ga) \cup \L^{-1}(\diamond).
\end{equation}
of variables integrated with respect to martingales. 

The following result is an analogue of Theorem~\ref{thm:IterIntMomentBound} for the renormalised integrals. 

\begin{theorem}\label{thm:integral-bound-renorm}
	In the setting of Theorem~\ref{thm:IterIntMomentBound}, let $\L$ be a labeling of the contraction $\gamma \in \fC_{m}(n)$. Then for every $p \geq 2$ and $T \in [0, 1]$
	\begin{align}\label{eq:IterIntMomentBoundRenormalised}
		&\E_p \biggl[ \sup_{t \in [0, T]} \big| ( \CI^{\eps, \L}_{\ga, \sigma} F )_t \big| \biggr] \\
		&\qquad \lesssim_{p} \sum_{\substack{\bp \in \SP_{\!m} :\\ \bp^{-1}(1) \cap \Gamma(\ga) = \emptyset, \\ \L^{-1}(\triangledown) \subset \bp^{-1}(1)}} \eps^{\alpha_\ga(\bp)} \Vert F^{\ga,\sigma} \Vert_{L^{\bp}_\eps}^{\beta_{\ga, p}(\bp)} \Biggl(\prod_{\substack{i \in \bp^{-1}(\infty) \setminus \Gamma_{\!1}(\ga) : \\ i \geq 2}} \eps^{-\kappa_{\ga, i}(\bp)} \Vert F^{\ga, \sigma} \Vert^{\beta_{\ga^{\smallgeq i}, p}(\bp^{\smallgeq i})}_{\SC_{s_{i}}^1(L^\infty_\eps)}\Biggr)^{\frac{1}{p}}, \nonumber
	\end{align}
	for some constants $\kappa_{\ga, i}(\bp) > 0$, where $\alpha_\gamma(\bp)$ and $\beta_{\ga, p}(\bp)$ are defined in \eqref{eq:alpha-power} and \eqref{eq:beta-power}.
\end{theorem}

\begin{proof}
The proof is analogous to the proof of Theorem~\ref{thm:IterIntMomentBound}, where we use the recursive definition \eqref{eq:general_integral} and the fact that the martingales $\Nm_\eps$ satisfy Assumption~\ref{a:Martingales}. The restriction $\L^{-1}(\triangledown) \subset \bp^{-1}(1)$ in the sum in \eqref{eq:IterIntMomentBoundRenormalised} follows from the uniform bound on the integral with the label $\triangledown$.
\end{proof}

\begin{remark}\label{rem:bound-like-Gaussian-2}
The same argument as in Remark~\ref{rem:bound-like-Gaussian} implies that as $\eps \to 0$ the non-vanishing expectations \eqref{eq:IterIntMomentBoundRenormalised} are those with contractions $\gamma$ having components of cardinalities at most $2$ and functions $\bp$ satisfying $\bp^{-1}(\infty) = \emptyset$ and $\bp^{-1}(2) = \Gamma(\ga)$.
\end{remark}

\section{Kernels given by generalised convolutions}
\label{sec:GeneralizedConvolutions}

In this section we prove moment bounds for the iterated integrals \eqref{eq:integral}, when the function $F$ is given by convolutions of singular kernels (similar to the one introduced in \cite[Appendix~A]{HQ18}). This type of kernels appears in canonical lifts of random noises in the theory of regularity structures. However, the result presented in this section is different from \cite{HQ18} because of two reasons: first, our noise is non-Gaussian, and second, we prove bounds on the stochastic integrals rather than on deterministic objects which appear after Wick contractions of Gaussian noises. Moment bounds for stochastic integrals driven by a general stationary non-Gaussian noise were proved in \cite{ChandraShen}. In the latter work, the authors generalised the framework of \cite[Appendix~A]{HQ18} which allowed them to deal with more general contractions of noises. In our setting, we need to use Theorem~\ref{thm:integral-bound-renorm}, which requires estimating more complicated norms of the functions, in contrast to the $L^2$ norms when the noise is Gaussian. Since we adjust the ideas of \cite[Appendix~A]{HQ18} to our framework, we equip our results and definitions with references to their analogues from this article.

We will work in the space $\R^{d+1}$ with the parabolic scaling $\s = (2, 1, \ldots, 1)$, where the first coordinate is time and the other $d$ coordinates are spatial. We denote $| \s | := 2 + d$, and $\Vert z \Vert_{\mathfrak{s}} := |t|^{1 / 2} + |x|$ for any $z = (t,x) \in \R^{d + 1}$, and $\lambda^{\s} z := (\lambda^{2} t, \lambda x)$. For a multi-index $k = (k_0, \ldots, k_d) \in \N_0^{d+1}$ we define $| k |_{\s} := 2 k_0 + \sum_{i=1}^d k_i$. Then we denote by $\SC_\s^r$ the space of function on $\R^{d+1}$ with bounded mixed derivatives of the scaled order not exceeding $r$.

It will be convenient to consider processes $\M_{\eps}(t, x)$ defined on the whole time line $\R$. For this, we denote by $\widetilde{\M}_{\eps}(t, x)$ an independent copy of $\M_{\eps}(t, x)$ and define  
\begin{equation}\label{eq:martingale-extension}
\M_{\eps}(t, x) := 
\begin{cases}
\M_{\eps}(t, x) &\text{for}~t \geq 0,\\
\widetilde{\M}_{\eps}(-t, x) &\text{for}~t < 0, 
\end{cases}
\end{equation}
for all $t \in \R$. Then the stochastic integral \eqref{eq:one-fold_integral} can be naturally extended as
\begin{equation}\label{eq:integral-wrt-extension}
	\int_{\R \times \Lambda_\eps} F(z)\, \d \bM_{\eps}(z) = \eps^d \sum_{x \in \Le} \int_{s = 0}^\infty F(s,x)\, \d \M_{\eps}(s, x) + \eps^d \sum_{x \in \Le} \int_{s = 0}^\infty F(-s,x)\, \d \widetilde{\M}_{\eps}(s, x).
\end{equation}
The multiple integrals developed in Section~\ref{sec:iterated-integrals} can be then naturally extended to whole $\R$ in the time variable. With a little ambiguity we will use the notation as in Section~\ref{sec:iterated-integrals} for the integral defined with respect to $\bM_{\eps}$ on $\R \times \Le$. 

Following the idea of \cite[Appendix~A]{HQ18}, it will be convenient to describe generalised convolutions using labelled graphs. More precisely, we consider a finite directed {\it graph} $\CCG = (\CCV,\CCE)$ with a set of vertices $\CCV$ and with edges $e \in \CCE$ labelled by pairs $(a_e, r_e) \in \R_+ \times \Z$. We assume that the graph is {\it weakly connected} and {\it loopless}, i.e. every vertex has either an outgoing or incoming edge, and there are no edges from a vertex to itself. We require $\CCG$ to contain a distinguished vertex $\STAR \in \CCV$, connected by an outgoing edge with exactly one other vertex, denoted by $v_\star^\uparrow \in \CCV \setminus \{\STAR\}$. We also allow $\STAR$ to have incoming edges, which by the loopless assumption above cannot come from $v_\star^\uparrow$. Finally, we assume that the graph contains a set $\CCV_{\!\var}$ of distinguished vertices, which can be empty and which satisfies $\STAR \notin \CCV_{\!\var}$, and if it is non-empty, then it has only outgoing edges (``$\var$'' stands for ``variables'' because these vertices correspond to the variables integrated in the stochastic integral). This implies that there are no edges connecting two vertices from $\CCV_{\!\var}$. In Figure~\ref{fig:graph} we provide an example of such graph $\CCG$, where we omit labels and use various decorations for nodes and edges.

We define the set $\CCV_{\!\bar\star} := \CCV \setminus \{\STAR\}$ and for a directed edge $e \in \CCE$ we write $e_+$ and $e_-$ for the two vertices such that $e = (e_-,e_+)$ is directed from $e_-$ to $e_+$. We make the following assumption on the labels of the edges.

\begin{assumption}\label{a:mainGraph}
	The described graph $\CCG$ has the following properties:
	\begin{enumerate}[topsep=2pt, itemsep=0pt]
		\item every edge $e$ containing $\STAR$ has $r_e = 0$;
		\item the edge $e = (\STAR, v_\star^\uparrow)$ has the label $(a_e, r_e) = (0,0)$;
		\item at most one edge with $r_e > 0$ may be incident to the same vertex;
		\item if there are two vertices $e_-$ and $e_+$ such that the edge $e=(e_-, e_+)$ has $r_e < 0$, then $e_-$ and $e_+$ have no other incident edge.
	\end{enumerate}
\end{assumption}

\begin{figure}[h] \centering
		\begin{tikzpicture} 
			[dot/.style={circle,fill=black!90,inner sep=0pt, minimum size=1.5mm},
				empty/.style={inner sep=0pt, minimum size=1.5mm},
				spine/.style={very thick},
				edge/.style={thick}]
			\begin{scope}
				\node[rectangle,draw=blue!50,rounded corners=3pt,minimum width=35mm, minimum height=5mm] at (0,2) {};

				\node at (0,0)  [root,label=below:$\STAR$] (zero) {};
				\node at (-1.5,1) [dot] (left) {};
				\node at (0.4,1) {$v_{\star}^\uparrow$};
				\node at (0,1) [dot] (up) {};
				\node at (1.5,1) [dot] (right) {};
				\draw[->,spine, color=testcolor] (zero) to node [sloped,below] {} (up);
				\draw[->,edge] (right) to node [sloped,below] {} (zero);
				\draw[->,edge] (left) to node [sloped,below] {} (up);

				\node at (-1.5,2) [var] (up_left) {};
				\node at (-2.1,2) {$\CCV_{\!\var}$};
				\node at (0,2) [var] (up_up) {};
				\node at (1.5,2) [var] (up_right) {};

				\draw[->,edge] (up_left) to node [sloped,below] {} (left);
				\draw[->,edge] (up_left) to node [sloped,below] {} (up);
				\draw[->,edge] (up_up) to node [sloped,below] {} (right);
				\draw[->,edge] (up_right) to node [sloped,below] {} (right);
			\end{scope}
		\end{tikzpicture}
		\caption{An example of the graph $\CCG$, where the green edge connects the distinguished vertices $\protect\STAR$ and $v_{\star}^{\uparrow}$. The white vertices are in $\CCV_{\!\var}$ and have only outgoing edges. The distinguished vertex $\protect\STAR$ has an incoming edge, but it cannot come from $v_{\star}^{\uparrow}$.}
		\label{fig:graph}
	\end{figure}
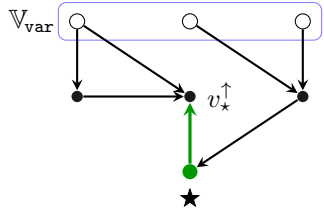

Let $|\CCV_{\!\var}| = n \geq 1$. Then we label the elements $\CCV_{\!\var}$ by $1, \ldots, n$, which gives a bijection between $\CCV_{\!\var}$ and $\set{n}$. Using the notation of Section~\ref{sec:contractions}, we write $\fC(\CCV_{\!\var})$ for the set $\fC(n)$ of all contractions on $\CCV_{\!\var}$. For a graph $\CCG = (\CCV,\CCE)$ and a contraction $\gamma \in \fC(\CCV_{\!\var})$ we define the multigraph (i.e. two vertices are allowed to be connected by multiple edges) $\CCG_\gamma = (\tilde \CCV, \tilde\CCE)$, with labels $(\tilde a_e, \tilde r_e)$, in the following way: the set of vertices $\tilde \CCV \subset \CCV$ is obtained from $\CCV$ by identifying those vertices from $\CCV_{\!\var}$ which belong to the same component in $\gamma$. We denote this ``identification'' by a surjective map $\i_\gamma : \CCV \to \tilde \CCV$. In particular, $\i_\gamma$ maps the vertices from $\CCV \setminus \CCV_{\!\var}$ (which includes $\STAR$) to themselves. We define $\tilde \CCV_{\!\var}$ to be the image of $\CCV_{\!\var}$ under the map $\i_\gamma$. Then we define the set of edges $\tilde \CCE$ on $\tilde \CCV$ to contain $\tilde e = (\i_\gamma(e_-), \i_\gamma(e_+))$ for all $e = (e_-, e_+) \in \CCE$, with the label $(a_{\tilde e}, r_{\tilde e}) = (a_e, r_e)$. In what follows, we call $\CCG_\gamma = (\tilde \CCV, \tilde \CCE)$\label{lab:contracted_graph} the {\it contracted (multi)graph} corresponding to $\CCG$ and $\gamma$. To consider renormalised integrals, we will use a labeling $\L$ of the components of the contraction $\gamma$, defined as in Section~\ref{sec:renormalized-integrals}. We by analogy with \eqref{eq:Gamma-set-new} we define the set of vertices  
\begin{equation}\label{eq:Gamma-sets-V}
\Gamma(\ga) := \{ v \in \tilde{\CCV}_{\!\var} :  |\i_\gamma^{-1}(v)| = 1 \} \cup \L^{-1}(\diamond),
\end{equation}
which correspond to the variables integrated with respect to martingales. Throughout this section we will use the shorthand $\Gamma = \Gamma(\ga)$ because the contraction $\ga$ will be always fixed.

A special case is $\CCV_{\!\var} = \emptyset$, in which all the definitions in the previous paragraph make sense for the identity contraction $\gamma$, and the contracted graph $\CCG_\gamma$ coincides with the original one $\CCG$. 

It will be useful to define a \emph{simple} (containing no multiedges) graph $(\hat \CCV, \hat \CCE)$, such that $\hat \CCV = \tilde \CCV$ and the unique edge $e \in \hat \CCE$ from $e_-$ to $e_+$ is obtained by contracting all edges $\tilde e$ from $e_-$ to $e_+$ in $\tilde \CCE$, with the label $(\hat a_e, r_e)$ of $e$ being the sum of the labels of all such parallel edges $\tilde e$. It follows from Assumption~\ref{a:mainGraph} that if there is more than one edge connecting $e_-$ to $e_+$ in $\CCG_\gamma$, then the value $r_e$ associated to the contracted edge is either $0$ (if all these edges $\tilde e \in \tilde \CCE$ have $r_{\tilde e} = 0$), or coincides with the only value $r_{\tilde e} > 0$, for $\tilde e \in \tilde \CCE$ connecting $e_-$ to $e_+$. We can have $r_e < 0$ only if there is a unique edge $\tilde e$ from $e_-$ to $e_+$ with $r_{\tilde e} < 0$.

For a subset $\bar\CCV \subset \CCV$ we define the outgoing edges $\CCE^\uparrow (\bar \CCV) := \{e \in  \CCE : e_- \in \bar\CCV\}$, incoming edges $ \CCE^\downarrow (\bar \CCV) := \{ e\in  \CCE : e_+ \in \bar\CCV\}$, internal edges $ \CCE_0 (\bar \CCV) := \{e \in  \CCE : e_{\pm} \in \bar \CCV\}$, and incident edges $ \CCE (\bar \CCV) := \{e \in  \CCE : e_{-} \in \bar \CCV \text{ or } e_{+} \in \bar \CCV\}$. If $\bar \CCV =  \CCV$, we simply write $ \CCE^\uparrow$, $ \CCE^\downarrow$, etc. Furthermore, we define the sets $\CCE_+ (\bar \CCV) := \{e \in  \CCE(\bar \CCV) :  r_e > 0\}$, $ \CCE_- (\bar \CCV) := \{e \in  \CCE(\bar \CCV) :  r_e < 0\}$, $ \CCE_+^\uparrow :=  \CCE_+ \cap  \CCE^\uparrow$ and $ \CCE_+^\downarrow :=  \CCE_+ \cap  \CCE^\downarrow$. These sets, defined for the edges $\tilde \CCE$ and $\hat \CCE$, will have the respective decorations. 

Then we require the contracted graph to satisfy the following assumption, which we state for the simple graph $(\hat \CCV, \hat \CCE)$ defined above. 

\begin{assumption}\label{a:mainContraction}
	The graph $\CCG = (\CCV,\CCE)$ and the contraction $\gamma \in \fC(\CCV_{\!\var})$ are such that the graph $(\hat \CCV, \hat \CCE)$, defined above, has the following properties:
	\begin{enumerate}[topsep=2pt, itemsep=0ex]
		\item\label{it:mainContraction-1} for any edge $e \in \hat\CCE$ one has $\hat a_e + (r_e \wedge 0) < |\s|$;
		\item\label{it:mainContraction-2} for every subset $\bar \CCV \subset \hat \CCV_{\!\bar\star}$ of cardinality at least $3$ one has
		      \begin{equation*}
			      \sum_{e  \in \hat\CCE_0(\bar \CCV)} \hat a_e < \Bigl(2 |\bar \CCV| - |\bar \CCV \cap \Gamma| - 1 - \1_{\bar \CCV \cap \Gamma = \emptyset} \Bigr) \frac{|\s|}{2};
		      \end{equation*}
		\item for every subset $\bar \CCV \subset \hat \CCV$ containing $\STAR$ of cardinality at least $2$ one has
		      \begin{equation*}
			      \sum_{e \in \hat \CCE_0 (\bar \CCV)} \hat a_e + \sum_{e \in \hat \CCE^{\uparrow}_+ (\bar \CCV)} \bigl(\hat a_e + r_e - 1\bigr) - \sum_{e \in \hat \CCE^{\downarrow}_+ (\bar \CCV)} r_e < \Bigl(2 |\bar \CCV| - |\bar \CCV \cap \Gamma| \Bigr) \frac{|\s|}{2};
		      \end{equation*}
		\item for every non-empty subset $\bar \CCV \subset \hat \CCV_{\!\bar\star} \setminus \{ v_{\star}^\uparrow \}$ one has
		      \begin{equation*}
			      \sum_{e \in \hat \CCE(\bar \CCV) \setminus \hat \CCE^{\downarrow}_+(\bar \CCV)} \hat a_e + \sum_{e \in \hat \CCE^{\uparrow}_+(\bar \CCV)} r_e - \sum_{e \in \hat \CCE^\downarrow_+(\bar \CCV)}( r_e - 1) > \Bigl(2 |\bar \CCV| - |\bar \CCV \cap \Gamma| \Bigr) \frac{|\s|}{2}.
		      \end{equation*}
	\end{enumerate}
\end{assumption}

\begin{remark}
Assumption~\ref{a:mainContraction} coincides with Assumption~3.17 in \cite{ChandraShen} on an ``elementary graph'',  where the set of ``external vertices'' (see Definition~3.13 in \cite{ChandraShen}) is given in our case by the set $\Gamma$.
\end{remark}

\subsection{Kernels associated to the graph}

Given a graph $\CCG = (\CCV,\CCE)$ as above, to each edge we associate a kernel and each vertex corresponds to a variable in the domain $\R^{d+1}$. Then, for $e \in \CCE$, the values $a_e$ will describe the order of singularity of the kernel associated to the edge $e$. The value $r_e$ will describe the order of renormalisation of this kernel. For every vertex $v \notin \CCV_{\!\var} \cup \{\STAR\}$ we assume to be given a measure $\mu^\eps_v$ on $\R^{d+1}$ of the form
\begin{equation}\label{eq:point-measure}
	\mu^\eps (\d z) = \eps^d \sum_{y \in \eps \Z^d}  \delta(y - x)\, \d t\, \d x, 
\end{equation}
where $z = (t, x)$ with $t \in \R$ and $x \in \R^{d}$, and where $\delta$ is the Dirac delta function on $\R^d$. This measure counts the points in the space lattice and is the Lebesgue measure in time.
Notice that, as $\eps \to 0$, the measure $\mu^{\eps}_v$ converges in the weak-$*$ topology to the Lebesgue measure on $\R^{d+1}$.

For each edge of the graph we associate a kernel with the following properties.

\begin{assumption}\label{a:Kernels}
 For every $e \in \CCE$ we consider a smooth\footnote{In all our applications it is sufficient to have kernels sufficiently many times differentiable. For example, we can take them to be in $\SC^q_\s$ for $q = \sum_{e \in \CCE} (|r_e| + 2)$.} kernel $K^{\eps}_e \colon \R^{d+1} \to \R$, which can be written as $K^{\eps}_e(z) = \sum_{n = 0}^{N} K^{\eps, n}_{e}(z)$ for $N = - \lfloor \log_2 \epsEH \rfloor$ and for some $\epsEH \in [\eps,1]$, where the smooth functions $\{K^{\eps, n}_{e}\}_{0 \leq n \leq N}$ have the following properties:
	\begin{enumerate}[topsep=2pt, itemsep=0ex]
		\item\label{a:DKernels_support} the function $K^{\eps, n}_{e}(z)$ is supported in $C_1 2^{-n} \leq \|z\|_\s \leq C_2 2^{-n}$ for some $0 < C_1 < C_2$;
		\item\label{a:DKernels_bounds} for any $q \geq 0$ and for some $C > 0$, independent of $\eps$ and $\epsEH$, one has
		      \begin{equation}\label{eq:Kn_bound}
			      |D^k K^{\eps, n}_{e}(z)| \leq C 2^{n(a_e + |k|_\s)}\;,
		      \end{equation}
		      uniformly in $z$, $|k|_\s \leq q$ and $0 \leq n \leq N$;
		\item\label{a:DKernels_poly} if $r_e < 0$, then for all $0 \leq n \leq N$ and $|k|_\s < |r_e|$ one has 
		\begin{equation}\label{eq:Kn_killer}
			\int_{\R^{d+1}} z^k K^{\eps, n}_{e}(z) \mu_{e_+}^{\eps}(\d z) = 0.
		\end{equation}
	\end{enumerate}	
\end{assumption}

The necessity to introduce a new parameter $\epsEH$ can be seen in our application in Section~\ref{sec:applicationStochQuant}, where the mesh size of the grid is $\eps$ and the interaction range is defined on the scale $\eps^{\frac{3}{4}} \geq \eps$.

We see from \eqref{eq:Kn_bound} that the value $a_e$ characterizes the order of singularity of the kernel. Moreover, the value $r_e$, assigned to an edge $e \in \CCE$, describes a renormalisation of the singularity, which for positive and negative values are defined in different ways in the following section.

\begin{lemma}\label{lem:K_bound}
	 If Assumption~\ref{a:Kernels} is satisfied, then for any $q \geq 0$, the following quantity is bounded uniformly in $\epsEH \in [\eps,1]$ and $\eps \in (0, 1]$
	\begin{equation}\label{eq:K_bound}
					\| K^{\eps}_e \|^{(\epsEH)}_{a_e; q} := 
					\sup_{z \in \R^{d+1}} \sup_{|k|_\s < q} (\|z\|_\s + \epsEH)^{a_e + |k|_\s} |D^k K^{\eps}_{e}(z)|.
	\end{equation}
	The reverse statement if also true, i.e. if for a kernel $K^{\eps}_e$ the quantity \eqref{eq:K_bound} is bounded uniformly, then it has all the properties listen in Assumption~\ref{a:Kernels}. 
\end{lemma}

\begin{proof}
	The bound \eqref{eq:K_bound} is a direct consequence of Assumption~\ref{a:Kernels}\eqref{a:DKernels_support}-\eqref{a:DKernels_bounds}. The second part of the lemma follows by repeating the proof of \cite[Lemma~5.4]{HM18}.
\end{proof}

\subsubsection{Renormalisation}
\label{sec:renormalisation}

If $r_e \neq 0$, then the kernel corresponding to the edge $e$ requires renormalisation. For positive and negative values of $r_e$ the renormalisation is defined different. For $r_e > 0$ the renormalisation of the smooth kernel is required to get a sufficiently fast decay of the kernel at the origin. In the case $r_e < 0$ the renormalisation is required to make the kernel, with a very strong singularity at the origin, integrable. 

In the case $r_e > 0$, we define the renormalised kernel
\begin{equation}\label{e:defKhatn2}
	\hat K^{\eps}_e(z_{e_-}, z_{e_+}) := K^{\eps}_e(z_{e_+}-z_{e_-}) - \sum_{|k|_\s < r_e} {z_{e_+}^k \over k!} D^k K^{\eps}_e(-z_{e_-}),
\end{equation}
where the sum runs over all multi-indices $k \in \N_0^{d+1}$ such that $|k|_\s < r_e$. In the case $r_e = 0$ we simply define $\hat K^{\eps}_e(z_{e_-}, z_{e_+}) := K^{\eps}_e(z_{e_+}-z_{e_-})$. The positive renormalisation \eqref{e:defKhatn2} allows to define kernels, which have sufficiently fast polynomial decay at the diagonal $z_{e_-} = z_{e_+}$. This is the case when $K^{\eps}_e$ is smooth with uniformly bounded derivatives.

If $r_e < 0$, then for a smooth and compactly supported function $\phi$ on $\R^{d+1} \times \R^{d+1}$ we define the expansion
\begin{equation}\label{e:Taylor}
	(T_{r_e} \phi)(z_{e_-}, z_{e_+}) := \phi(z_{e_-}, z_{e_+}) - \sum_{|k|_\s < |r_e|} {(z_{e_+} - z_{e_-})^k \over k!} D_2^k \phi(z_{e_-}, z_{e_-}),
\end{equation}
where the sum runs over all multi-indices $k \in \N_0^{d+1}$ satisfying $|k|_\s < |r_e|$, and where $D_2^k$ is the multi-derivative in the second argument. Furthermore, we associate to $K^{\eps}_e$ the distribution
\begin{align}\label{e:defRen}
	\hat{K}^{\eps}_e(\phi) := \int_{\R^{d+1}} \int_{\R^{d+1}} \!K^{\eps}_e(z_{e_+} - z_{e_-}) (T_{r_e} \phi)(z_{e_-}, z_{e_+})\,\mu^{\eps}_{e_-}(\d z_{e_-}) \mu^{\eps}_{e_+}(\d z_{e_+}),
\end{align}
which is obtained from $K^{\eps}_e$ by subtracting delta-functions and their derivatives. Expression \eqref{e:defRen} is just another way to write the integral
\begin{equation*}
	\int_{\R^{d+1}} \int_{\R^{d+1}} \!K^{\eps}_e(z_{e_+} - z_{e_-}) \phi(z_{e_-}, z_{e_+})\,\mu^{\eps}_{e_-}(\d z_{e_-}) \mu^{\eps}_{e_+}(\d z_{e_+}),
\end{equation*}
since $\int_{\R^{d+1}} \!K^{\eps}_e(z) z^k\, \mu^{\eps}_{e_+}(\d z) = 0$ and the measure $\mu^{\eps}_{e_+}$ is translation invariant. 

\begin{example}
	In all of the applications that we have in mind, we deal with labels $r_e$ taking values $+1$, $0$ or $-1$. In Section~\ref{sec:applicationStochQuant}, for example, we have $r_e = 1$ only for the tree in Section \ref{sec:last_symbol}; we use negative renormalisation with $r_e = -1$ only for the tree which is dealt with in \eqref{eq:important_subtree}. All the other edges in the trees of Section~\ref{sec:applicationStochQuant} always have $r_e=0$.
	
	Clearly, when $r_e =0$, we have no transformation to do on the kernels. When $r_e = 1$, on the other hand, we have \[ \hat K^{\eps}_e(z_{e_-}, z_{e_+}) := K^{\eps}_e(z_{e_+}-z_{e_-}) - K^{\eps}_e(-z_{e_-}), \] while when $r_e = -1$, we get \[ \hat K^{\eps}_e(z_{e_-}, z_{e_+}) := K^{\eps}_e(z_{e_+}-z_{e_-}) - \left( \int_{\R^{d+1}} K^{\eps}_e(z_{e_+} - z_{e_-}) \mu^{\eps}_{e_+}(\d z_{e_+}) \right) \delta(z_{e_+}-z_{e_-}), \] where $\delta$ is the Dirac delta-function. Observe that positive renormalisation corresponds to subtracting the value of the kernel itself at the ``base'' point $z_{e_{-}}$, while negative renormalisation means removing singularities at the base point.
\end{example}

\subsubsection{A generalised convolution}

Let us fix a graph $\CCG = (\CCV,\CCE)$ as described above. Then for a smooth and compactly supported function $\phi : \R^{d+1} \to \R$, for $z = (t, x), \bar z = (\bar t, \bar x) \in \R^{d+1}$ and for $\lambda \in (0,1]$ we define its rescaling and recentering
\begin{equation}\label{eq:test-function}
	\phi^\lambda_{\bar z}(z) := \lambda^{-|\s|} \phi \bigl(\lambda^{-2}(t - \bar t), \lambda^{-1}(x - \bar x)\bigr).
\end{equation}
For fixed $\bz^{\var} \in (\R^{d+1})^{\CCV_{\!\var}}$ we define the product measure on $\bz \in (\R^{d+1})^{\CCV_{\!\bar\star}}$
\begin{equation}\label{e:mu_def}
	\mu_{\CCV_{\!\bar\star}, \bz^{\var}}^{\eps}(\d \bz) := \Bigl(\prod_{v \in \CCV_{\!\bar\star} \setminus \CCV_{\!\var}} \mu_v^{\eps}(\d z_v)\Bigr) \Bigl(\prod_{w \in \CCV_{\!\var}} \delta(z_w - z^{\var}_w) \d z_w\Bigr),
\end{equation}
where again $\delta$ is the Dirac delta-function, and where $\bz = (z_v \in \R^{d+1} : v \in \CCV_{\!\bar\star})$ and $\bz^{\var} = (z^{\var}_v \in \R^{d+1} : v \in \CCV_{\!\var})$. In other words, the variable $z_v$, corresponding to the vertex $v \in \CCV_{\!\bar\star} \setminus \CCV_{\!\var}$, is integrated with respect to the measures $\mu_v^{\eps}$, and the variables corresponding to the vertices in $\CCV_{\!\var}$ are fixed to be equal to $\bz^{\var}$. These are the variables which we want to integrate with respect to martingales. Then we define the \emph{generalised convolution}
\begin{equation}\label{e:genconv}
	\CK^{\lambda, \eps}_\CCG(\bz^{\var}) := \int_{(\R^{d+1})^{\CCV_{\!\bar\star}}} \Bigl( \prod_{e\in \CCE}{\hat K}^{\eps}_e(z_{e_-}, z_{e_+})\Bigr) \phi^\lambda_0(z_{v_{\star}^\uparrow}) \mu_{\CCV_{\!\bar\star}, \bz^{\var}}^{\eps}(\d \bz).
\end{equation}
Since the kernels $K^{\eps}_e$ are smooth, our assumptions on the graph guarantee that the generalised convolution \eqref{e:genconv} is well-defined.

We fix any order of the elements in $\CCV_{\!\var}$ (which respectively fixed the order of the variables in $\CK^{\lambda, \eps}_\CCG(\bz^{\var})$) and we define 
\begin{equation} \label{eq:def_integrals_kernels}
	(\CI^{\eps, \L}_{\gamma} \CK^{\lambda, \eps}_\CCG )_t := \sum_{\sigma \in \Sigma_\ga} (\CI^{\eps, \L}_{\gamma, \sigma} \CK^{\lambda, \eps}_\CCG )_t,
\end{equation}
where the stochastic integral $\CI^{\eps, \L}_{\gamma, \sigma}$ is defined in Section~\ref{sec:renormalized-integrals} with respect to the fixed order of the variables. The following is our main result of this section.

\begin{theorem}\label{thm:convolutions}
	Let $\CCG=(\CCV,\CCE)$ be a graph with labels $\{a_e,r_e\}_{e\in \CCE}$ satisfying Assumption~\ref{a:mainGraph}, let $\gamma \in \fC_m(\CCV_{\!\var})$, with $1 \leq m \leq |\CCV_{\!\var}|$, be a contraction with a labeling $\L$ such that Assumption~\ref{a:mainContraction} is satisfied. Let the measures be defined as in \eqref{eq:point-measure} and let the kernels satisfy Assumption~\ref{a:Kernels}. 
	Let furthermore $\CI^{\eps, \L}_{\gamma}$ be a stochastic integral with respect to \cadlag martingales satisfying Assumption~\ref{a:Martingales}, let the set $\Gamma$ be defined in \eqref{eq:Gamma-sets-V}, and let 
	\begin{equation}\label{eq:alpha-gamma}
		\nu_\gamma := |\s| |\hat \CCV_{\!\bar\star} \setminus \{ \hat{v}^\uparrow_{\!\star} \}| - \frac{|\s|}{2} | \Gamma|  - \sum_{e \in \hat \CCE} \hat a_e < 0. 
	\end{equation}
	Then for any $p \geq 2$ and $\kappa > 0$ there is a constant $C$ for which the following bound holds
	\begin{align}\label{e:genconvBound}
		\biggl(\E \biggl[\sup_{t \in \R_+} \bigl| (\CI^{\eps, \L}_{\gamma} \CK^{\lambda, \eps}_\CCG)_t \bigr|^p\biggr]\biggr)^{\frac{1}{p}} &\le C (\lambda \vee \epsEH)^{\nu_\gamma} \sum_{\substack{\bp \in \SP_{\!m} :\, \bp^{-1}(1) \cap \Gamma = \emptyset, \\ \L^{-1}(\triangledown) \subset \bp^{-1}(1),\, \bp^{-1}(\infty) = \emptyset}} \eps^{\alpha_\ga(\bp)} \epsEH^{- \delta_{\gamma}(\bp)} \\
		&\qquad + C (\lambda \vee \epsEH)^{\nu_\gamma} \sum_{\substack{\bp \in \SP_{\!m} :\, \bp^{-1}(1) \cap \Gamma = \emptyset, \\ \L^{-1}(\triangledown) \subset \bp^{-1}(1),\, \bp^{-1}(\infty) \neq \emptyset}} \eps^{\alpha_\ga(\bp) - \kappa} \epsEH^{- \delta_{\gamma}(\bp)} \nonumber
	\end{align}
	uniformly in $\lambda \in (0,1]$, $\epsEH \in [\eps, 1]$ and $\eps \in (0,1]$, where the set of functions $\SP_{\!m}$ is defined in Section~\ref{sec:GeneralBounds}, the constant $\alpha_\gamma(\bp)$ is defined in \eqref{eq:alpha-power}, and
	\begin{equation}\label{eq:delta-gamma-p}
	\delta_{\gamma}(\bp) := \frac{|\s|}{2} \biggl( 2 |\bp^{-1}(\infty) \setminus \Gamma | + |\bp^{-1}(\infty) \cap \Gamma | + |\bp^{-1}(2) \setminus \Gamma |\biggr).
	\end{equation}
\end{theorem}

We prove this theorem in Section~\ref{sec:proof_bound}, and before that we need to get some preliminary results. 

\begin{remark}
From the proof of Theorem~\ref{thm:convolutions} we can see that there exists a value $q \geq 0$ and a compact set $\fK \subset \R^{d+1}$, such that the constant $C$ in \eqref{e:genconvBound} is proportional to
\begin{equation*}
\Bigl(\prod_{e \in \CCE} \| K^{\eps}_e \|^{(\epsEH)}_{a_e; q} \Bigr) \Bigl(\prod_{v \notin \CCV_{\!\var} \cup \{\star\}} \| \mu^{\eps}_{v} \|_{\TV(\fK)} \Bigr),
\end{equation*}
which by our assumptions is bounded uniformly in $\eps$ and $\epsEH$. For example, we can take a very rough value $q = \sum_{e \in \CCE} (|r_e| + 2)$.
\end{remark}

If we would like to consider a recentered test function $\phi^\lambda_{\bar z}(z)$, we need to shift respectively all the variables in the generalised convolution:
\begin{equation}\label{e:genconv_shift}
	\CK^{\lambda, \eps}_{\CCG, \bar z}(\bz^{\var}) := \int_{(\R^{d+1})^{\CCV_{\!\bar\star}}} \Bigl(\prod_{e\in \CCE}{\hat K}^{\eps}_e(z_{e_-} - \bar z, z_{e_+} - \bar z)\Bigr) \phi^\lambda_{\bar z}(z_{v_{\star}^\uparrow}) \mu_{\CCV_{\!\bar\star}, \bz^{\var}}^{\eps}(\d \bz).
\end{equation}
Then the following result can be proved as Theorem~\ref{thm:convolutions}, by changing the value of the variable $z_{\star}$ from $0$ to $\bar z$. Uniformity in $\bar z$ holds, because the norms of the kernels \eqref{e:wantedBound4} are independent of this variable. 

\begin{corollary}\label{cor:convolutions}
Under the assumptions of Theorem~\ref{thm:convolutions}, the bound \eqref{e:genconvBound} holds for the multiple integral $\CI^{\eps, \L}_\gamma \CK^{\lambda, \eps}_{\CCG, \bar z}$, locally uniformly in $\bar z$. 
\end{corollary}

Applying Minkowski inequality, we get from the definition \eqref{eq:def_integrals_kernels} the bound 
\begin{equation}\label{eq:integral-and-sigmas}
\E_p \sup_{t \in \R_+} \bigl| (\CI^{\eps, \L}_{\gamma} \CK^{\lambda, \eps}_\CCG)_t \bigr| \leq \sum_{\sigma \in \Sigma_\ga} \E_p\sup_{t \in \R_+} \bigl| (\CI^{\eps, \L}_{\gamma, \sigma} \CK^{\lambda, \eps}_\CCG)_t \bigr|.
\end{equation}
In the rest of the section we are going to prove the bound \eqref{e:genconvBound} for the integral $\CI^{\eps, \L}_{\gamma, \sigma}$ with a fixed $\sigma$. One can see from the proof, that this bound is independent of the order of the variables (although the order plays a role in some intermediate results like Lemma~\ref{lem:KTildeBound}), and the same bound \eqref{e:genconvBound} holds for every integral $\CI^{\eps, \L}_{\gamma, \sigma}$ in \eqref{eq:integral-and-sigmas}.

\subsection{Multiscale decomposition of the generalised convolution}
\label{sec:decomp}

Our aim is to write the kernels $K_e^{\eps}$ in the generalised convolution \eqref{e:genconv} as sums of localised functions. For the edge $(\star, v_{\star}^\uparrow)$, we view the test function $\phi^\lambda_0(z_{v_{\star}^\uparrow})$ in \eqref{e:genconv} as a new kernel $K_{(\star, v_{\star}^\uparrow)}(z_{v_{\star}^\uparrow})$, supported on $\| z_{v_{\star}^\uparrow} \|_\s \lesssim \lambda$ and satisfying $\| K_{(\star, {v_{\star}^\uparrow})} \|_{0; q} \lesssim \lambda^{-|\s|}$ (recall that this edge has the labels $a_e = r_e = 0$ in the graph).

Our next aim is to decompose the kernels in \eqref{e:genconv} into sums of localised functions. To this end, for $e \in \CCE$ with $r_e > 0$, we take any smooth functions $\psi^{(\eps, n)} : \R^{d+1} \to \R$, such that $\psi^{(\eps, n)}(z)$ is supported in $C_1 2^{- n} \leq \|z\|_\s \leq C_2 2^{-n}$ (where $C_1, C_2$ are from Assumption~\ref{a:Kernels}), scales as $2^{-n}$ and satisfies $\sum_{n = 0}^N \psi^{(\eps, n)}(z) = 1$ for all $z$. Let us denote for convenience $\N_{\leq N} := \{0,1,\ldots, N\}$.\label{lab:NN} Then for $r_e > 0$ and $\n = (k,p,m) \in \N_{\leq N}^3$ we set
\begin{equation}\label{e:defKhatn}
	\hat K_e^{\eps, \n}(z,\bar z) := \psi^{(\eps, k)}(\bar z-z) \psi^{(\eps, p)}(z) \psi^{(\eps, m)}(\bar z) \hat K_e^{\eps}(z,\bar z),
\end{equation}
where the kernel $\hat K^{\eps}_e$ has been defined in \eqref{e:defKhatn2}. For $\n \in \N_{\leq N}^3$ and $e \in \CCE$ such that $r_e \leq 0$, we define the function
\begin{equation*}
	\hat K_e^{\eps, \n}(z,\bar z) :=
	\begin{cases}
		K^{\eps, k}_e(\bar z - z), & \text{if}~ \n = (k,0,0),\; 0 \leq k \leq N, \\
		0,                         & \text{otherwise},
	\end{cases}
\end{equation*}
where we made use of the expansion of the kernel from Assumption~\ref{a:Kernels}.

For $\lambda \in (0,1]$ we define the set $\CN^{\epsEH}_{\lambda}$ of functions $\n \colon \CCE \to \N_{\leq N}^3$ satisfying $2^{-|\n_{(\star, v^{\uparrow}_{\star})}|} \le \lambda \vee \epsEH$, with $\n_{(\star, v^{\uparrow}_{\star})}$ being the evaluation of the function $\n$ on the edge $(\star, v^{\uparrow}_{\star})$. Then for a function $\n \in \CN^{\epsEH}_{\lambda}$ and a point $\bz = (z_v : v \in \CCV_{\!\bar \star})$, we define
\begin{equation}\label{def:kayhat}
	\hat K^{\eps, \n}(\bz) := \prod_{e \in \CCE} \hat K_e^{\eps, \n_e}(z_{e_-}, z_{e_+}),
\end{equation}
where $z_\star = 0$. Since the functions $\psi^{(\eps, n)}$ sum up to $1$ and since we consider the test function $\phi^\lambda_0$ as a kernel, one can rewrite the generalised convolution \eqref{e:genconv} as
\begin{equation}\label{e:bigsum}
	\CK^{\lambda, \eps}_\CCG(\bz^{\var}) := \sum_{\n \in \CN^{\epsEH}_{\lambda}} \CK^{\eps, \n}_\CCG(\bz^{\var}), \qquad  \CK^{\eps, \n}_\CCG(\bz^{\var}) := \int_{(\R^{d+1})^{\CCV_{\!\bar\star}}} \!\!{\hat K}^{\eps, \n}(\bz)\, \mu_{\CCV_{\!\bar\star}, \bz^{\var}}^{\eps}(\d \bz).
\end{equation}

Since we are interested in estimating the integrals $\CI^{\eps, \L}_{\gamma, \sigma} \CK^{\lambda, \eps}_\CCG$, we can exploit the fact that the integration variables $z_v$ in the kernel \eqref{e:bigsum}, for vertices $v$ belonging to the same component of $\gamma$, are equal. More precisely, we define the set $\CN^{\epsEH}_{\lambda, \gamma}$ in the same way as $\CN^{\epsEH}_{\lambda}$, but using the contracted graph $\CCG_\gamma = (\tilde \CCV, \tilde \CCE)$. Then for a function $\n \in \CN^{\epsEH}_{\lambda, \gamma}$ and a point $\bz = (z_v : v \in \tilde \CCV_{\!\bar \star})$, we define the kernel $\hat K^{\eps, \n}(\bz)$ as in \eqref{def:kayhat}, but with the product over $\tilde \CCE$. Furthermore, we define the measure $\mu_{\tilde \CCV_{\!\bar \star}, \bz^{\var}}^{\eps}$ on $(\R^{d+1})^{\tilde \CCV_{\!\bar\star}}$ by
\begin{equation}\label{eq:measure-mu-general}
	\mu_{\tilde \CCV_{\!\bar \star}, \bz^{\var}}^{\eps}(\d \bz) := \Bigl(\prod_{v \in \tilde \CCV_{\!\bar\star} \setminus \tilde \CCV_{\!\var}} \mu_v^{\eps}(\d z_v)\Bigr) \Bigl(\prod_{w \in \tilde \CCV_{\!\var}} \delta_{z_w - z^{\var}_{w_*}} \d z_w\Bigr),
\end{equation}
where $w_*$ is the first element (with respect to a chosen order of vertices) in $\i_\gamma^{-1}(w)$, and the map $\i_\gamma$ has been introduced in the beginning of this section. In other words, this measure identifies the variables in $\tilde \CCV_{\!\var}$ which correspond to the same component of $\gamma$. Then we define the kernel
\begin{equation}\label{e:Kernel_hat}
	\CK^{\eps, \n}_{\CCG_\gamma}(\bz^{\var}) := \int_{(\R^{d+1})^{\tilde \CCV_{\!\bar\star}}} \!\!{\hat K}^{\eps, \n}(\bz)\, \mu_{\tilde \CCV_{\!\bar \star}, \bz^{\var}}^{\eps}(\d \bz),
\end{equation}
and write the multiple stochastic integral as $\CI^{\eps, \L}_{\gamma, \sigma} \CK^{\lambda, \eps}_\CCG = \sum_{\n \in \CN^{\epsEH}_{\lambda, \gamma}} \!\!\CI^{\eps, \L}_{\gamma, \sigma} \CK_{\CCG_\gamma}^{\eps, \n}$. Using this expansion and applying Minkowski's inequality, we obtain the bound
\begin{equation}\label{e:wantedBound2}
	\E_p \sup_{t \in \R_+} \bigl| (\CI^{\eps, \L}_{\gamma, \sigma} \CK^{\lambda, \eps}_\CCG)_t \bigr| \leq \sum_{\n \in \CN^{\epsEH}_{\lambda, \gamma}} \E_p \sup_{t \in \R_+} \bigl|(\CI^{\eps, \L}_{\gamma, \sigma} \CK_{\CCG_\gamma}^{\eps, \n})_t \bigr|.
\end{equation}
Bounding a multiple integral of the generalised convolution boils down to bounding integrals in \eqref{e:wantedBound2} and summing over the functions $\n \in \CN^{\epsEH}_{\lambda, \gamma}$. This is what we do in the next sections, where, following the idea of \cite[Appendix~A.2]{HQ18}, we use a multiscale clustering in the sum over $\n$.

\subsection{Bounds on iterated integrals}
\label{sec:bounds-multiscale}

We associate to every point $\bz \in (\R^{d+1})^{\tilde \CCV}$ a rooted labelled binary tree $(T,\ell)$, such that $\|z_v - z_{w}\|_\s \sim 2^{-\ell_{v \wedge w}}$ and $\ell_{v \wedge w} \in \N_{\leq N}$, where $v \wedge w$ is the closest common ancestor of $v$ and $w$. Moreover, the labels $\ell$ satisfy $\ell_\v \ge \ell_\w$ whenever $\v \ge \w$, where $\v \ge \w$ means that $\w$ belongs to the shortest path from $\v$ to the root of the tree $T$. See \cite[Appendix~A.2]{HQ18} for construction of such tree and also for the terminology which we are going to use. Given a set of vertices $\tilde \CCV$, we denote by $\CCT^{\epsEH}(\tilde\CCV)$  the set of rooted labelled binary trees $(T,\ell)$ as above, which have $\tilde\CCV$ as their set of leaves. Denote furthermore by $\CCT_{\!\lambda}^\epsEH(\tilde\CCV)$ the subset of those labelled trees in $\CCT^{\epsEH}(\tilde\CCV)$ with the property that $2^{-\ell_{\star \wedge v^\uparrow_{\star}}} \le \lambda$. 

Our next aim is to write summation in \eqref{e:wantedBound2} over such labelled trees $(T, \ell)$ and then over those functions $\n$ which are close in some sense to the labeling $\ell$. To this end, for the constant\footnote{Our value of $c$ is different from the analogous value in \cite[Definition~A.8]{HQ18}, because the kernels $K^{\epsEH, n}_{e}$ from Assumption~\ref{a:Kernels} have a different support. The need to define $c$ in this way can be seen from the proof of \cite[Lemma~A.9]{HQ18}.} $c := (\log_2 |\tilde\CCV|+ |\log_2 C_1|) \vee |\log_2 C_2|$, where the constants $C_1, C_2$ are from Assumption~\ref{a:Kernels}, we define the set $\CN^{\epsEH}_\gamma(T, \ell)$ consisting of all functions $\n \colon \tilde \CCE \to \N_{\leq N}^3$ such that
\begin{enumerate}[topsep=2pt, itemsep=0ex]
	\item for every edge $e = (v,w)$ with $r_e \le 0$, one has $\n_e = (k,0,0) \in \N_{\leq N}^3$ with $|k - \ell_{v\wedge w}| \le c$,
	\item for every edge $e = (v,w)$ with $r_e > 0$, one has $\n_e = (k,p,m) \in \N_{\leq N}^3$ with $|k - \ell_{v\wedge w}| \le c$, $|p - \ell_{v\wedge \star}| \le c$, and $|m - \ell_{w\wedge \star}| \le c$.
\end{enumerate}
Then we have the following analogue of \cite[Lemma~A.9]{HQ18}, which is proved in exactly the same way.

\begin{lemma}\label{lem:from_n_to_tree}
	Let us fix a point $\bz^{\var} \in (\R^{d+1})^{\CCV_{\!\var}}$. Let $\n \colon \tilde\CCE \to \N_{\leq N}^3$ be such that the kernel $\CK^{\eps, \n}_{\CCG_\gamma}(\bz^{\var})$ defined in \eqref{e:Kernel_hat} does not vanish. Then there exists a labelled tree $(T, \ell) \in \CCT^{\epsEH}_{\!\lambda}(\tilde\CCV)$ such that $\n \in \CN^{\epsEH}_\gamma(T, \ell)$.
\end{lemma}
 
Using this result, the right-hand side of \eqref{e:wantedBound2} can be estimated as
\begin{equation}\label{e:wantedBound3.5}
	\E_p \sup_{t \in \R_+} \bigl|(\CI^{\eps, \L}_{\gamma, \sigma} \CK^{\lambda, \eps}_\CCG)_t \bigr| \leq \sum_{(T, \ell) \in \CCT^{\epsEH}_{\!\lambda}(\tilde\CCV)} \sum_{\n \in \CN^{\epsEH}_\gamma(T, \ell)} \E_p \sup_{t \in \R_+} \bigl|(\CI^{\eps, \L}_{\gamma, \sigma} \CK_{\CCG_\gamma}^{\eps, \n})_t \bigr|.
\end{equation}

We will now modify the kernels in \eqref{def:kayhat} in the same way how it was done in \cite[Appendix~A.5]{HQ18}. Let $\Amin \subset \tilde\CCE$ contain those edges $e = (e_-, e_+)$ which have the label $r_e < 0$, and for which any two vertices $\{u,v\}$ satisfying $u \wedge v = e_- \wedge e_+$ coincide with $\{e_-,e_+\}$. Then we can factorize \eqref{def:kayhat} as
\begin{equation}\label{eq:K-hat-A-minus}
	\hat K^{\eps, \n}(\bz) = \hat G^{\epsEH, \n}(\bz) \Bigl(\prod_{e \in \Amin} \hat K_e^{\eps, \n_e}(z_{e_-}, z_{e_+})\Bigr), \qquad \hat G^{\epsEH, \n}(\bz) := \prod_{e \notin \Amin} \hat K_e^{\eps, \n_e}(z_{e_-}, z_{e_+}).
\end{equation}
For $e = (e_-, e_+)$ and $r > 0$ we define the operator $\SY^r_e$ acting on sufficiently smooth functions $V: (\R^{d+1})^{\tilde \CCV} \to \R$ as
\begin{equation*}
	(\SY^r_e V)(\bz) := V(\bz) - \sum_{|k|_\s < r} \frac{(z_{e_+} - z_{e_-})^k}{k!} (D^k_{e_+} V) (P_e(\bz)),
\end{equation*}
where $D^k_{e_+}$ is a derivative with respect to $z_{e_+}$ and where $(P_e(\bz))_v = z_v$ if $v \neq e_+$ and $(P_e(\bz))_v = z_{e_-}$ if $v = e_+$. Furthermore, writing $\Amin = \{e^{(1)}, \ldots, e^{(k)}\}$ for some $k \geq 0$, we define the kernel
\begin{equation}\label{eq:DKernel_tilde}
	\tilde K^{\eps, \n}(\bz) := \Bigl(\SY^{r_{e^{(k)}}}_{e^{(k)}} \cdots \SY^{r_{e^{(1)}}}_{e^{(1)}}\hat G^{\epsEH, \n}(\bz)\Bigr) \Bigl(\prod_{e \in \Amin} \hat K_e^{\eps, \n_e}(z_{e_-}, z_{e_+})\Bigr).
\end{equation}
Then for every $\bz^{\var}$ we have
\begin{equation}\label{eq:two_kernels}
	\int_{(\R^{d+1})^{\tilde\CCV_{\!\bar \star}}} \!\!{\hat K}^{\eps, \n}(\bz)\, \mu_{\tilde\CCV_{\!\bar \star}, \bz^{\var}}^{\eps}(\d \bz) = \int_{(\R^{d+1})^{\tilde\CCV_{\!\bar \star}}} \!\!{\tilde K}^{\eps, \n}(\bz)\, \mu_{\tilde\CCV_{\!\bar \star}, \bz^{\var}}^{\eps}(\d \bz),
\end{equation}
which is just a reformulation of the argument below \cite[Equation~A.26]{HQ18} in our context. Then \eqref{e:wantedBound3.5} can be written as
\begin{equation}\label{e:wantedBound4}
	\E_p \sup_{t \in \R_+} \bigl|(\CI^{\eps, \L}_{\gamma, \sigma} \CK^{\lambda, \eps}_\CCG)_t \bigr| \leq \sum_{(T, \ell) \in \CCT^{\epsEH}_{\!\lambda}(\tilde\CCV)} \sum_{\n \in \CN^{\epsEH}_\gamma(T, \ell)} \E_p \sup_{t \in \R_+} \bigl|(\CI^{\eps, \L}_{\gamma, \sigma} \tilde \CK_{\CCG_\gamma}^{\eps, \n})_t \bigr|,
\end{equation}
with the new kernels
\begin{equation}\label{e:K_tilde}
	\tilde \CK^{\eps, \n}_{\CCG_\gamma}(\bz^{\var}) := \int_{(\R^{d+1})^{\tilde\CCV_{\!\bar \star}}} \!\!{\tilde K}^{\eps, \n}(\bz)\,\mu_{\tilde\CCV_{\!\bar \star}, \bz^{\var}}^{\eps}(\d \bz).
\end{equation}
Using the notation \eqref{eq:F_gamma}, we denote with $(\tilde \CK^{\eps, \n}_{\CCG_\gamma})^{\gamma, \sigma} : (\R^{d+1})^{\tilde\CCV_{\!\var}} \to \R$ the kernel which is obtained from $\tilde \CK^{\eps, \n}_{\CCG_\gamma}$ by making all the variables from the same component of $\gamma$ equal. We would like to apply  Theorem~\ref{thm:integral-bound-renorm} to bound the stochastic integrals in \eqref{e:wantedBound4}. For this, we need to estimate the norms \eqref{eq:norms-def} of the kernel $(\tilde \CK^{\eps, \n}_{\CCG_\gamma})^{\gamma, \sigma}(\bz^{\var})$, which is what we are going to do now. 

\medskip
Let $T^\circ$ denote the set of interior nodes of the tree $T$. Then for $e \in \hat E$ let us define the function $\eta_e : T^\circ \to \R$ by
\begin{align*}
	\eta_e(v) & := - \hat{a}_e \1_{e_\uparrow} (v) +  r_e \bigl(\1_{e_+\wedge \star} (v) - \1_{e_\uparrow} (v)\bigr) \1_{ r_e >0,\, e_+\wedge \star > e_\uparrow} \\
	          & \qquad + (1-  r_e -  \hat{a}_e) \bigl(\1_{e_-\wedge \star} (v) - \1_{e_\uparrow} (v)\bigr) \1_{ r_e >0,\, e_-\wedge \star > e_\uparrow},
\end{align*}
where $\1_v(w) := \1_{v = w}$ and $e_\uparrow := e_- \wedge e_+ \in T^\circ$ for an edge $e = (e_-, e_+) \in \hat\CCE$. The function $\eta_e$ coincides with the one defined in \cite[Equation~A.20]{HQ18} and is used to bound the generalised convolution without taking into account negative renormalisation. To consider negative renormalisation we define by analogy with \cite[Equation~A.27]{HQ18} a modified function 
\begin{equation}\label{e:tilde_eta}
	\tilde \eta(v) := |\s| + \sum_{e \in \hat\CCE} \tilde \eta_e(v), \qquad \tilde \eta_e(v) := \eta_e(v) - r_e\, \1_{e \in \Amin} \bigl(\1_{e_\uparrow} (v) - \1_{e_\Uparrow} (v) \bigr),
\end{equation}
where the interior node $e_\Uparrow \in T^\circ$ is of the form $w \wedge e_-$ with $w \not \in e$ which is the furthest from the root. 

From the fixed order $\sigma$ of the variables in \eqref{eq:def_integrals_kernels} we obtain an oder of the vertices in $\tilde\CCV_{\!\var}$. Then we write $\tilde{\CCV}_{\!\var} = (v_1, \ldots, v_m)$ according to this order. For every $v \in \tilde\CCV_{\!\var}$ we denote by $v^{\shortrightarrow}$ the element following after $v$ with respect to this order, and in case when there is no following element we define $v^{\shortrightarrow} = \STAR$. For $A \subset \tilde \CCV_{\!\var}$, let $T_A$ be the subtree of $T$ containing all the leaves $v$ and $v^{\shortrightarrow}$, for $v \in A \sqcup \{\STAR\}$, and all the inner nodes $v \wedge v^{\shortrightarrow}$ for such $v$. Let $T_A^\circ$ contain the inner nodes of $T_A$. 

Then we have the following bound on the norms \eqref{eq:norms-def} of the kernels $(\tilde \CK^{\eps, \n}_{\CCG_\gamma})^{\gamma, \sigma}$.

\begin{lemma}\label{lem:KTildeBound}
In the setting of Theorem~\ref{thm:convolutions}, let $\bp$ be one of the functions in the sum in \eqref{e:genconvBound}. Then for any labeled tree $(T, \ell) \in \CCT^{\epsEH}_{\!\lambda}(\tilde\CCV)$ there is a constant $C$ such that for every $\n \in \CN^{\epsEH}_\gamma(T,\ell)$ one has the bound
	\begin{equation}
		\big\Vert (\tilde \CK^{\eps, \n}_{\CCG_\gamma})^{\gamma, \sigma} \big\Vert_{L^{\bp}_\eps} \leq C \epsEH^{- \delta_{\gamma}(\bp)} \Bigl(\prod_{\nu \in T^\circ} 2^{-\ell_\nu \tilde \eta(\nu)}\Bigr) \Bigl(\prod_{\nu \in T_{\Gamma}^\circ} 2^{\ell_{\nu}|\s| }\Bigr)^{\frac{1}{2}}, \label{eq:KTildeBound1} 
	\end{equation}
	where we use the norm \eqref{eq:norms-def} and the constant $\delta_{\gamma}(\bp)$ defined in \eqref{eq:delta-gamma-p}.
\end{lemma}

\begin{proof}
We are going to prove a more general result; namely, we will prove a bound on the norm of the kernel $(\tilde \CK^{\eps, \n}_{\CCG_\gamma})^{\gamma, \sigma}$ some of whose variables are fixed. For this, we take $0 \leq M < m$ and the set $D = \{v_{M+1}, \ldots, v_m\} \subset \tilde \CCV_{\!\var}$ of vertices and we will fix the values of the variables corresponding to these vertices. More precisely, for $\bz_{D} \in (\R^{d+1})^D$ we write $(\tilde \CK^{\eps, \n}_{\CCG_\gamma})^{\gamma, \sigma} \big|_{\bz_{D}}$ for the function from $(\R^{d+1})^{\tilde\CCV_{\!\var} \setminus D}$ to $\R$, which is obtained from $(\tilde \CK^{\eps, \n}_{\CCG_\gamma})^{\gamma, \sigma}(\bz^{\var})$ by fixing the values of the variables $z^{\var}_v$ with $v \in D$. We extend this definition for $D = \emptyset$ (which corresponds to $M = m$) by $(\tilde \CK^{\eps, \n}_{\CCG_\gamma})^{\gamma, \sigma} \big|_{\bz_{\emptyset}} = (\tilde \CK^{\eps, \n}_{\CCG_\gamma})^{\gamma, \sigma}$.

For $1 \leq M \leq m$ and for a function $\bp \in \SP_{\!M}$, we are going to prove the bound
\begin{equation}\label{eq:KTildeBound-need}
		\big\Vert (\tilde \CK^{\eps, \n}_{\CCG_\gamma})^{\gamma, \sigma} \big|_{\bz_{D}} \big\Vert_{L^{\bp}_\eps} \leq C \Bigl(\prod_{\nu \in T^\circ} 2^{-\ell_\nu \tilde \eta(\nu)}\Bigr) \Bigl(\prod_{\nu \in T^\circ_{D \sqcup \bp^{-1}(\infty)}} 2^{\ell_\nu|\s|}\Bigr) \Bigl(\prod_{\nu \in T^\circ_{\bp^{-1}(2)}} 2^{\ell_\nu|\s|}\Bigr)^{\frac{1}{2}}
\end{equation}
uniformly in $\bz_{D} \in (\R^{d+1})^D$. Moreover, we will show that for $M = 0$ (in which case $D = \tilde\CCV_{\!\var}$) the same bound holds for the absolute value of $(\tilde \CK^{\eps, \n}_{\CCG_\gamma})^{\gamma, \sigma} \big|_{\bz_{D}}$.

We can see that the bound \eqref{eq:KTildeBound1} follows from \eqref{eq:KTildeBound-need} in the particular case $M = m$ corresponding to $D = \emptyset$. To see it, we note that \eqref{eq:KTildeBound-need} simplifies to
\begin{equation}\label{eq:KTildeBound-need-D-is-empty}
		\big\Vert (\tilde \CK^{\eps, \n}_{\CCG_\gamma})^{\gamma, \sigma} \big\Vert_{L^{\bp}_\eps} \leq C \Bigl(\prod_{\nu \in T^\circ} 2^{-\ell_\nu \tilde \eta(\nu)}\Bigr) \Bigl(\prod_{\nu \in T^\circ_{\bp^{-1}(\infty)}} 2^{\ell_\nu|\s|}\Bigr) \Bigl(\prod_{\nu \in T^\circ_{\bp^{-1}(2)}} 2^{\ell_\nu|\s|}\Bigr)^{\frac{1}{2}}.
\end{equation}
Our next goal is to replace the product over $\nu \in T^\circ_{\bp^{-1}(2)}$ by the product over $\nu \in T^\circ_{\Gamma}$. We do it by noting that $\tilde\CCV_{\!\var} = \bp^{-1}(1) \sqcup \bp^{-1}(2) \sqcup \bp^{-1}(\infty)$ and using simple operations on the sets. Namely, we have $\bp^{-1}(2) = \bigl(\Gamma \sqcup (\bp^{-1}(2) \setminus \Gamma)\bigr) \setminus \bigl(\bp^{-1}(\infty) \cap \Gamma\bigr)$, where we used the assumption $\bp^{-1}(1) \cap \Gamma = \emptyset$ in \eqref{e:genconvBound}. Then we write the products over $\nu \in T^\circ_{\bp^{-1}(2)}$ as
\begin{equation*}
\prod_{\nu \in T^\circ_{\bp^{-1}(2)}} 2^{\ell_{\nu}|\s| }= \Bigl(\prod_{\nu \in T^\circ_{\Gamma}} 2^{\ell_\nu|\s| }\Bigr) \Bigl(\prod_{\nu \in T^\circ_{\bp^{-1}(2)} \setminus T^\circ_{\Gamma}} 2^{\ell_\nu|\s| }\Bigr) \Bigl(\prod_{\nu \in T^\circ_{\bp^{-1}(\infty)} \cap T^\circ_{\Gamma}} 2^{-\ell_\nu|\s|}\Bigr).
\end{equation*}
Hence, the product on the right-hand side of \eqref{eq:KTildeBound-need-D-is-empty} equals 
\begin{align*}
&\Bigl(\prod_{\nu \in T^\circ} 2^{-\ell_\nu \tilde \eta(\nu)}\Bigr) \Bigl(\prod_{\nu \in T^\circ_{\Gamma}} 2^{\ell_\nu|\s| }\Bigr)^{\frac{1}{2}} \Bigl(\prod_{\nu \in T^\circ_{\bp^{-1}(\infty)}} 2^{\ell_\nu|\s|}\Bigr) \Bigl(\prod_{\nu \in T^\circ_{\bp^{-1}(2)} \setminus T^\circ_{\Gamma}} 2^{\ell_\nu|\s| }\Bigr)^{\frac{1}{2}} \Bigl(\prod_{\nu \in T^\circ_{\bp^{-1}(\infty)} \cap T^\circ_{\Gamma}} 2^{-\ell_\nu|\s|}\Bigr)^{\frac{1}{2}} \\
& =\Bigl(\prod_{\nu \in T^\circ} 2^{-\ell_\nu \tilde \eta(\nu)}\Bigr) \Bigl(\prod_{\nu \in T^\circ_{\Gamma}} 2^{\ell_\nu|\s| }\Bigr)^{\frac{1}{2}} \Bigl(\prod_{\nu \in T^\circ_{\bp^{-1}(\infty) \setminus T^\circ_{\Gamma}}} 2^{\ell_\nu|\s|}\Bigr) \Bigl(\prod_{\nu \in T^\circ_{\bp^{-1}(\infty)} \cap T^\circ_{\Gamma}} 2^{\ell_\nu|\s|}\Bigr)^{\frac{1}{2}} \Bigl(\prod_{\nu \in T^\circ_{\bp^{-1}(2)} \setminus T^\circ_{\Gamma}} 2^{\ell_\nu|\s| }\Bigr)^{\frac{1}{2}}.
\end{align*}
	According to our definition of the labels $\ell$ in Section~\ref{sec:bounds-multiscale} we have $2^{-\ell_\nu} \geq \epsEH$, and we can bound the preceding expression by the right-hand side of \eqref{eq:KTildeBound1}.
\medskip

Now, we turn to the proof of \eqref{eq:KTildeBound-need}. From \cite[Lemma~A.16]{HQ18} we conclude that the kernel \eqref{eq:DKernel_tilde} satisfies
	\begin{equation}\label{e:wantedBoundK}
		\sup_{\bz \in (\R^{d+1})^{\tilde \CCV_{\!\bar \star}}}|\tilde K^{\eps, \n}(\bz)| \lesssim \prod_{\nu \in T^\circ} 2^{-\ell_\nu(\tilde \eta(\nu) - |\s|)},
	\end{equation}
	uniformly over all $\n \in \CN^{\epsEH}_\gamma(T,\ell)$. We will use this estimate to bound the norms \eqref{eq:norms-def}. 
	
	Let us first consider the case $M = 0$ corresponding to $D = \tilde{\CCV}_{\!\var}$. We have $(\tilde \CK^{\eps, \n}_{\CCG_\gamma})^{\gamma, \sigma} \big|_{\bz_{D}} = (\tilde \CK^{\eps, \n}_{\CCG_\gamma})^{\gamma, \sigma} (\bz_{D})$, and we are going to bound it absolutely. From \eqref{eq:measure-mu-general} and \eqref{e:K_tilde} we get 
	\begin{equation}\label{eq:bounds-induction}
	(\tilde \CK^{\eps, \n}_{\CCG_\gamma})^{\gamma, \sigma} \Big|_{\bz_{D}} = \int_{(\R^{d+1})^{\tilde\CCV_{\!\bar \star} \setminus \tilde \CCV_{\!\var}}} \!\!(\tilde \CK^{\eps, \n}_{\CCG_\gamma})^{\gamma, \sigma}(\bz) \Big|_{\bz^{\var} = \bz_{D}}\, \prod_{v \in \tilde \CCV_{\!\bar\star} \setminus \tilde \CCV_{\!\var}} \mu_v^{\eps}(\d z_v).
	\end{equation}
	We write $\bz = (\bz_{\tilde\CCV_{\!\bar \star} \setminus \tilde \CCV_{\!\var}}, \bz^{\var})$, where $\bz_{\tilde\CCV_{\!\bar \star} \setminus \tilde \CCV_{\!\var}}$ contains the variables $z_v$ with $v \in \tilde\CCV_{\!\bar \star} \setminus \tilde \CCV_{\!\var}$. The definition of the kernel and properties of the measures $\mu^\eps_v$ allow to bound the preceding expression by a constant times
	\begin{equation*}
		|\CA^\eps_{\gamma}|_{|\tilde\CCV_{\!\bar \star} \setminus \tilde \CCV_{\!\var}|} \sup_{\bz_{\tilde\CCV_{\!\bar \star} \setminus \tilde \CCV_{\!\var}} \in (\R^{d+1})^{\tilde\CCV_{\!\bar \star} \setminus \tilde \CCV_{\!\var}}}|\tilde K^{\eps, \n}(\bz_{\tilde\CCV_{\!\bar \star} \setminus \tilde \CCV_{\!\var}}, \bz_{D})|,
	\end{equation*}
	where we write $|\cdot|_\alpha$ for the $(\alpha |\s|)$-dimensional Lebesgue measure, and the set $\CA^\eps_{\gamma}$ contains all points $\{ z_v : v \in \tilde\CCV_{\!\bar \star} \setminus \tilde \CCV_{\!\var}\}$, satisfying the conditions 
	\begin{align*}
		\|z_v - z_w\|_\s \leq C' 2^{-\ell_{v\wedge w}} \qquad &\text{for}~v, w \in \tilde\CCV_{\!\bar \star} \setminus \tilde \CCV_{\!\var},\\
		\|z_v - z^{\var}_w\|_\s \leq C' 2^{-\ell_{v\wedge w}} \qquad &\text{for}~v \in \tilde\CCV_{\!\bar \star} \setminus \tilde \CCV_{\!\var}, ~ w \in \tilde \CCV_{\!\var}.
	\end{align*}
Here, we use the fact that $2^{-\ell_{v\wedge w}} \geq \epsEH$, which is a consequence of the assumption  $(T,\ell) \in \CCT^{\epsEH}_{\!\lambda}(\tilde\CCV)$. For an interior node $\nu \in T^\circ$, let us choose $v_\pm \in \tilde \CCV$ to be such that $v_- \wedge v_+ = \nu$ and there is an edge from $v_-$ to $v_+$. Then the collection of edges $\{(v_-, v_+) : \nu \in T^\circ\}$ forms a spanning tree of $\tilde\CCV$, and $\CA^\eps_{\gamma}$ is a subset of
	\begin{align*}
		 \Bigl\{\bz_{\tilde\CCV_{\!\bar \star} \setminus \tilde \CCV_{\!\var}} \in (\R^{d+1})^{\tilde\CCV_{\!\bar \star} \setminus \tilde \CCV_{\!\var}}  :\; & \| z_{v_+} - z_{v_-}\|_\s \leq C' 2^{-\ell_{\nu}}\;\; \forall\; \nu \in T^\circ, v_\pm \notin \tilde \CCV_{\!\var}, \\
		 &\qquad \| z_{v_+} - z^{\var}_{v_-}\|_\s \leq C' 2^{-\ell_{\nu}}\;\; \forall\; \nu \in T^\circ, v_- \in \tilde \CCV_{\!\var} \Bigr\},
	\end{align*}
	where $z_\star = 0$. Here, we used the property that the vertices in $\tilde \CCV_{\!\var}$ have only outgoing edges. Next, we compute the Lebesgue measure of this set. We integrate out the variables $z_v$ one by one, for $v \notin \tilde \CCV_{\!\var}$, which gives an expression of order $\prod_{\substack{\nu \in T^\circ \setminus T^\circ_{\tilde \CCV_{\!\var}}}} 2^{-\ell_\nu|\s|}$. Hence,
	\begin{equation*}
	|\CA^\eps_{\gamma}|_{|\tilde\CCV_{\!\bar \star} \setminus \tilde \CCV_{\!\var}|} \lesssim \prod_{\substack{\nu \in T^\circ \setminus T^\circ_{\tilde \CCV_{\!\var}}}} 2^{-\ell_\nu|\s|},
	\end{equation*}
	combining which with the estimate on the kernel \eqref{e:wantedBoundK} we get
	\begin{equation}
	\Big| (\tilde \CK^{\eps, \n}_{\CCG_\gamma})^{\gamma, \sigma} \big|_{\bz_{D}} \Bigr| \lesssim \Bigl(\prod_{\substack{\nu \in T^\circ \setminus T^\circ_{\tilde \CCV_{\!\var}}}} 2^{-\ell_\nu|\s|}\Bigr) \Bigl(\prod_{\nu \in T^\circ} 2^{-\ell_\nu(\tilde \eta(\nu) - |\s|)}\Bigr) \lesssim \Bigl(\prod_{\nu \in T^\circ} 2^{-\ell_\nu \tilde \eta(\nu)}\Bigr) \Bigl(\prod_{\substack{\nu \in T^\circ_{\tilde \CCV_{\!\var}}}} 2^{\ell_\nu|\s|}\Bigr). \label{eq:K-simple-bound2}
	\end{equation}
	Recalling that $D = \tilde \CCV_{\!\var}$, this is exactly the right-hand side of \eqref{eq:KTildeBound-need} with $\bp^{-1}(\infty) = \bp^{-1}(2) = \emptyset$.
	
	\medskip
	Now we proceed with the proof of \eqref{eq:KTildeBound-need} by induction over $M = 1, \ldots, m$. For $\bp \in \SP_{\!M}$ with $M \geq 2$, let $\bar \bp$ be the restriction of $\bp$ to $\{1, \ldots, M-1\}$. Let us furthermore define the function with respect to the variable $z_{v_M}$ corresponding to the vertex $v_M$:
	\begin{equation}\label{eq:F}
	F(z_{v_M}) := \Bigl\Vert (\tilde \CK^{\eps, \n}_{\CCG_\gamma})^{\gamma, \sigma} \big|_{\bz_{D \sqcup \{v_M\}}} \Bigr\Vert_{L^{\bar \bp}_\eps}
	\end{equation}
	if $M \geq 2$ and $F(z_{v_M}) := (\tilde \CK^{\eps, \n}_{\CCG_\gamma})^{\gamma, \sigma} \big|_{\bz_{D \sqcup \{v_M\}}}$ if $M = 1$. Then we use the definition \eqref{eq:norms-def} to write
	\begin{equation}\label{eq:KTildeBound1_1}
		\Bigl\Vert (\tilde \CK^{\eps, \n}_{\CCG_\gamma})^\gamma \big|_{\bz_{D}} \Bigr\Vert_{L^{\bp}_\eps} \leq \| F \|_{L^{\bp(M)}_\eps}.
	\end{equation}
	We got an inequality because we omitted the indicator functions in the definition \eqref{eq:norms-def}, which corresponds to increasing the domain of integration of the function. We need to bound the norm on the right-hand side of \eqref{eq:KTildeBound1_1}, for what we consider all possible values of $\bp(M)$ one-by-one.
\medskip
	
If $\bp(M) = 2$, then the definition \eqref{eq:norm-eps} yields
	\begin{equation}\label{eq:K-simple-bound}
	\| F \|_{L^2_\eps} = \biggl(\eps^d \sum_{x \in \Le} \int_{0}^\infty F(r,x)^2 \d r\biggr)^{\frac{1}{2}}.
	\end{equation}
The norms \eqref{eq:norms-def} are defined on the time interval $[0, T]$, which means that the integral in \eqref{eq:K-simple-bound} with respect to $r$ should be on $[0, T]$. Since the kernel $(\tilde \CK^{\eps, \n}_{\CCG_\gamma})^{\gamma, \sigma}$ is compactly supported, we can take $T$ big enough so that the integrals can be written on $[0, \infty)$. We use this convention in all formulas below. 

Recalling the definitions of the kernel \eqref{e:K_tilde} and the function $\n$ in Section~\ref{sec:bounds-multiscale}, we conclude that the function $F(z_{v_M})$ is supported on $\|z_{v_M} - z_{v^{\shortrightarrow}_M}\|_\s \lesssim 2^{- \ell_{v_M \wedge v^{\shortrightarrow}_M}}$. We recall that $v^{\shortrightarrow}_M = v_{M+1}$ if $M < m$ and $v^{\shortrightarrow}_m = \STAR$. Then the norm \eqref{eq:K-simple-bound} can be bounded as
\begin{equation}\label{eq:F1-bound}
	\| F \|_{L^2_\eps} \lesssim \biggl( 2^{- \ell_{v_M \wedge v^{\shortrightarrow}_M} |\s|} \sup_{\|z_{v_M} - z_{v^{\shortrightarrow}_M}\|_\s \lesssim 2^{- \ell_{v_M \wedge v^{\shortrightarrow}_M}}} F(z_{v_M})^2 \biggr)^{\frac{1}{2}}.
\end{equation}

If $M = 1$, i.e. the set $D \sqcup \{v_1\}$ in \eqref{eq:F} equals $\tilde{\CCV}_{\!\var}$, then the function $F(z_{v_1})$ satisfies the bound \eqref{eq:K-simple-bound2}. Then the expression \eqref{eq:F1-bound} is bounded by a constant times
\begin{align}\nonumber
	&2^{- \ell_{v_1 \wedge v^{\shortrightarrow}_1} |\s|/2} \Bigl(\prod_{\nu \in T^\circ} 2^{-\ell_\nu \tilde \eta(\nu)}\Bigr) \Bigl(\prod_{\substack{\nu \in T^\circ_{\tilde \CCV_{\!\var}}}} 2^{\ell_\nu|\s|}\Bigr) \\
	&\qquad = \Bigl(\prod_{\nu \in T^\circ} 2^{-\ell_\nu \tilde \eta(\nu)}\Bigr) \Bigl(\prod_{\substack{\nu \in T^\circ_{\tilde{\CCV}_{\!\var}} \setminus \{v_1 \wedge v^{\shortrightarrow}_1\}}} 2^{\ell_\nu|\s|}\Bigr) 2^{\ell_{v_1 \wedge v^{\shortrightarrow}_1} |\s|/2}. \label{eq:F1-bound1}
\end{align}
Since $T^\circ_{\tilde{\CCV}_{\!\var}} \setminus \{v_1 \wedge v^{\shortrightarrow}_1\} = T^\circ_{\tilde{\CCV}_{\!\var} \setminus \{v_1\}}$, this is exactly \eqref{eq:KTildeBound-need} with $\bp^{-1}(2) = \{1\}$, $\bp^{-1}(1) = \bp^{-1}(\infty) = \emptyset$ and $D = \tilde{\CCV}_{\!\var} \setminus \{v_1\}$.

If $M \geq 2$, i.e. the set $D \sqcup \{v_M\}$ in \eqref{eq:F} is a strict subset of $\tilde{\CCV}_{\!\var}$, then by the induction hypothesis the function $F(z_{v_M})$ satisfies the bound \eqref{eq:KTildeBound-need} with the function $\bar \bp$ and the set $D \sqcup \{v_M\}$. Then the expression \eqref{eq:F1-bound} is bounded by a constant multiple of 
\begin{equation}\label{eq:F1-bound2}
2^{- \ell_{v_M \wedge v^{\shortrightarrow}_M} |\s| / 2} \Bigl(\prod_{\nu \in T^\circ} 2^{-\ell_\nu \tilde \eta(\nu)}\Bigr) \Bigl(\prod_{\nu \in T^\circ_{D \sqcup \{v_M\} \sqcup \bar \bp^{-1}(\infty)}} 2^{\ell_\nu|\s| }\Bigr) \Bigl(\prod_{\nu \in T^\circ_{\bar \bp^{-1}(2)}} 2^{\ell_\nu|\s| }\Bigr)^{\frac{1}{2}}.
\end{equation}
Since $T^\circ_{D \sqcup \{v_M\} \sqcup \bar \bp^{-1}(\infty)} = T^\circ_{D \sqcup \bar \bp^{-1}(\infty)} \sqcup \{v_M \wedge v_M^{\shortrightarrow}\}$ and $T^\circ_{\bar \bp^{-1}(2)} = T^\circ_{\bp^{-1}(2)} \setminus \{v_M \wedge v_M^{\shortrightarrow}\}$, this gives the required expression \eqref{eq:KTildeBound-need}.
\medskip

Now we consider the case $\bp(M) = 1$ in \eqref{eq:KTildeBound1_1}. Similarly to \eqref{eq:F1-bound} we get 
\begin{equation}\label{eq:F2-bound}
	\| F \|_{L^1_\eps} \lesssim 2^{- \ell_{v_M \wedge v^{\shortrightarrow}_M} |\s|} \sup_{\|z_{v_M} - z_{v^{\shortrightarrow}_M}\|_\s \lesssim 2^{- \ell_{v_M \wedge v^{\shortrightarrow}_M}}} |F(z_{v_M})|.
\end{equation}
If $M = 1$, then by analogy with \eqref{eq:F1-bound1} we bound the preceding expression by a constant times 
\begin{equation*}
	2^{- \ell_{v_1 \wedge v^{\shortrightarrow}_1} |\s|} \Bigl(\prod_{\nu \in T^\circ} 2^{-\ell_\nu \tilde \eta(\nu)}\Bigr) \Bigl(\prod_{\substack{\nu \in T^\circ_{\tilde \CCV_{\!\var}}}} 2^{\ell_\nu|\s|}\Bigr) = \Bigl(\prod_{\nu \in T^\circ} 2^{-\ell_\nu \tilde \eta(\nu)}\Bigr) \Bigl(\prod_{\substack{\nu \in T^\circ_{\tilde{\CCV}_{\!\var}} \setminus \{v_1 \wedge v^{\shortrightarrow}_1\}}} 2^{\ell_\nu|\s|}\Bigr),
\end{equation*}
which is exactly \eqref{eq:KTildeBound-need} with $D = \tilde{\CCV}_{\!\var} \setminus \{v_1\}$. In the case $M \geq 2$ we use the induction hypothesis, and by analogy with \eqref{eq:F1-bound2}, we bound \eqref{eq:F2-bound} by a constant times
\begin{equation*}
2^{- \ell_{v_M \wedge v^{\shortrightarrow}_M} |\s|} \Bigl(\prod_{\nu \in T^\circ} 2^{-\ell_\nu \tilde \eta(\nu)}\Bigr) \Bigl(\prod_{\nu \in T^\circ_{D \sqcup \bar \bp^{-1}(\infty) \sqcup \{v_M\}}} 2^{\ell_\nu|\s| }\Bigr) \Bigl(\prod_{\nu \in T^\circ_{\bar \bp^{-1}(2)}} 2^{\ell_\nu|\s| }\Bigr)^{\frac{1}{2}},
\end{equation*}
which is the required bound \eqref{eq:KTildeBound-need}.
\medskip

Finally, we consider the case $\bp(M) = \infty$ in \eqref{eq:KTildeBound1_1}. Similarly to \eqref{eq:F1-bound} we can bound
\begin{equation}\label{eq:F3-bound}
	\| F \|_{L^p_\eps} \lesssim \sup_{\|z_{v_M} - z_{v^{\shortrightarrow}_M}\|_\s \lesssim 2^{- \ell_{v_M \wedge v^{\shortrightarrow}_M}}} |F(z_{v_M})|.
\end{equation} 
In the case $M = 1$, we use \eqref{eq:K-simple-bound2} to bound this expression by a constant multiple of 
\begin{equation*}
 \Bigl(\prod_{\nu \in T^\circ} 2^{-\ell_\nu \tilde \eta(\nu)}\Bigr) \Bigl(\prod_{\substack{\nu \in T^\circ_{\tilde \CCV_{\!\var}}}} 2^{\ell_\nu|\s|}\Bigr),
\end{equation*}
which is the required bound \eqref{eq:KTildeBound-need} with $\bp^{-1}(1) = \bp^{-1}(2) = \emptyset$, $\bp^{-1}(\infty) = \{v_1\}$ and $D = \tilde{\CCV}_{\!\var} \setminus \{v_1\}$. If $M \geq 2$, we use the induction hypothesis and similarly to \eqref{eq:F1-bound2} we bound the expression \eqref{eq:F3-bound} by a constant times
\begin{equation*}
\Bigl(\prod_{\nu \in T^\circ} 2^{-\ell_\nu \tilde \eta(\nu)}\Bigr) \Bigl(\prod_{\nu \in T^\circ_{D \sqcup \bar \bp^{-1}(\infty) \sqcup \{v_M\}}} 2^{\ell_\nu|\s| }\Bigr) \Bigl(\prod_{\nu \in T^\circ_{\bar \bp^{-1}(2)}} 2^{\ell_\nu|\s| }\Bigr)^{\frac{1}{2}},
\end{equation*}
which is exactly \eqref{eq:KTildeBound-need}.
\end{proof}

Since the product in \eqref{eq:KTildeBound-need} is different from the one in \cite[Lemma~A.10]{HQ18}, we need to have an analogous result in our context. For this we define the function $\hat \eta : T^\circ \to \R$ by $\hat \eta(v) = \tilde \eta(v)$ if $v \notin T^\circ_{\Gamma}$, and $\hat \eta(v) = \tilde \eta(v) - |\s|/2$ if $v \in T^\circ_{\Gamma}$, where we use the set $T^\circ_{\Gamma}$ introduced above Lemma~\ref{lem:KTildeBound}.

\begin{lemma}\label{lem:eta-satisfies}
In the setting of Theorem~\ref{thm:convolutions}, the function $\hat \eta$ satisfies assumptions of \cite[Lemma~A.10]{HQ18}, and 
\begin{equation}\label{eq:eta-hat-length}
|\hat \eta| := \sum_{v \in \T^\circ} \hat\eta(v) = |\s| |\tilde \CCV_{\!\bar\star} \setminus \{ v^\uparrow_{\!\star} \}| - \sum_{e \in \hat \CCE} \hat a_e - \frac{|\s|}{2} |\Gamma|.
\end{equation}
\end{lemma}

\begin{proof}
Using the definition of the function $\hat \eta$, the assumptions of \cite[Lemma~A.10]{HQ18} follow at once if we prove the following two properties 
\begin{enumerate}[topsep=2pt, itemsep=0ex]
\item For every $\nu \in T^\circ$ one has $\sum_{v \geq \nu} \tilde \eta(v) > \frac{|\s|}{2} \#\{v \in T^\circ_{\Gamma} : v \geq \nu\}$.
\item For every $\nu \in T^\circ$ such that $\nu \leq \nu_\star$ one has $\sum_{v \not\geq \nu} \tilde \eta(v) < \frac{|\s|}{2} \#\{v \in T^\circ_{\Gamma} : v \not\geq \nu\}$, provided that
this sum contains at least one term, where $\nu_\star$ is a fixed distinguished inner node.
\end{enumerate}
These bounds can be shown by repeating the proof of \cite[Lemma~A.19]{HQ18} and using Assumption~\ref{a:mainContraction}. To compute \eqref{eq:eta-hat-length} we use $|T^\circ_{\Gamma}| = |\Gamma|$.
\end{proof}

\begin{lemma}\label{lem:bound-for-eta}
Let the function $\tilde \eta$ be defined in \eqref{e:tilde_eta}. In the setting of Theorem~\ref{thm:convolutions} the following bound holds uniformly over $\lambda \in (0, 1]$:
\begin{equation}\label{eq:bound-for-eta}
\sum_{\ell \in \CN^{\epsEH}_\lambda(T^\circ)} \Bigl(\prod_{\nu \in T^\circ} 2^{-\ell_\nu \tilde \eta(\nu)}\Bigr) \Bigl(\prod_{\nu \in T^\circ_{\Gamma}} 2^{\ell_\nu|\s| / 2}\Bigr) \lesssim (\lambda \vee \epsEH)^{|\hat \eta|},
\end{equation}
where $|\hat \eta|$ is computed in \eqref{eq:eta-hat-length}.
\end{lemma}

\begin{proof}
Using the function $\hat \eta$, defined above Lemma~\ref{lem:eta-satisfies}, we can write the left-hand side of \eqref{eq:bound-for-eta} as
\begin{equation*}
\sum_{\ell \in \CN^{\epsEH}_\lambda(T^\circ)} \prod_{\nu \in T^\circ} 2^{-\ell_\nu \hat \eta_\nu}.
\end{equation*}
Then Lemma~\ref{lem:eta-satisfies} implies that the function $\hat \eta$ satisfies the assumptions of \cite[Lemma~A.10]{HQ18}, and \cite[Equation~A.29]{HQ18} allows to bound the left-hand side of \eqref{eq:bound-for-eta} by $\lambda^{|\hat \eta|}$, where $|\hat \eta|$ is computed in \eqref{eq:eta-hat-length}.
\end{proof}

\subsection{Proof of Theorem~\ref{thm:convolutions}}
\label{sec:proof_bound}

We use formulas \eqref{eq:integral-and-sigmas} and \eqref{e:wantedBound4}, and apply Theorem~\ref{thm:integral-bound-renorm} to each term:
\begin{align}
	&\E_p \sup_{t \in \R_+} \bigl| (\CI^{\eps, \L}_{\gamma} \CK^{\lambda, \eps}_\CCG)_t \bigr| \nonumber \\
	&\qquad \lesssim_{p} \sum_{\sigma \in \Sigma_\ga} \sum_{\substack{\bp \in \SP_{\!m} :\\ \bp^{-1}(1) \cap \Gamma(\ga) = \emptyset, \\ \L^{-1}(\triangledown) \subset \bp^{-1}(1)}} \sum_{(T, \ell) \in \CCT^{\epsEH}_{\!\lambda}(\tilde\CCV)} \sum_{\n \in \CN^{\epsEH}_\gamma(T, \ell)} \eps^{\alpha_\ga(\bp)} \Vert (\tilde \CK^{\eps, \n}_{\CCG_\gamma})^{\gamma, \sigma} \Vert_{L^{\bp}_\eps}^{\beta_{\ga, p}(\bp)} \label{eq:main-bound-proof} \\
	&\hspace{4cm} \times \Biggl(\prod_{\substack{i \in \bp^{-1}(\infty) \setminus \Gamma_{\!1}(\ga) : \\ i \geq 2}} \eps^{-\kappa_{\ga, i}(\bp)} \Vert (\tilde \CK^{\eps, \n}_{\CCG_\gamma})^{\gamma, \sigma} \Vert^{\beta_{\ga^{\smallgeq i}, p}(\bp^{\smallgeq i})}_{\SC_{s_{i}}^1(L^\infty_\eps)}\Biggr)^{\frac{1}{p}}, \nonumber
\end{align}
for some constants $\kappa_{\ga, i}(\bp) > 0$. Next, we are going to bound the terms in the sum in \eqref{eq:main-bound-proof} for different functions $\bp$.

If the function $\bp$ satisfies $\bp^{-1}(\infty) = \emptyset$, then we have $\beta_{\ga, p}(\bp) = 1$ and the product in the parentheses in \eqref{eq:main-bound-proof} equals $1$. Then the inner double sum in \eqref{eq:main-bound-proof} simplifies to 
\begin{equation*}
\eps^{\alpha_\ga(\bp)} \sum_{(T, \ell) \in \CCT^{\epsEH}_{\!\lambda}(\tilde\CCV)} \sum_{\n \in \CN^{\epsEH}_\gamma(T, \ell)} \Vert (\tilde \CK^{\eps, \n}_{\CCG_\gamma})^{\gamma, \sigma} \Vert_{L^{\bp}_\eps}.
\end{equation*}
Using Lemma~\ref{lem:KTildeBound}, we bound this expression by a constant multiple of 
\begin{equation*}
\eps^{\alpha_\ga(\bp)} \epsEH^{- \delta_{\gamma}(\bp)} \sum_{(T, \ell) \in \CCT^{\epsEH}_{\!\lambda}(\tilde\CCV)} \sum_{\n \in \CN^{\epsEH}_\gamma(T, \ell)}\Bigl(\prod_{\nu \in T^\circ} 2^{-\ell_\nu \tilde \eta(\nu)}\Bigr) \Bigl(\prod_{\nu \in T_{\Gamma}^\circ} 2^{\ell_{\nu}|\s| }\Bigr)^{\frac{1}{2}}.
\end{equation*}
Lemma~\ref{lem:bound-for-eta} allows to bound this expression by a constant times $\eps^{\alpha_\ga(\bp)} \epsEH^{- \delta_{\gamma}(\bp)} (\lambda \vee \epsEH)^{\nu_\gamma}$, with $\nu_\gamma$ defined in \eqref{eq:alpha-gamma}.

Let $\bp^{-1}(\infty) \neq \emptyset$. We can bound the norm $\Vert (\tilde \CK^{\eps, \n}_{\CCG_\gamma})^{\gamma, \sigma} \Vert_{\SC_{s_{i}}^1(L^\infty_\eps)}$ defined in \eqref{eq:norm-for-F}. More precisely, Lemma~\ref{lem:KTildeBound} yields 
\begin{equation}\label{eq:norm-C1}
\Vert \partial^k_{s_{i}} (\tilde \CK^{\eps, \n}_{\CCG_\gamma})^{\gamma, \sigma} \Vert_{L^\infty_\eps} \lesssim \epsEH^{- \delta_{\gamma}(\infty) - 2k} \Bigl(\prod_{\nu \in T^\circ} 2^{-\ell_\nu \tilde \eta(\nu)}\Bigr) \Bigl(\prod_{\nu \in T_{\Gamma}^\circ} 2^{\ell_{\nu}|\s| }\Bigr)^{\frac{1}{2}}, 
\end{equation}
for $k = 0$ and $k = 1$, where we write $\delta_{\gamma}(\infty)$ for the constant \eqref{eq:delta-gamma-p} defined via the function $\bp \equiv \infty$, and where the derivative $\partial_{s_{i}}$ gives the multiplier $\epsEH^{-2}$ (as follows from the scaling properties of the kernels). The norm \eqref{eq:norm-C1} can be brutally bounded by a negative power of $\epsEH \gtrsim \eps$, and hence the whole expression in the parentheses in \eqref{eq:main-bound-proof} can be bounded by $\eps^{-\kappa}$ for some $\kappa \geq 0$. Hence, the inner double sum in \eqref{eq:main-bound-proof} is estimated by 
\begin{equation*}
 \sum_{(T, \ell) \in \CCT^{\epsEH}_{\!\lambda}(\tilde\CCV)} \sum_{\n \in \CN^{\epsEH}_\gamma(T, \ell)} \eps^{\alpha_\ga(\bp) -\frac{\kappa}{p}} \Vert (\tilde \CK^{\eps, \n}_{\CCG_\gamma})^{\gamma, \sigma} \Vert_{L^{\bp}_\eps}^{\beta_{\ga, p}(\bp)}.
\end{equation*}
Since $\beta_{\ga, p}(\bp) \leq 1$ (see \eqref{eq:beta-power}), we use Jensen's inequality to estimate this expression by a constant times 
\begin{equation*}
 \eps^{\alpha_\ga(\bp) -\frac{\kappa}{p}} \biggl(\sum_{(T, \ell) \in \CCT^{\epsEH}_{\!\lambda}(\tilde\CCV)} \sum_{\n \in \CN^{\epsEH}_\gamma(T, \ell)} \Vert (\tilde \CK^{\eps, \n}_{\CCG_\gamma})^{\gamma, \sigma} \Vert_{L^{\bp}_\eps}\biggr)^{\beta_{\ga, p}(\bp)}.
\end{equation*}
We use Lemmas~\ref{lem:KTildeBound} and \ref{lem:bound-for-eta} to bound the double sum by $\epsEH^{- \delta_{\gamma}(\bp)} (\lambda \vee \epsEH)^{\nu_\gamma}$, and we bound the preceding expression by
\begin{equation*}
\eps^{\alpha_\ga(\bp) -\frac{\kappa}{p}} \Bigl(\epsEH^{- \delta_{\gamma}(\bp)} (\lambda \vee \epsEH)^{\nu_\gamma}\Bigr)^{\beta_{\ga, p}(\bp)}.
\end{equation*}
The expression in the parentheses is smaller than one (recall that we assumed $\nu_\gamma < 0$), and this expression can be estimated by $\eps^{\alpha_\ga(\bp) - \frac{\kappa}{p}} \epsEH^{- \delta_{\gamma}(\bp)} (\lambda \vee \epsEH)^{\nu_\gamma}$. Taking $p$ sufficiently large, we get the required bound \eqref{e:genconvBound}. 

\section{Application to a discrete martingale model}
\label{sec:applicationStochQuant}

This section is a showcase of the theory we developed in this paper. We introduce a family of martingales indexed by points of the lattice $\Le$; we then build trees as iterated integrals against the martingales themselves; and we finally apply our theory to prove uniform bounds.

The martingales in this section are chosen to resemble those that appear in our companion paper \cite{3dIsingKac}, where we prove convergence of the dynamical Ising-Kac model to $\Phi^4_3$ and this proof of convergence is what motivated the development of the theory here in the first place. We have therefore chosen to present a family of martingales which is both simpler and similar to the one found in the Ising-Kac model. In this way, we aim to give the reader a concrete and easy example of how the theory above can be applied.

For the proof of convergence in our companion paper \cite{3dIsingKac}, we use the theory of regularity structures \cite{Regularity} (see also \cite{Book, Notes}), together with the discretisation framework by \cite{EH19}. We prefer not to reintroduce all the concepts developed in these articles. Generally speaking, however, the theory of regularity structures is used as a solution theory for the (continuous) $\Phi^4_3$ equation, while the discretisation framework by \cite{EH19} gives us a solution theory for the discrete Ising-Kac model which preserves the formalism of regularity structures; \cite{EH19} also supplies us with some convergence tools, while our theory develops the missing tool for convergence of models, namely uniform boundedness in the scaling parameter. 

As follows from \cite{Regularity}, the regularity structure for the $\Phi^4_3$ equation has a basis which is convenient to write as formal expressions, which are written using the symbols $\Xi$, $\CI$ and $X_i$, $i = 0, \ldots, 3$. Here, the symbol $\Xi$ corresponds to the driving noise of the equation, $\CI$ corresponds to the convolution map with respect to the heat kernel, and $X_i$ are the time-space variables. For example, the expression $\Psi := \CI(\Xi)$ corresponds to the convolution of the heat kernel with the driving noise. The first several basis elements of the regularity structure are $\Xi$, $\Psi$, $\Psi^2$, $\Psi^3$, $\Psi^2 X_i$, $\CI(\Psi^3) \Psi$, $\CI(\Psi^3) \Psi^2$, $\CI(\Psi^3) \Psi^3$. In this section we will prove moment bounds for a discrete model acting only on the elements $\Xi$, $\Psi$, $\Psi^2$, $\CI(\Psi^3) \Psi^2$, which we believe are the most interesting. We refer the reader to our companion paper \cite{3dIsingKac} for a full description of the regularity structure for the Ising-Kac model.

\medskip
A \emph{model} is a pair of linear maps $(\Pi, \Gamma)$ on a regularity structure, which map the basis elements into functions/distributions. These maps are required to have certain algebraic and analytic properties which can be found in \cite{Regularity}. In this section we will consider a discrete model (in the sense of \cite{EH19}) and, more precisely, only a discretisation of the map $\Pi$. For this, we need to make some definitions. 

For any $\alpha \in (0, 1)$, we define $\epsEH := \eps^\al$ and the function $\psi_{\epsEH}: \R^3 \to \R_+$ by $\psi_{\epsEH}(x) := \epsEH^{-3} \psi( \epsEH^{-1} x )$, with $\psi$ being a smooth function, supported in the ball centered at the origin and of radius $1$, and satisfying $\int_{\R^3} \psi(x) \d x = 1$. While any $\al \in (0, 1)$ works for our purposes, we choose $\al = 3/4$ to be in the setting of \cite{3dIsingKac}. Then for $t \geq 0$ and $x \in \Lambda_\eps := (\eps \Z \slash \Z)^3$ we define the martingale 
\begin{equation}\label{eq:M-application}
	\mathcal{M}_{\eps}(t, x) = \frac{1}{\sqrt 2} \eps^{\frac{5}{2}} \sum_{y \in \Le} \psi_{\epsEH}(x-y) \Big( \CP_{\eps^{-2}t}(\eps^{-1} y) - \tilde \CP_{\eps^{-2}t}(\eps^{-1} y) \Big),
 \end{equation}
where $\CP_t(x)$ and $\tilde \CP_t(x)$ are independent Poisson processes of intensities $1$. We extend these martingales periodically to $x \in \eps \Z^3$ and we extend them to $\R$ in time as in \eqref{eq:martingale-extension}. We denote the new space-time domain by $D_\eps := \R \times \eps \Z^3$.

As mentioned above, we want to make the family of martingales \eqref{eq:M-application} as similar as possible to the family of martingales of the Ising-Kac interaction system; and, by the choice of $\al = \frac{3}{4}$, the two families of martingales have the same limiting behaviour as $\eps \to 0$. We refer the reader to the \cite{3dIsingKac} for a more detailed explanation. 

Using these martingales, we are going to define a discretisation of the map $\Pi$, which we denote by $\hat \Pi^{\eps}$. As we mentioned above, we will bound this map only for the four elements $\Xi$, $\Psi$, $\Psi^2$ and $\CI(\Psi^3) \Psi^2$ of the regularity structure, and a complete analysis of the map in a similar context is performed in \cite{3dIsingKac}. For every fixed $z \in D_\eps$ the action of this map on the element $\Xi$ is defined as 
\begin{equation}\label{eq:Pi-and-Xi}
\bigl(\hat \Pi^{\eps}_z \Xi\bigr)(\bar z) = \d \CM_{\eps}(\bar z),
\end{equation}
which means that for every test function $\phi : \R^4 \to \R$ we have 
\begin{equation*}
 \iota_\eps \bigl(\hat \Pi_{z}^{\eps}\Xi\bigr)(\phi) = \int_{D_\eps} \!\! \phi (\bar z)\, \d \CM_{\eps}(\bar z),
\end{equation*}
where we used the extension \eqref{eq:extension}, the expression on the left-hand side means the duality pairing of the distribution $\iota_\eps (\hat \Pi_{z}^{\eps}\Xi)$ with the test function $\phi$, and where the integral with respect to the martingale is defined as in \eqref{eq:integral-wrt-extension}.

Let $P(t,x) := \frac{1}{(4 \pi t)^{3/2}} e^{- |x|/(4t)}$ be the heat kernel on $\R^3$. In order to integrate it with respect to the martingales $\CM_{\eps}$, we need to remove the singularity of $P$ at the origin. For this, we will convolve $P$ with a smooth function. More precisely, let us take any smooth function $\tilde \psi : \R^4 \to \R$, supported in the unit ball with the center at the origin and which satisfies $\int_{\R^4} \tilde \psi(z)\, \d z = 1$. Let us set $\tilde \psi_{\epsEH}(t,x) := \epsEH^{-5} \tilde \psi(\epsEH^{-2} t, \epsEH^{-1} x)$. Then we define a smoothened heat kernel $P^\eps := P * \tilde \psi_\epsEH$, where the convolution is over $\R^4$. As follows from \cite[Lemma~7.7]{Regularity}, we can write $P^\eps = K^\eps + R^\eps$, where $K^\eps$ is a compactly supported singular part of the kernel (i.e. $K^\eps(0)$ diverges as $\eps \to 0$) and $R^\eps$ is smooth. Then we set 
\begin{equation}\label{eq:Pi-and-Psi}
\bigl(\hat \Pi_{z}^{\eps}\Psi\bigr)(\bar z) = \int_{D_\eps} \!\! K^\eps(\bar z - \tilde z) \, \d \CM_{\eps}(\tilde z),
\end{equation}
where we recall that $\Psi = \CI(\Xi)$. 

We will also use the kernel $K^\epsEH := K^\eps \star_\eps \psi_{\epsEH}$, where $\star_\eps$ is the convolution on $\eps \Z^3$. Then we set 
\begin{equation}\label{eq:Pi-and-Psi2}
\bigl(\hat \Pi_{z}^{\eps}\Psi^2\bigr)(\bar z) = \bigl(\hat \Pi_{z}^{\eps}\Psi\bigr)(\bar z)^2 - C^\eps_1,
\end{equation}
with the renormalisation constant 
\begin{equation}\label{e:Phi_C1}
 C^\eps_1 = \int_{D_\eps} \!\! K^\epsEH(z)^2 \, \d z.
\end{equation}
It is not difficult to see that $\hat \Pi_{z}^{\eps}\Psi$ converges to a distribution as $\eps \to 0$. This implies that the product $(\hat \Pi_{z}^{\eps}\Psi^2)(\bar z)^2$ diverges in the limit, and in order to have a non-trivial limit we need to renormalise the product by subtracting the divergent constant $ C^\eps_1$. The precise formula for this constant will be explained in Section~\ref{sec:symbol-3} below.

Finally, for the element $\CI(\Psi^3) \Psi^2$ we set 
\begin{equation}\label{eq:Pi-and-last}
\bigl(\hat \Pi_{z}^{\eps} \CI(\Psi^3) \Psi^2\bigr)(\bar z) = \bigl(\hat \Pi_{z}^{\eps}\Psi\bigr)(\bar z)^2 \int_{D_\eps} \!\! \bigl(K^\eps(\bar z - \tilde z) - K^\eps(z - \tilde z)\bigr) \bigl(\hat \Pi_{z}^{\eps}\Psi\bigr)(\tilde z)^3 \, \d \tilde z - 3 C^\eps_2\, (\hat \Pi_{z}^{\eps}\Psi)(\bar z),
\end{equation}
where the new renormalisation constant is
\begin{equation}\label{e:Phi_C2}
 C^\eps_2 = 2 \int_{D_\eps} \int_{D_\eps} \int_{D_\eps} \!\! K^\epsEH(z_1) K^\epsEH(z_1-z_3) K^\epsEH(z_2) K^\epsEH(z_2-z_3) K^\epsEH(z_3) \, \d z_1\, \d z_2\, \d z_3.
\end{equation}
Again, we need to subtract the renormalisation constant to have non-divergent moment bounds for the function. The formula for the renormalisation constant is explained in Section~\ref{sec:last_symbol} below.

For a fixed $\kappa > 0$, we assign to these four basis elements a \emph{homogeneity} $| \bigcdot |$ as
\[ | \Xi | = - \frac52 - \kappa, \qquad | \Psi | = - \frac12 - \kappa, \qquad | \Psi^2 | = -1 - 2 \kappa, \qquad | \CI(\Psi^3) \Psi^2 | = -\frac12 - 5 \kappa. \]
Let $\tau$ be one of these elements. We are going to prove that for some $\bar \kappa > 0$, any $p \geq 2$ and any test function $\phi : \R^4 \to \R$ the following bound holds:
\begin{equation}\label{e:model_bound}
	\E_p \bigl| \iota_\eps \bigl(\hat \Pi_{z}^{\eps}\tau\bigr)(\phi^\lambda_z)\bigr| \leq C (\lambda \vee \epsEH)^{|\tau| + \bar \kappa},
\end{equation}
uniformly in $z \in D_\eps$, $\lambda \in (0,1]$ and $\eps \in (0,1]$, where we use the extension \eqref{eq:extension} and a recentered and rescaled test function \eqref{eq:test-function}. The constant $C$ in this bound may depend on $p$. 
\medskip

For every element $\tau \in \{\Xi, \Psi, \Psi^2, \CI(\Psi^3) \Psi^2\}$, we use \eqref{eq:expansion} to write $(\hat \Pi_{z}^{\eps}\tau)(\phi_z^\lambda)$ as a sum of terms of the form
\begin{equation*}
	\int_{D_\eps} \phi_z^\lambda(\bar z) \biggl(\int_{\CD_\ga} F_{\bar z} (z_1, \ldots, z_n) \, \d \CM^{n}_{\eps} (z_1, \ldots, z_n)\biggr) \d \bar z,
\end{equation*}
where the measure $\CM^{n}_{\eps}$ is the product measure built from $\CM_\eps$ as in \eqref{eq:expansion2}, and the function $F$ and the contraction $\gamma$ (with $n$ components) will be specified case by case. In order to bound such terms, we are going to use Corollary~\ref{cor:convolutions}. For this, by analogy with the Ising-Kac model in \cite{IsingKac}, we use the definition \eqref{eq:M-application} and rewrite the previous expression in the form
\begin{equation*}
	\int_{D_\eps} \phi_z^\lambda(\bar z) \biggl(\int_{\CD_\ga} \Big( \big( F \ae^n \psi_\epsEH \big)_{\bar z} (z_1, \ldots, z_n) \Big) \, \d \bM^{n}_{\eps} (z_1, \ldots, z_n)\biggr) \d \bar z,
\end{equation*}
where now the measure $\bM^{n}_{\eps}$ is built as in \eqref{eq:expansion2} using the family of martingales 
\begin{equation}\label{eq:M-application-convolved}
	\M_{\eps}(t, x) := \frac{1}{\sqrt 2} \eps^{-\frac{1}{2}} \bigl(\CP_{\eps^{-2}t}(\eps^{-1} x) - \tilde \CP_{\eps^{-2}t}(\eps^{-1} x)\bigr)
\end{equation}
 and where $F \ae^n \psi_\epsEH$ is the discrete convolution of $F$ against the function $\psi_\epsEH$ in each of the variables of $F$. In particular, $\M_{\eps}$ are \cadlag martingales, satisfying Assumption~\ref{a:Martingales} with $\kone = -\frac12$ and $\C_\eps \equiv \Cd_\eps \equiv 1$. We note that we can replace the Poisson processes in \eqref{eq:M-application} and in \eqref{eq:M-application-convolved} by their compensated versions, because the integrals of their intensities cancel each other.

Additionally, it is convenient to use graphical notation to represent the function $F$ and the integrals. In the graphical notation, nodes represent variables and arrows represent kernels. The vertex ``\,\tikz[baseline=-3] \node [root] {};\,'' labelled with $z$ represents the basis point $z \in D_\eps$. The arrow ``\,\tikz[baseline=-0.1cm] \draw[testfcn] (1,0) to (0,0);\,'' represents a test function $\phi^\lambda_z$. The arrow ``\,\tikz[baseline=-0.1cm] \draw[keps] (0,0) to (1,0);\,'' represents either the discrete kernel $K^{\eps}$ or $K^{\epsEH}$, and we will write two labels $(a_e, r_e)$ on this arrow, which correspond to the labels on graphs as described in Section~\ref{sec:GeneralizedConvolutions}. More precisely, since the kernels $K^{\eps}$  and $K^{\epsEH}$ satisfy the bound \eqref{eq:K_bound} with $a_e=3$ (this follows from \cite[Lemma~7.7]{Regularity}), Lemma~\ref{lem:K_bound} implies that the kernels $K^{\eps}$  and $K^{\epsEH}$ have all the properties from Assumption~\ref{a:Kernels} with the values $a_e=3$ and $r_e=0$. Hence, we will depict this kernel by ``\,\tikz[baseline=-0.1cm] \draw[keps] (0,0) to node[labl,pos=0.45] {\tiny 3,0} (1,0);\,''. Whenever a contracted variable $z_i$ is integrated with respect to the measure $\bM^{n}_{\eps}$ with $n \geq 2$, we denote it by a node ``\,\tikz[baseline=-3] \node [var_very_blue] {};\,''. Moreover, the variable integrated with respect to $\M_{\eps}$ will be denoted by ``\,\tikz[baseline=-3] \node [var_blue] {};\,''. By the node ``\,\tikz[baseline=-3] \node [dot] {};\,'' we denote a variable integrated out in $D_\eps$. 

Using this notation, we will now prove the bounds \eqref{e:model_bound} for each of the four elements $\tau$.

\subsection{The element $\tau = \Xi$} 

The definition \eqref{eq:Pi-and-Xi} yields 
\begin{equation*}
 \iota_\eps (\hat \Pi_{z}^{\eps}\Xi)(\phi_z^\lambda) = \int_{D_\eps} \!\! \big( \phi_z^\lambda \ae \psi_\epsEH \big) (\bar z)\, \d \M_{\eps}(\bar z) .
\end{equation*}
Lemma~\ref{lem:Wiener} now gives the required bound \eqref{e:model_bound} with $|\Xi| = - \frac{5}{2} - \kappa$.

\subsection{The element $\tau = \Psi$}

Using \eqref{eq:Pi-and-Psi}, we can represent the map $\hat \Pi_{z}^{\eps}\tau$ diagrammatically as
\begin{equation*}
\iota_\eps (\hat \Pi_{z}^{\eps}\tau)(\phi_z^\lambda)
\;=\; 
\begin{tikzpicture}[scale=0.35, baseline=0cm]
	\node at (0.9,0.2)  [root] (root) {};
	\node at (0.9,0.2) [rootlab] {$z$};
	\node at (-2.3, 0.2)  [dot] (int) {};
	\node at (-5.5,0.2)  [var_blue] (left) {};	
	\draw[testfcn] (int) to (root);	
	\draw[keps] (left) to node[labl,pos=0.45] {\tiny 3,0} (int);
\end{tikzpicture}\;.
\end{equation*}
This diagram is the stochastic integral $\CI^{\eps} (F)$, where the kernel $F$ is in this case the generalised convolution $\CK^{\lambda, \eps}_{\CCG, z}$, as in \eqref{e:genconv_shift}, given by
\begin{equation*}
\CK^{\lambda, \eps}_{\CCG, z}(z^{\var}) = \int_{D_\eps} \!\!\phi_z^\lambda(\bar z)\, K^{\epsEH}(\bar z - z^{\var})\, \d \bar z.
\end{equation*}
One can check that Assumption~\ref{a:mainContraction} is satisfied for this diagram with a trivial contraction: from the diagram we see that $| \hat \CCV_{\!\var}| = 1$ and $|\hat \CCV_{\!\bar\star} \setminus \{ v^\uparrow_{\!\star} \}| = 1$. The space-time scaling is $\s = (2, 1, 1, 1)$, so that $|\s|=5$ and the value of the constant $\nu_\gamma$ in \eqref{e:genconvBound} is $-\frac{1}{2}$. Applying Corollary~\ref{cor:convolutions} with $\Gamma = \{ 1 \}$ and recalling that $|\tau| = -\frac{1}{2}-\kappa$, one obtains the bound 
\begin{equation*}
\Bigl( \E \bigl| \iota_\eps ( \hat \Pi^\eps_z \tau)(\phi^\lambda_z)\bigr|^p\Bigr)^{\frac{1}{p}} \lesssim (\lambda \vee \epsEH)^{- \frac{1}{2}} \big( 1 + \eps^{\frac52 - \theta} \epsEH^{-\frac52} \big),
\end{equation*}
for any $\theta > 0$. As such, we immediately get \eqref{e:model_bound} for the element $\tau$.

In what follows we always have $|\s|=5$ and we prefer not to specify contractions $\gamma$ every time, as it will be clear from diagrams. 

\subsection{The element $\tau = \Psi^2$}
\label{sec:symbol-3}

Taking into account the renormalisation in \eqref{eq:Pi-and-Psi2}, the map $\hat \Pi_{z}^{\eps}\tau$ can be represented by the diagrams
\begin{equation}\label{e:Pi2}
	\iota_\eps (\hat \Pi_{z}^{\eps}\tau)(\phi_z^\lambda)
	\;=\; 
	\begin{tikzpicture}[scale=0.35, baseline=0cm]
		\node at (0,-2.2)  [root] (root) {};
		\node at (0,-2.2) [rootlab] {$z$};
		\node at (0,-2.5) {$$};
		\node at (0,0)  [dot] (int) {};
		\node at (-1.5,2.5)  [var_blue] (left) {};
		\node at (1.5,2.5)  [var_blue] (right) {};
		
		\draw[testfcn] (int) to (root);
		
		\draw[keps] (left) to node[labl,pos=0.45] {\tiny 3,0} (int);
		\draw[keps] (right) to node[labl,pos=0.45] {\tiny 3,0} (int);
	\end{tikzpicture}
	\; +\;
	\begin{tikzpicture}[scale=0.35, baseline=0cm]
		\node at (0,-2.2)  [root] (root) {};
		\node at (0,-2.2) [rootlab] {$z$};
		\node at (0,0)  [dot] (int) {};
		\node at (0,2.5)  [var_very_blue] (left) {};
		
		\draw[testfcn] (int) to (root);
		
		\draw[keps] (left) to[bend left=60] node[labl,pos=0.45] {\tiny 3,0} (int);
		\draw[keps] (left) to[bend left=-60] node[labl,pos=0.45] {\tiny 3,0} (int);
	\end{tikzpicture}
	\; -\; C^\eps_1\, 
	\begin{tikzpicture}[scale=0.35, baseline=-0.5cm]
		\node at (0,-2.2)  [root] (root) {};
		\node at (0,-2.2) [rootlab] {$z$};
		\node at (0,0)  [dot] (int) {};
		
		\draw[testfcn] (int) to (root);
	\end{tikzpicture}\;.
\end{equation}
where the renormalisation constant is given in \eqref{e:Phi_C1}. Let us denote by ``\,\tikz[baseline=-3] \node [var_red_square] {};\,'' the variable which is integrated with respect to the martingales $t \mapsto ( [ \M_\eps(x) ]_t - \langle \M_\eps(x) \rangle_t )_{x \in \Le}$. Then, by \eqref{eq:RenormIntegrals}, our choice of the renormalisation constant allows to write \eqref{e:Pi2} as 
\begin{equation}\label{e:Pi2-new}
\iota_\eps (\hat \Pi_{z}^{\eps}\tau)(\phi_z^\lambda)
\;=\; 
\begin{tikzpicture}[scale=0.35, baseline=0cm]
	\node at (0,-2.2)  [root] (root) {};
	\node at (0,-2.2) [rootlab] {$z$};
	\node at (0,-2.5) {$$};
	\node at (0,0)  [dot] (int) {};
	\node at (-1.5,2.5)  [var_blue] (left) {};
	\node at (1.5,2.5)  [var_blue] (right) {};
	
	\draw[testfcn] (int) to (root);
	
	\draw[keps] (left) to node[labl,pos=0.45] {\tiny 3,0} (int);
	\draw[keps] (right) to node[labl,pos=0.45] {\tiny 3,0} (int);
\end{tikzpicture}
\; +\;
\begin{tikzpicture}[scale=0.35, baseline=0cm]
	\node at (0,-2.2)  [root] (root) {};
	\node at (0,-2.2) [rootlab] {$z$};
	\node at (0,0)  [dot] (int) {};
	\node at (0,2.5)  [var_red_square] (left) {};
	
	\draw[testfcn] (int) to (root);
	
	\draw[keps] (left) to[bend left=60] node[labl,pos=0.45] {\tiny 3,0} (int);
	\draw[keps] (left) to[bend left=-60] node[labl,pos=0.45] {\tiny 3,0} (int);
\end{tikzpicture},
\end{equation}
Let us denote these two diagrams by $\iota_\eps (\hat \Pi_{z}^{\eps, 1}\tau)(\phi_z^\lambda)$ and $\iota_\eps (\hat \Pi_{z}^{\eps, 2}\tau)(\phi_z^\lambda)$ respectively.

Let us start with the first diagram in \eqref{e:Pi2-new}. Assumption~\ref{a:mainContraction} is satisfied for it with a trivial contraction, and the bound \eqref{e:genconvBound} holds with the set $\Gamma = \{1, 2\}$. Furthermore, we have $| \hat \CCV_{\!\var}| = 2$ and $|\hat \CCV_{\!\bar\star} \setminus \{ v^\uparrow_{\!\star} \}| = 2$ and the value of the constant $\nu_\gamma$ in \eqref{e:genconvBound} is $-1$. Applying Corollary~\ref{cor:convolutions} to this diagram, we get the bound 
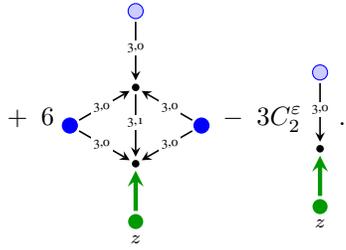
\begin{equation*}
 \Bigl(\E \bigl| \iota_\eps( \hat \Pi^{\eps, 1}_z \tau)(\phi^\lambda_z)\bigr|^p\Bigr)^{\frac{1}{p}} \lesssim (\lambda \vee \epsEH)^{-1} \bigl(1 + \eps^{\frac52 - \theta} \epsEH^{-\frac52} + \eps^{5 - \theta} \epsEH^{-5}\bigr).
\end{equation*}
Recalling that $|\tau| = -1-2\kappa$ and that $\eps$ is smaller than $\epsEH$, we get the bound \eqref{e:model_bound}.

The second diagram in \eqref{e:Pi2-new} does not satisfy Assumption~\ref{a:mainContraction}, because the kernels have very strong singularities. To solve this problem, we notice that multiplication of a kernel by a positive power of $\eps$ decreases the order of singularity in \eqref{eq:K_bound}. Hence, for any $0 < a < 3$ we can write the last diagram in \eqref{e:Pi2-new} as
\begin{equation}\label{eq:cherry-contracted} \eps^{2 (a - 3)}
\begin{tikzpicture}[scale=0.35, baseline=0cm]
	\node at (0,-2.2)  [root] (root) {};
	\node at (0,-2.2) [rootlab] {$z$};
	\node at (0,-2.5) {$$};
	\node at (0,0)  [dot] (int) {};
	\node at (0,2.5)  [var_red_square] (left) {};
	
	\draw[testfcn] (int) to (root);
	
	\draw[keps] (left) to[bend left=60] node[labl,pos=0.45] {\tiny a,0} (int);
	\draw[keps] (left) to[bend left=-60] node[labl,pos=0.45] {\tiny a,0} (int);
\end{tikzpicture},
\end{equation}
where we multiplied each kernel by $\eps^{3-a}$. For $a < \frac{5}{2}$ Assumption~\ref{a:mainContraction} is satisfied and Corollary~\ref{cor:convolutions} yields 
\begin{equation*}
	\Bigl(\E \bigl|\iota_\eps( \hat \Pi^{\eps, 2}_z \tau)(\phi^\lambda_z)\bigr|^p\Bigr)^{\frac{1}{p}} \lesssim \eps^{2a - 6} (\lambda \vee \epsEH)^{\frac52 - 2a} \big( \eps^{\frac52} + \eps^{5 - \theta} \epsEH^{\frac52} \big) \lesssim \eps^{2a - \frac72} .
\end{equation*}
If we choose $a > \frac74$, then this term disappears in the limit. Notice that here the only difference with the case $\tau = \Psi$ is that the coefficient $\alpha_\ga(\bp)$ coming from Corollary~\ref{cor:convolutions} is larger.

\subsection{The element $\tau = \CI(\Psi^3) \Psi^2$} 
\label{sec:last_symbol}

Using the definition \eqref{eq:Pi-and-last}, the diagrams for the map $\hat \Pi_{z}^{\eps}\tau$ are the following:
\begin{align} \nonumber
\iota_\eps &(\hat \Pi_{z}^{\eps}\tau)(\phi_z^\lambda)
\;=\; 
\begin{tikzpicture}[scale=0.35, baseline=-0.5cm]
	\node at (0,-5.1)  [root] (root) {};
	\node at (0,-5.1) [rootlab] {$z$};
	\node at (0,0) [dot] (int) {};
	\node at (-2,2) [var_blue] (left) {};
	\node at (0,2.9)  [var_blue] (cent) {};
	\node at (2,2)  [var_blue] (right) {};
	\node at (0,-2.9) [dot] (cent1) {};
	\node at (2,-0.9)  [var_blue] (right1) {};
	\node at (-2,-0.9)  [var_blue] (left1) {};
	\draw[testfcn] (cent1) to (root);
	\draw[keps] (left) to node[labl,pos=0.45] {\tiny 3,0} (int);
	\draw[keps] (right) to node[labl,pos=0.45] {\tiny 3,0} (int);
	\draw[keps] (cent) to node[labl,pos=0.45] {\tiny 3,0} (int);
	\draw[keps] (int) to node[labl,pos=0.45] {\tiny 3,1} (cent1);
	\draw[keps] (right1) to node[labl,pos=0.45] {\tiny 3,0} (cent1);
	\draw[keps] (left1) to node[labl,pos=0.45] {\tiny 3,0} (cent1);
\end{tikzpicture}
\; +\;3\,
\begin{tikzpicture}[scale=0.35, baseline=-0.5cm]
	\node at (0,-5.1)  [root] (root) {};
	\node at (0,-5.1) [rootlab] {$z$};
	\node at (0,0)  [dot] (int) {};
	\node at (-1.3,2.3)  [var_red_square] (left) {};
	\node at (2.1,2.3)  [var_blue] (right) {};
	\node at (0,-2.9) [dot] (cent1) {};
	\node at (2,-0.9)  [var_blue] (right1) {};
	\node at (-2,-0.9)  [var_blue] (left1) {};
	\draw[testfcn] (cent1) to (root);
	\draw[keps] (left) to[bend left=60] node[labl,pos=0.45] {\tiny 3,0} (int);
	\draw[keps] (left) to[bend left=-60] node[labl,pos=0.45] {\tiny 3,0} (int);
	\draw[keps] (right) to[bend left=30] node[labl,pos=0.45] {\tiny 3,0} (int);
	\draw[keps] (int) to node[labl,pos=0.45] {\tiny 3,1} (cent1);
	\draw[keps] (right1) to node[labl,pos=0.45] {\tiny 3,0} (cent1);
	\draw[keps] (left1) to node[labl,pos=0.45] {\tiny 3,0} (cent1);
\end{tikzpicture}
\; +\;
\begin{tikzpicture}[scale=0.35, baseline=-0.5cm]
	\node at (0,-5.1)  [root] (root) {};
	\node at (0,-5.1) [rootlab] {$z$};
	\node at (0,0)  [dot] (int) {};
	\node at (0,2.9)  [var_very_blue] (left) {};
	\node at (0,-2.9) [dot] (cent1) {};
	\node at (2,-0.9)  [var_blue] (right1) {};
	\node at (-2,-0.9)  [var_blue] (left1) {};
	\draw[testfcn] (cent1) to (root);
	\draw[keps] (left) to[bend left=90] node[labl,pos=0.4] {\tiny 3,0} (int);
	\draw[keps] (left) to[bend left=-90] node[labl,pos=0.4] {\tiny 3,0} (int);
	\draw[keps] (left) to node[labl,pos=0.6] {\tiny 3,0} (int);
	\draw[keps] (int) to node[labl,pos=0.45] {\tiny 3,1} (cent1);
	\draw[keps] (right1) to node[labl,pos=0.45] {\tiny 3,0} (cent1);
	\draw[keps] (left1) to node[labl,pos=0.45] {\tiny 3,0} (cent1);
\end{tikzpicture}
\; +\;
2\,
\begin{tikzpicture}[scale=0.35, baseline=-0.5cm]
	\node at (0,-5.1)  [root] (root) {};
	\node at (0,-5.1) [rootlab] {$z$};
	\node at (0,0)  [dot] (int) {};
	\node at (0,2.9)  [var_very_blue] (left) {};
	\node at (0,-2.9) [dot] (cent1) {};
	\node at (2,-0.9)  [var_blue] (right1) {};
	\draw[testfcn] (cent1) to (root);
	\draw[keps] (left) to[bend left=90] node[labl,pos=0.4] {\tiny 3,0} (int);
	\draw[keps] (left) to[bend left=-90] node[labl,pos=0.4] {\tiny 3,0} (int);
	\draw[keps] (left) to node[labl,pos=0.6] {\tiny 3,0} (int);
	\draw[keps] (int) to node[labl,pos=0.45] {\tiny 3,1} (cent1);
	\draw[keps] (left) to[bend right=100] node[labl,pos=0.55] {\tiny 3,0} (cent1);
	\draw[keps] (right1) to node[labl,pos=0.45] {\tiny 3,0} (cent1);
\end{tikzpicture}
\; +\;3\,
\begin{tikzpicture}[scale=0.35, baseline=-0.5cm]
	\node at (0,-5.1)  [root] (root) {};
	\node at (0,-5.1) [rootlab] {$z$};
	\node at (0,-4.5) {$$};
	\node at (0,0) [dot] (int) {};
	\node at (-2,2) [var_blue] (left) {};
	\node at (0,2.9)  [var_blue] (cent) {};
	\node at (0,-2.9) [dot] (cent1) {};
	\node at (3,-1.45)  [var_very_blue] (right1) {};
	\draw[testfcn] (cent1) to (root);
	\draw[keps] (left) to node[labl,pos=0.45] {\tiny 3,0} (int);
	\draw[keps] (right1) to[bend left=-30] node[labl,pos=0.45] {\tiny 3,0} (int);
	\draw[keps] (cent) to node[labl,pos=0.45] {\tiny 3,0} (int);
	\draw[keps] (int) to node[labl,pos=0.45] {\tiny 3,1} (cent1);
	\draw[keps] (right1) to[bend left=-30] node[labl,pos=0.45] {\tiny 3,0} (cent1);
	\draw[keps] (right1) to[bend left=30] node[labl,pos=0.45] {\tiny 3,0} (cent1);
\end{tikzpicture}\\[0.5cm] \nonumber
&
\; +\;6\,
\begin{tikzpicture}[scale=0.35, baseline=-0.5cm]
	\node at (0,-5.1)  [root] (root) {};
	\node at (0,-5.1) [rootlab] {$z$};
	\node at (0,0)  [dot] (int) {};
	\node at (-1.3,2.3)  [var_very_blue] (left) {};
	\node at (0,-2.9) [dot] (cent1) {};
	\node at (2.5,-1.45)  [var_very_blue] (right1) {};
	\draw[testfcn] (cent1) to (root);
	\draw[keps] (left) to[bend left=60] node[labl,pos=0.45] {\tiny 3,0} (int);
	\draw[keps] (left) to[bend left=-60] node[labl,pos=0.45] {\tiny 3,0} (int);
	\draw[keps] (right1) to node[labl,pos=0.45] {\tiny 3,0} (int);
	\draw[keps] (int) to node[labl,pos=0.45] {\tiny 3,1} (cent1);
	\draw[keps] (right1) to node[labl,pos=0.45] {\tiny 3,0} (cent1);
	\draw[keps] (left) to[bend left=-80] node[labl,pos=0.55] {\tiny 3,0} (cent1);
\end{tikzpicture}
\; +\;3\,
\begin{tikzpicture}[scale=0.35, baseline=-0.5cm]
	\node at (0,-5.1)  [root] (root) {};
	\node at (0,-5.1) [rootlab] {$z$};
	\node at (0,0)  [dot] (int) {};
	\node at (-1.3,2.3)  [var_red_square] (left) {};
	\node at (0,-2.9) [dot] (cent1) {};
	\node at (3,-1.45) [var_very_blue] (right1) {};
	\draw[testfcn] (cent1) to (root);
	\draw[keps] (left) to[bend left=60] node[labl,pos=0.45] {\tiny 3,0} (int);
	\draw[keps] (left) to[bend left=-60] node[labl,pos=0.45] {\tiny 3,0} (int);
	\draw[keps] (right1) to[bend left=-20] node[labl,pos=0.45] {\tiny 3,0} (int);
	\draw[keps] (int) to node[labl,pos=0.45] {\tiny 3,1} (cent1);
	\draw[keps] (right1) to[bend left=-30] node[labl,pos=0.45] {\tiny 3,0} (cent1);
	\draw[keps] (right1) to[bend left=30] node[labl,pos=0.45] {\tiny 3,0} (cent1);
\end{tikzpicture}
\;+\;6\,
\begin{tikzpicture}[scale=0.35, baseline=-0.5cm]
	\node at (0,-5.1)  [root] (root) {};
	\node at (0,-5.1) [rootlab] {$z$};
	\node at (0,0)  [dot] (int) {};
	\node at (2.5,-1.45)  [var_very_blue] (right1) {};
	\node at (-2,-0.9)  [var_blue] (left1) {};
	\node at (-2,2) [var_blue] (left) {};
	\node at (0,2.9)  [var_blue] (cent) {};
	\draw[testfcn] (cent1) to (root);
	\draw[keps] (left) to node[labl,pos=0.45] {\tiny 3,0} (int);
	\draw[keps] (cent) to node[labl,pos=0.45] {\tiny 3,0} (int);
	\draw[keps] (right1) to node[labl,pos=0.45] {\tiny 3,0} (int);
	\draw[keps] (int) to node[labl,pos=0.45] {\tiny 3,1} (cent1);
	\draw[keps] (right1) to node[labl,pos=0.45] {\tiny 3,0} (cent1);
	\draw[keps] (left1) to node[labl,pos=0.45] {\tiny 3,0} (cent1);
\end{tikzpicture}
\; +\;6\,
\begin{tikzpicture}[scale=0.35, baseline=-0.5cm]
	\node at (0,-5.1)  [root] (root) {};
	\node at (0,-5.1) [rootlab] {$z$};
	\node at (0,0)  [dot] (int) {};
	\node at (-1.3,2.3)  [var_very_blue] (left) {};
	\node at (2.1,2.3)  [var_blue] (right) {};
	\node at (0,-2.9) [dot] (cent1) {};
	\node at (2,-0.9)  [var_blue] (right1) {};
	\draw[testfcn] (cent1) to (root);s
	\draw[keps] (left) to[bend left=60] node[labl,pos=0.45] {\tiny 3,0} (int);
	\draw[keps] (left) to[bend left=-60] node[labl,pos=0.45] {\tiny 3,0} (int);
	\draw[keps] (right) to[bend left=30] node[labl,pos=0.45] {\tiny 3,0} (int);
	\draw[keps] (int) to node[labl,pos=0.45] {\tiny 3,1} (cent1);
	\draw[keps] (left) to[bend left=-80] node[labl,pos=0.45] {\tiny 3,0} (cent1);
	\draw[keps] (right1) to node[labl,pos=0.45] {\tiny 3,0} (cent1);
\end{tikzpicture}
\; +\;
\begin{tikzpicture}[scale=0.35, baseline=-0.5cm]
	\node at (0,-5.1)  [root] (root) {};
	\node at (0,-5.1) [rootlab] {$z$};
	\node at (0,0)  [dot] (int) {};
	\node at (0,2.9)  [var_very_blue] (left) {};
	\node at (0,-2.9) [dot] (cent1) {};
	\draw[testfcn] (cent1) to (root);
	\draw[keps] (left) to[bend left=90] node[labl,pos=0.4] {\tiny 3,0} (int);
	\draw[keps] (left) to[bend left=-90] node[labl,pos=0.4] {\tiny 3,0} (int);
	\draw[keps] (left) to node[labl,pos=0.6] {\tiny 3,0} (int);
	\draw[keps] (int) to node[labl,pos=0.45] {\tiny 3,1} (cent1);
	\draw[keps] (left) to[bend left=100] node[labl,pos=0.55] {\tiny 3,0} (cent1);
	\draw[keps] (left) to[bend right=100] node[labl,pos=0.55] {\tiny 3,0} (cent1);
\end{tikzpicture}\\[0.5cm] \label{eq:last-tree}
&\; +\;3\,
\begin{tikzpicture}[scale=0.35, baseline=-0.5cm]
	\node at (0,-5.1)  [root] (root) {};
	\node at (0,-5.1) [rootlab] {$z$};
	\node at (0,0)  [dot] (int) {};
	\node at (-1.3,2.3)  [var_red_square] (left) {};
	\node at (2.1,2.3)  [var_blue] (right) {};
	\node at (0,-2.9) [dot] (cent1) {};
	\node at (2.9,-1.7) [var_red_square] (right1) {};
	\draw[testfcn] (cent1) to (root);
	\draw[keps] (left) to[bend left=60] node[labl,pos=0.45] {\tiny 3,0} (int);
	\draw[keps] (left) to[bend left=-60] node[labl,pos=0.45] {\tiny 3,0} (int);
	\draw[keps] (int) to node[labl,pos=0.45] {\tiny 3,1} (cent1);
	\draw[keps] (right1) to[bend left=-30] node[labl,pos=0.45] {\tiny 3,0} (cent1);
	\draw[keps] (right1) to[bend left=30] node[labl,pos=0.45] {\tiny 3,0} (cent1);
	\draw[keps] (right) to[bend left=30] node[labl,pos=0.45] {\tiny 3,0} (int);
\end{tikzpicture}
\; +\;6\,
\begin{tikzpicture}[scale=0.35, baseline=-0.5cm]
	\node at (0,-5.1)  [root] (root) {};
	\node at (0,-5.1) [rootlab] {$z$};
	\node at (0,0)  [dot] (int) {};
	\node at (-1.3,2.3)  [var_red_square] (left) {};
	\node at (0,-2.9) [dot] (cent1) {};
	\node at (2.5,-1.45)  [var_very_blue] (right1) {};
	\node at (-2,-0.9)  [var_blue] (left1) {};
	\draw[testfcn] (cent1) to (root);
	\draw[keps] (left) to[bend left=60] node[labl,pos=0.45] {\tiny 3,0} (int);
	\draw[keps] (left) to[bend left=-60] node[labl,pos=0.45] {\tiny 3,0} (int);
	\draw[keps] (right1) to node[labl,pos=0.45] {\tiny 3,0} (int);
	\draw[keps] (int) to node[labl,pos=0.45] {\tiny 3,1} (cent1);
	\draw[keps] (right1) to node[labl,pos=0.45] {\tiny 3,0} (cent1);
	\draw[keps] (left1) to node[labl,pos=0.45] {\tiny 3,0} (cent1);
\end{tikzpicture}
\;+\;
\begin{tikzpicture}[scale=0.35, baseline=-0.5cm]
	\node at (0,-5.1)  [root] (root) {};
	\node at (0,-5.1) [rootlab] {$z$};
	\node at (0,0) [dot] (int) {};
	\node at (-2,2) [var_blue] (left) {};
	\node at (0,2.9) [var_blue] (cent) {};
	\node at (2,2) [var_blue] (right) {};
	\node at (0,-2.9) [dot] (cent1) {};
	\node at (2.9,-1.7) [var_red_square] (right1) {};
	\draw[testfcn] (cent1) to (root);
	\draw[keps] (left) to node[labl,pos=0.45] {\tiny 3,0} (int);
	\draw[keps] (right) to node[labl,pos=0.45] {\tiny 3,0} (int);
	\draw[keps] (cent) to node[labl,pos=0.45] {\tiny 3,0} (int);
	\draw[keps] (int) to node[labl,pos=0.45] {\tiny 3,1} (cent1);
	\draw[keps] (right1) to[bend left=-30] node[labl,pos=0.45] {\tiny 3,0} (cent1);
	\draw[keps] (right1) to[bend left=30] node[labl,pos=0.45] {\tiny 3,0} (cent1);
\end{tikzpicture}
\; +\;
\begin{tikzpicture}[scale=0.35, baseline=-0.5cm]
	\node at (0,-5.1)  [root] (root) {};
	\node at (0,-5.1) [rootlab] {$z$};
	\node at (0,0)  [dot] (int) {};
	\node at (0,2.9)  [var_very_blue] (left) {};
	\node at (0,-2.9) [dot] (cent1) {};
	\node at (2.9,-1.7) [var_red_square] (right1) {};
	\draw[testfcn] (cent1) to (root);
	\draw[keps] (left) to[bend left=90] node[labl,pos=0.4] {\tiny 3,0} (int);
	\draw[keps] (left) to[bend left=-90] node[labl,pos=0.4] {\tiny 3,0} (int);
	\draw[keps] (left) to node[labl,pos=0.6] {\tiny 3,0} (int);
	\draw[keps] (int) to node[labl,pos=0.45] {\tiny 3,1} (cent1);
	\draw[keps] (right1) to[bend left=-30] node[labl,pos=0.45] {\tiny 3,0} (cent1);
	\draw[keps] (right1) to[bend left=30] node[labl,pos=0.45] {\tiny 3,0} (cent1);
\end{tikzpicture} + \\[0.5cm]
& \; +\; 6\;
\begin{tikzpicture}[scale=0.35, baseline=-0.5cm]
	\node at (0,-5.1)  [root] (root) {};
	\node at (0,-5.1) [rootlab] {$z$};
	\node at (0,0) [dot] (int) {};
	\node at (0,2.9)  [var_blue] (cent) {};
	\node at (0,-2.9) [dot] (cent1) {};
	\node at (2.5,-1.45)  [var_very_blue] (right1) {};
	\node at (-2.5,-1.45)  [var_very_blue] (left1) {};
	\draw[testfcn] (cent1) to (root);
	\draw[keps] (left1) to node[labl,pos=0.45] {\tiny 3,0} (int);
	\draw[keps] (right1) to node[labl,pos=0.45] {\tiny 3,0} (int);
	\draw[keps] (cent) to node[labl,pos=0.45] {\tiny 3,0} (int);
	\draw[keps] (int) to node[labl,pos=0.45] {\tiny 3,1} (cent1);
	\draw[keps] (right1) to node[labl,pos=0.45] {\tiny 3,0} (cent1);
	\draw[keps] (left1) to node[labl,pos=0.45] {\tiny 3,0} (cent1);
\end{tikzpicture}
\; -\;3 C^\eps_2\;
\begin{tikzpicture}[scale=0.35, baseline=-0.7cm]
	\node at (0,-5.1)  [root] (root) {};
	\node at (0,-5.1) [rootlab] {$z$};
	\node at (0,0)  [var_blue] (cent) {};
	\node at (0,-2.9) [dot] (cent1) {};
	\draw[testfcn] (cent1) to (root);
	\draw[keps] (cent) to node[labl,pos=0.45] {\tiny 3,0} (cent1);
\end{tikzpicture}\;. \nonumber
\end{align}
All these diagrams, except the tenth and the last two, can be bounded by a direct application of Corollary~\ref{cor:convolutions}.

To bound the tenth diagram (the one that contracts all leaves), we write
\begin{equation*}
	\bM^{5}_\eps ( A_\ga) = \eps^{15} \sum_{x \in \Le} \sum_{0 \leq s \leq T} \big(\Delta_s \M_{\eps}(x) \big)^{5} = \eps^{13} \sum_{x \in \Le} \sum_{0 \leq s \leq T} \Delta_s \M_{\eps}(x),
\end{equation*}
where $\gamma \in \fC_1(\CCV_{\!\var})$ and $|\CCV_{\!\var}| = |\gamma_1| = 5$. Then we can write 
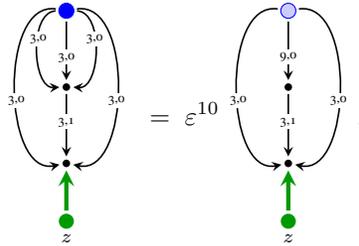
\begin{equation*}
\begin{tikzpicture}[scale=0.35, baseline=-0.5cm]
	\node at (0,-5.1)  [root] (root) {};
	\node at (0,-5.1) [rootlab] {$z$};
	\node at (0,0)  [dot] (int) {};
	\node at (0,2.9)  [var_very_blue] (left) {};
	\node at (0,-2.9) [dot] (cent1) {};
	\draw[testfcn] (cent1) to (root);
	\draw[keps] (left) to[bend left=90] node[labl,pos=0.4] {\tiny 3,0} (int);
	\draw[keps] (left) to[bend left=-90] node[labl,pos=0.4] {\tiny 3,0} (int);
	\draw[keps] (left) to node[labl,pos=0.6] {\tiny 3,0} (int);
	\draw[keps] (int) to node[labl,pos=0.45] {\tiny 3,1} (cent1);
	\draw[keps] (left) to[bend left=100] node[labl,pos=0.55] {\tiny 3,0} (cent1);
	\draw[keps] (left) to[bend right=100] node[labl,pos=0.55] {\tiny 3,0} (cent1);
\end{tikzpicture}
\;=\; \eps^{10}
\begin{tikzpicture}[scale=0.35, baseline=-0.5cm]
	\node at (0,-5.1)  [root] (root) {};
	\node at (0,-5.1) [rootlab] {$z$};
	\node at (0,0)  [dot] (int) {};
	\node at (0,2.9)  [var_blue] (left) {};
	\node at (0,-2.9) [dot] (cent1) {};
	\draw[testfcn] (cent1) to (root);
	\draw[keps] (left) to node[labl,pos=0.6] {\tiny 9,0} (int);
	\draw[keps] (int) to node[labl,pos=0.45] {\tiny 3,1} (cent1);
	\draw[keps] (left) to[bend left=100] node[labl,pos=0.55] {\tiny 3,0} (cent1);
	\draw[keps] (left) to[bend right=100] node[labl,pos=0.55] {\tiny 3,0} (cent1);
\end{tikzpicture},
\end{equation*}
where the stochastic integral is with respect to the martingale $\M_{\eps}$. Taking powers of $\eps$ to improve the singularities of the kernels, Corollary~\ref{cor:convolutions} allows to bound moments of this diagram by $(\lambda \vee \epsEH)^{|\tau|}$, and the remaining positive $\eps$ power makes it vanish in the limit. 
\medskip

We now turn our attention to the last two diagrams in \eqref{eq:last-tree}, which are also the most interesting ones. We can write
\begin{align} \label{eq:important_subtree}
	2\;
	\begin{tikzpicture}[scale=0.35, baseline=-0.5cm]
		\node at (0,-5.1)  [root] (root) {};
		\node at (0,-5.1) [rootlab] {$z$};
		\node at (0,0) [dot] (int) {};
		\node at (0,2.9)  [var_blue] (cent) {};
		\node at (0,-2.9) [dot] (cent1) {};
		\node at (2.5,-1.45)  [var_very_blue] (right1) {};
		\node at (-2.5,-1.45)  [var_very_blue] (left1) {};
		\draw[testfcn] (cent1) to (root);
		\draw[keps] (left1) to node[labl,pos=0.45] {\tiny 3,0} (int);
		\draw[keps] (right1) to node[labl,pos=0.45] {\tiny 3,0} (int);
		\draw[keps] (cent) to node[labl,pos=0.45] {\tiny 3,0} (int);
		\draw[keps] (int) to node[labl,pos=0.45] {\tiny 3,1} (cent1);
		\draw[keps] (right1) to node[labl,pos=0.45] {\tiny 3,0} (cent1);
		\draw[keps] (left1) to node[labl,pos=0.45] {\tiny 3,0} (cent1);
	\end{tikzpicture}
	\; -\; C^\eps_2 \;
	\begin{tikzpicture}[scale=0.35, baseline=-0.7cm]
		\node at (0,-5.1)  [root] (root) {};
		\node at (0,-5.1) [rootlab] {$z$};
		\node at (0,0)  [var_blue] (cent) {};
		\node at (0,-2.9) [dot] (cent1) {};
		\draw[testfcn] (cent1) to (root);
		\draw[keps] (cent) to node[labl,pos=0.45] {\tiny 3,0} (cent1);
	\end{tikzpicture}
	\; = \; \left(
	2\;
	\begin{tikzpicture}[scale=0.35, baseline=-0.5cm]
		\node at (0,-5.1)  [root] (root) {};
		\node at (0,-5.1) [rootlab] {$z$};
		\node at (0,0) [dot] (int) {};
		\node at (0,2.9)  [var_blue] (cent) {};
		\node at (0,-2.9) [dot] (cent1) {};
		\node at (2.5,-1.45)  [var_very_blue] (right1) {};
		\node at (-2.5,-1.45)  [var_very_blue] (left1) {};
		\draw[testfcn] (cent1) to (root);
		\draw[keps] (left1) to node[labl,pos=0.45] {\tiny 3,0} (int);
		\draw[keps] (right1) to node[labl,pos=0.45] {\tiny 3,0} (int);
		\draw[keps] (cent) to node[labl,pos=0.45] {\tiny 3,0} (int);
		\draw[keps] (int) to node[labl,pos=0.45] {\tiny 3,0} (cent1);
		\draw[keps] (right1) to node[labl,pos=0.45] {\tiny 3,0} (cent1);
		\draw[keps] (left1) to node[labl,pos=0.45] {\tiny 3,0} (cent1);
	\end{tikzpicture}
	\; -\; C^\eps_2 \;
	\begin{tikzpicture}[scale=0.35, baseline=-0.7cm]
		\node at (0,-5.1)  [root] (root) {};
		\node at (0,-5.1) [rootlab] {$z$};
		\node at (0,0)  [var_blue] (cent) {};
		\node at (0,-2.9) [dot] (cent1) {};
		\draw[testfcn] (cent1) to (root);
		\draw[keps] (cent) to node[labl,pos=0.45] {\tiny 3,0} (cent1);
	\end{tikzpicture} \right)
	\; -\; 2\;
	\begin{tikzpicture}[scale=0.35, baseline=-0.5cm]
		\node at (0,-5.1)  [root] (root) {};
		\node at (0,-5.1) [rootlab] {$z$};
		\node at (0,0) [dot] (int) {};
		\node at (0,2.9)  [var_blue] (cent) {};
		\node at (0,-2.9) [dot] (cent1) {};
		\node at (2.5,-1.45)  [var_very_blue] (right1) {};
		\node at (-2.5,-1.45)  [var_very_blue] (left1) {};
		\draw[testfcn] (cent1) to (root);
		\draw[keps] (left1) to node[labl,pos=0.45] {\tiny 3,0} (int);
		\draw[keps] (right1) to node[labl,pos=0.45] {\tiny 3,0} (int);
		\draw[keps] (cent) to node[labl,pos=0.45] {\tiny 3,0} (int);
		\draw[keps] (right1) to node[labl,pos=0.45] {\tiny 3,0} (cent1);
		\draw[keps] (left1) to node[labl,pos=0.45] {\tiny 3,0} (cent1);
		\draw [keps] (int) to[out=0,in=90] (3.5,-2.5) to [out=-90, in=0] node[labl,pos=0] {\tiny 3,0} (root);
	\end{tikzpicture}\;.
\end{align}
The last diagram can be first decomposed as 
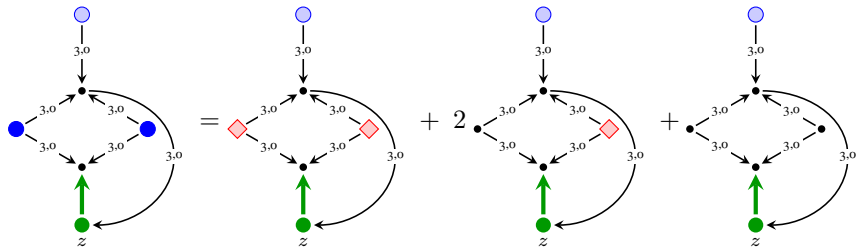
\[
	\begin{tikzpicture}[scale=0.35, baseline=-0.5cm]
		\node at (0,-5.1)  [root] (root) {};
		\node at (0,-5.1) [rootlab] {$z$};
		\node at (0,0) [dot] (int) {};
		\node at (0,2.9)  [var_blue] (cent) {};
		\node at (0,-2.9) [dot] (cent1) {};
		\node at (2.5,-1.45)  [var_very_blue] (right1) {};
		\node at (-2.5,-1.45)  [var_very_blue] (left1) {};
		\draw[testfcn] (cent1) to (root);
		\draw[keps] (left1) to node[labl,pos=0.45] {\tiny 3,0} (int);
		\draw[keps] (right1) to node[labl,pos=0.45] {\tiny 3,0} (int);
		\draw[keps] (cent) to node[labl,pos=0.45] {\tiny 3,0} (int);
		\draw[keps] (right1) to node[labl,pos=0.45] {\tiny 3,0} (cent1);
		\draw[keps] (left1) to node[labl,pos=0.45] {\tiny 3,0} (cent1);
		\draw [keps] (int) to[out=0,in=90] (3.5,-2.5) to [out=-90, in=0] node[labl,pos=0] {\tiny 3,0} (root);
	\end{tikzpicture}\;
	=
	\begin{tikzpicture}[scale=0.35, baseline=-0.5cm]
		\node at (0,-5.1)  [root] (root) {};
		\node at (0,-5.1) [rootlab] {$z$};
		\node at (0,0) [dot] (int) {};
		\node at (0,2.9)  [var_blue] (cent) {};
		\node at (0,-2.9) [dot] (cent1) {};
		\node at (2.5,-1.45)  [var_red_square] (right1) {};
		\node at (-2.5,-1.45)  [var_red_square] (left1) {};
		\draw[testfcn] (cent1) to (root);
		\draw[keps] (left1) to node[labl,pos=0.45] {\tiny 3,0} (int);
		\draw[keps] (right1) to node[labl,pos=0.45] {\tiny 3,0} (int);
		\draw[keps] (cent) to node[labl,pos=0.45] {\tiny 3,0} (int);
		\draw[keps] (right1) to node[labl,pos=0.45] {\tiny 3,0} (cent1);
		\draw[keps] (left1) to node[labl,pos=0.45] {\tiny 3,0} (cent1);
		\draw [keps] (int) to[out=0,in=90] (3.5,-2.5) to [out=-90, in=0] node[labl,pos=0] {\tiny 3,0} (root);
	\end{tikzpicture}\;
	+ \; 2 \;
	\begin{tikzpicture}[scale=0.35, baseline=-0.5cm]
		\node at (0,-5.1)  [root] (root) {};
		\node at (0,-5.1) [rootlab] {$z$};
		\node at (0,0) [dot] (int) {};
		\node at (0,2.9)  [var_blue] (cent) {};
		\node at (0,-2.9) [dot] (cent1) {};
		\node at (2.5,-1.45)  [var_red_square] (right1) {};
		\node at (-2.5,-1.45)  [dot] (left1) {};
		\draw[testfcn] (cent1) to (root);
		\draw[keps] (left1) to node[labl,pos=0.45] {\tiny 3,0} (int);
		\draw[keps] (right1) to node[labl,pos=0.45] {\tiny 3,0} (int);
		\draw[keps] (cent) to node[labl,pos=0.45] {\tiny 3,0} (int);
		\draw[keps] (right1) to node[labl,pos=0.45] {\tiny 3,0} (cent1);
		\draw[keps] (left1) to node[labl,pos=0.45] {\tiny 3,0} (cent1);
		\draw [keps] (int) to[out=0,in=90] (3.5,-2.5) to [out=-90, in=0] node[labl,pos=0] {\tiny 3,0} (root);
	\end{tikzpicture}\;
	+
	\begin{tikzpicture}[scale=0.35, baseline=-0.5cm]
		\node at (0,-5.1)  [root] (root) {};
		\node at (0,-5.1) [rootlab] {$z$};
		\node at (0,0) [dot] (int) {};
		\node at (0,2.9)  [var_blue] (cent) {};
		\node at (0,-2.9) [dot] (cent1) {};
		\node at (2.5,-1.45)  [dot] (right1) {};
		\node at (-2.5,-1.45)  [dot] (left1) {};
		\draw[testfcn] (cent1) to (root);
		\draw[keps] (left1) to node[labl,pos=0.45] {\tiny 3,0} (int);
		\draw[keps] (right1) to node[labl,pos=0.45] {\tiny 3,0} (int);
		\draw[keps] (cent) to node[labl,pos=0.45] {\tiny 3,0} (int);
		\draw[keps] (right1) to node[labl,pos=0.45] {\tiny 3,0} (cent1);
		\draw[keps] (left1) to node[labl,pos=0.45] {\tiny 3,0} (cent1);
		\draw [keps] (int) to[out=0,in=90] (3.5,-2.5) to [out=-90, in=0] node[labl,pos=0] {\tiny 3,0} (root);
	\end{tikzpicture}\;
\]
and applying Corollary~\ref{cor:convolutions} with $\nu_\ga = -\frac{11}{2}$, each term can be bounded by $\big( \lambda \vee \epsEH \big)^{-\frac{11}{2}} \eps^5 \leq \big( \lambda \vee \epsEH \big)^{-\frac{1}{2}}$. Then the expression in the brackets in \eqref{eq:important_subtree} can instead be written as
\begin{equation} \label{eq:last_tree_decomposition}
	2\;
	\begin{tikzpicture}[scale=0.35, baseline=-0.5cm]
		\node at (0,-5.1)  [root] (root) {};
		\node at (0,-5.1) [rootlab] {$z$};
		\node at (0,0) [dot] (int) {};
		\node at (0,2.9)  [var_blue] (cent) {};
		\node at (0,-2.9) [dot] (cent1) {};
		\node at (2.5,-1.45)  [var_very_blue] (right1) {};
		\node at (-2.5,-1.45)  [var_very_blue] (left1) {};
		\draw[testfcn] (cent1) to (root);
		\draw[keps] (left1) to node[labl,pos=0.45] {\tiny 3,0} (int);
		\draw[keps] (right1) to node[labl,pos=0.45] {\tiny 3,0} (int);
		\draw[keps] (cent) to node[labl,pos=0.45] {\tiny 3,0} (int);
		\draw[keps] (int) to node[labl,pos=0.45] {\tiny 3,0} (cent1);
		\draw[keps] (right1) to node[labl,pos=0.45] {\tiny 3,0} (cent1);
		\draw[keps] (left1) to node[labl,pos=0.45] {\tiny 3,0} (cent1);
	\end{tikzpicture}
	\; -\; C^\eps_2\;
	\begin{tikzpicture}[scale=0.35, baseline=-0.7cm]
		\node at (0,-5.1)  [root] (root) {};
		\node at (0,-5.1) [rootlab] {$z$};
		\node at (0,0)  [var_blue] (cent) {};
		\node at (0,-2.9) [dot] (cent1) {};
		\draw[testfcn] (cent1) to (root);
		\draw[keps] (cent) to node[labl,pos=0.45] {\tiny 3,0} (cent1);
	\end{tikzpicture} 
	\;=\;
	2 \;
	\begin{tikzpicture}[scale=0.35, baseline=-0.5cm]
		\node at (0,-5.1)  [root] (root) {};
		\node at (0,-5.1) [rootlab] {$z$};
		\node at (0,0) [dot] (int) {};
		\node at (0,2.9)  [var_blue] (cent) {};
		\node at (0,-2.9) [dot] (cent1) {};
		\node at (2.5,-1.45)  [var_red_square] (right1) {};
		\node at (-2.5,-1.45)  [var_red_square] (left1) {};
		\draw[testfcn] (cent1) to (root);
		\draw[keps] (left1) to node[labl,pos=0.45] {\tiny 3,0} (int);
		\draw[keps] (right1) to node[labl,pos=0.45] {\tiny 3,0} (int);
		\draw[keps] (cent) to node[labl,pos=0.45] {\tiny 3,0} (int);
		\draw[keps] (int) to node[labl,pos=0.45] {\tiny 3,0} (cent1);
		\draw[keps] (right1) to node[labl,pos=0.45] {\tiny 3,0} (cent1);
		\draw[keps] (left1) to node[labl,pos=0.45] {\tiny 3,0} (cent1);
	\end{tikzpicture}
	+ 4 \;
	\begin{tikzpicture}[scale=0.35, baseline=-0.5cm]
		\node at (0,-5.1)  [root] (root) {};
		\node at (0,-5.1) [rootlab] {$z$};
		\node at (0,0) [dot] (int) {};
		\node at (0,2.9)  [var_blue] (cent) {};
		\node at (0,-2.9) [dot] (cent1) {};
		\node at (2.5,-1.45)  [var_red_square] (right1) {};
		\node at (-2.5,-1.45)  [dot] (left1) {};
		\draw[testfcn] (cent1) to (root);
		\draw[keps] (left1) to node[labl,pos=0.45] {\tiny 3,0} (int);
		\draw[keps] (right1) to node[labl,pos=0.45] {\tiny 3,0} (int);
		\draw[keps] (cent) to node[labl,pos=0.45] {\tiny 3,0} (int);
		\draw[keps] (int) to node[labl,pos=0.45] {\tiny 3,0} (cent1);
		\draw[keps] (right1) to node[labl,pos=0.45] {\tiny 3,0} (cent1);
		\draw[keps] (left1) to node[labl,pos=0.45] {\tiny 3,0} (cent1);
	\end{tikzpicture}
	\;+\; \left( \; 2 \;
	\begin{tikzpicture}[scale=0.35, baseline=-0.5cm]
		\node at (0,-5.1)  [root] (root) {};
		\node at (0,-5.1) [rootlab] {$z$};
		\node at (0,-5.3) {$$};
		\node at (0,0) [dot] (int) {};
		\node at (0,2.9)  [var_blue] (cent) {};
		\node at (0,-2.9) [dot] (cent1) {};
		\node at (2.5,-1.45)  [dot] (right1) {};
		\node at (-2.5,-1.45)  [dot] (left1) {};
		\draw[testfcn] (cent1) to (root);
		\draw[keps] (left1) to node[labl,pos=0.45] {\tiny 3,0} (int);
		\draw[keps] (right1) to node[labl,pos=0.45] {\tiny 3,0} (int);
		\draw[keps] (cent) to node[labl,pos=0.45] {\tiny 3,0} (int);
		\draw[keps] (int) to node[labl,pos=0.45] {\tiny 3,0} (cent1);
		\draw[keps] (right1) to node[labl,pos=0.45] {\tiny 3,0} (cent1);
		\draw[keps] (left1) to node[labl,pos=0.45] {\tiny 3,0} (cent1);
	\end{tikzpicture}
	\; -\; C^\eps_2 \;
	\begin{tikzpicture}[scale=0.35, baseline=-0.7cm]
		\node at (0,-5.1)  [root] (root) {};
		\node at (0,-5.1) [rootlab] {$z$};
		\node at (0,0)  [var_blue] (cent) {};
		\node at (0,-2.9) [dot] (cent1) {};
		\draw[testfcn] (cent1) to (root);
		\draw[keps] (cent) to node[labl,pos=0.45] {\tiny 3,0} (cent1);
	\end{tikzpicture}
	\right),
\end{equation}
The first two diagrams above can again be bounded using Corollary~\ref{cor:convolutions} by a constant multiple of $(\lambda \vee \epsEH)^{- 1/ 2}$.

The last expression in the brackets in \eqref{eq:last_tree_decomposition} needs more attention. Let us define a new kernel 
\[
	\CG_\eps (z_1, z_2) \; := \;
	\begin{tikzpicture}[scale=0.35, baseline=-0.55cm]
		\node at (0,0) [root] (int) {};
		\node at (0, 0.7) [] () {$z_2$};
		\node at (0,-2.9) [root] (cent) {};
		\node at (0,-3.6) [] () {$z_1$};
		\node at (2.5,-1.45) [dot] (right1) {};
		\node at (-2.5,-1.45) [dot] (left1) {};
		\draw[keps] (left1) to node[labl,pos=0.45] {\tiny 3,0} (int);
		\draw[keps] (right1) to node[labl,pos=0.45] {\tiny 3,0} (int);
		\draw[keps] (int) to node[labl,pos=0.45] {\tiny 3,0} (cent);
		\draw[keps] (right1) to node[labl,pos=0.45] {\tiny 3,0} (cent);
		\draw[keps] (left1) to node[labl,pos=0.45] {\tiny 3,0} (cent);
	\end{tikzpicture}\;,
\]
which is in fact a function of the difference of the arguments, i.e. $\CG_\eps (z_1, z_2) = \CG_\eps (z_2 - z_1)$. As follows from the order of the singularity of the kernel $K^{\epsEH}$ and \cite[Lemma~7.3]{HM18}, the function $\CG_\eps$ satisfies $|D^k \CG_\eps(z)| \lesssim (\|z\|_\s + \epsEH)^{-5 - |k|_\s}$ for all multiindices $k$ with $|k|_\s$ large enough. Then we conclude from Lemma~\ref{lem:K_bound} that the function $\CG_\eps$ has all the properties listed in Assumption~\ref{a:Kernels} with the values $a_e = 5$ and $r_e = 0$. We denote the kernel $\CG_\eps$ by an edge ``\,\tikz[baseline=-0.1cm] \draw[kernelBig] (0,0) to node[labl,pos=0.45] {\tiny 5,0} (1,0);\,''. Then the first diagram in the brackets in \eqref{eq:last_tree_decomposition} can be represented as
\begin{equation*}
	\begin{tikzpicture}[scale=0.35, baseline=0cm]
		\node at (0,0)  [root] (root) {};
		\node at (0,0) [rootlab] {$z$};
		\node at (2.2,0) [dot] (cent1) {};
		\node at (5.1,0) [dot] (int) {};
		\node at (7.7,0)  [var_blue] (cent) {};
		\draw[testfcn] (cent1) to (root);
		\draw[keps] (cent) to node[labl,pos=0.45] {\tiny 3,0} (int);
		\draw[kernelBig] (int) to node[labl,pos=0.45] {\tiny 5,0} (cent1);
	\end{tikzpicture}\;,
\end{equation*}
and one can see that this diagram does not satisfy Assumption~\ref{a:mainContraction}(\ref{it:mainContraction-1}) (recall that $|\s| = 5$). To resolve this problem, we need to use a negative renormalisation (in the sense of Section~\ref{sec:renormalisation}) of the kernel $\CG_\eps$. More precisely, for any smooth function $\eta: \R^4 \times \R^4 \to \R$, we define its negative renormalisation as
\[
	\big( \SR_\eps \CG_\eps \big) (\eta) := \int_{D_\eps} \int_{D_\eps} \CG_\eps(z_1, z_2) \big( \eta(z_1, z_2) - \eta(z_1, z_1) \big) \, \d z_1 \d z_2,
\]
and we graphically depict $\SR_\eps \CG_\eps$ as ``\,\tikz[baseline=-0.1cm] \draw[kernelBig] (0,0) to node[labl,pos=0.45] {\tiny 5,-1} (1,0);\,'', where the label ``$-1$'' refers to the order of renormalisation. Since the renormalisation constant \eqref{e:Phi_C2} can be represented as
\begin{equation*}
 C^\eps_2 = 2 \;
 \begin{tikzpicture}[scale=0.35, baseline=-0.55cm]
		\node at (0,0) [dot] (int) {};
		\node at (0,-2.9) [root] (cent) {};
		\node at (0,-2.9) [rootlab] {$0$};
		\node at (2.5,-1.45) [dot] (right1) {};
		\node at (-2.5,-1.45) [dot] (left1) {};
		\draw[keps] (left1) to node[labl,pos=0.45] {\tiny 3,0} (int);
		\draw[keps] (right1) to node[labl,pos=0.45] {\tiny 3,0} (int);
		\draw[keps] (int) to node[labl,pos=0.45] {\tiny 3,0} (cent);
		\draw[keps] (right1) to node[labl,pos=0.45] {\tiny 3,0} (cent);
		\draw[keps] (left1) to node[labl,pos=0.45] {\tiny 3,0} (cent);
	\end{tikzpicture}\;,
\end{equation*}
the expression in the brackets in \eqref{eq:last_tree_decomposition} equals
\begin{equation*}
	2\; \begin{tikzpicture}[scale=0.35, baseline=-0.1cm]
		\node at (0,0)  [root] (root) {};
		\node at (0,0) [rootlab] {$z$};
		\node at (2.2,0) [dot] (cent1) {};
		\node at (5.1,0) [dot] (int) {};
		\node at (7.7,0)  [var_blue] (cent) {};
		\draw[testfcn] (cent1) to (root);
		\draw[keps] (cent) to node[labl,pos=0.45] {\tiny 3,0} (int);
		\draw[kernelBig] (int) to node[labl,pos=0.45] {\tiny 5,-1} (cent1);
	\end{tikzpicture}\;.
\end{equation*}
This diagram satisfies Assumption~\ref{a:mainContraction}, and using Corollary~\ref{cor:convolutions} we bound its moments by a constant multiple of $(\lambda \vee \epsEH)^{-\bar\kappa}$ for any $\bar\kappa > 0$.

\bibliographystyle{Martin}
\bibliography{bibliography}

\end{document}